%% file: lyzz-layered-green-function-rev.tex
\documentclass[hidelinks,onefignum,onetabnum]{siamart190516}

\usepackage[ntheorem]{empheq}

\input{ex_shared}
\usepackage{latexsym,amssymb,mathrsfs,amssymb}
\usepackage{cases}
\usepackage{empheq}

\newcommand{\Rb}{{\mathbb R}}
\newcommand{\RR}{{\mathcal{R}}}
\newcommand{\Tc}{{\mathcal T}}

\newcommand{\Cb}{{\mathbb C}}

\newcommand{\Sp}{{\mathbb S}}

\newcommand{\ds}{\displaystyle}

\newcommand{\be}{\begin{eqnarray}}
\newcommand{\ben}{\begin{eqnarray*}}
\newcommand{\en}{\end{eqnarray}}
\newcommand{\enn}{\end{eqnarray*}}
\newcommand{\ba}{\backslash}
\newcommand{\pa}{\partial}

\newcommand{\ov}{\overline}

\newcommand{\Grad}{{\rm Grad\,}}

\newcommand{\Ima}{{\rm Im\,}}

\newcommand{\ka}{\kappa}

\newcommand{\I}{{\rm Im}}
\newcommand{\Rt}{{\rm Re}}

\makeatletter

\newcommand{\Rmnum}[1]{\expandafter\@slowromancap\romannumeral #1@}
\makeatother

\DeclareMathOperator{\sech}{sech}
\DeclareMathOperator{\sgn}{sgn}

\usepackage{color}
\definecolor{hw}{rgb}{0,0,0}
\definecolor{hw-a}{rgb}{0,0,0}

\ifpdf
\hypersetup{
  pdftitle={Uniform far-field asymptotics of the two-layered Green function in 2D and application
to wave scattering in a two-layered medium},
  pdfauthor={Long Li, JianSheng Yang, Bo Zhang and Haiwen Zhang}
}
\fi

\begin{document}

\maketitle

\begin{abstract}
In this paper, we establish new results for the uniform far-field asymptotics of the two-layered Green function (together with its derivatives) in 2D in the frequency domain. To the best of our knowledge, our results are the sharpest yet obtained. The steepest descent method plays an important role in the proofs of our results.
Further, as an application of our new results, we derive the uniform far-field asymptotics of the scattered
field to the acoustic scattering problem by buried obstacles in a two-layered medium with a locally rough interface.
The results obtained in this paper provide a theoretical foundation for our recent work, where direct
imaging methods have been developed to image the locally rough interface
from phaseless total-field data or phased far-field data at a fixed frequency.
It is believed that the results obtained in this paper will also be useful on its own right.
\end{abstract}



\section{Introduction}\label{sec:1}
This paper is concerned with the uniform far-field asymptotics of two-dimensional two-layered Green function
(together with its derivatives) in the frequency domain and the application in the two-layered medium
scattering problems. Two-layered Green functions have attracted much attention in many fields such as radar,
remote sensing, ocean acoustics, exploration geophysics and outdoor sound propagation.

Two-layered Green functions play an important role in both
direct and inverse scattering problems in two-layered media.
Typically, the scattered waves for the scattering problems in two-layered media
can be formulated by boundary integral equations associated with
the two-layered Green functions (see, e.g., \cite{BHY,Lpj,Car17}).
Therefore, efficient and accurate calculations of the two-layered Green functions
have been extensively studied (see, e.g., \cite{B84,bg_2} for the method using high-frequency asymptotics, \cite{CY,Cui,CP} for the method using contour deformations and \cite{Lai} for the method using a new hybrid
integral representation).
Furthermore, with the benefit of high accuracy algorithms and special natures of the two-layered Green functions,
there are many works concerning the efficient numerical methods for the direct scattering problems in two-layered
media (see, e.g., \cite{BHY,Lpj} for the approach combining the boundary integral equations and
variational methods, \cite{CZ} for the perfectly matched layer method and {\color{hw-a}\cite{Cho,WZC19}} for the fast multipole method).
Moreover, based on the properties of the two-layered Green functions,
many fast and robust inversion algorithms have been developed for recovering the scatterers buried in
two-layered media from a knowledge of the scattered-field data or the far-field data, such as the MUSIC
method \cite{AIL05,WP}, the linear sampling method \cite{Coyle_2000},
the sampling type method \cite{Li_2015,LYZZ} and the migration algorithm \cite{liuet}.

Let $\mathbb R^2_{\pm}:=\{{(x_1,x_2)}\in\mathbb R^2:x_2\gtrless 0\}$, $\Gamma_0:=\{{(x_1,0)}:x_1\in\mathbb R\}$ and $\mathbb{S}^1_\pm:=\{x\in \Rb^2_\pm :\vert x \vert=1\}$.
For any point $x\in\Rb^2$ with $\vert x \vert\neq0$, let
$\hat{x}:=x/\vert x \vert$ denote the direction of $x$.
Then given any source point $y\in \Rb^2_+\cup\Rb^2_-$, the two-dimensional two-layered Green function $G(x,y)$
in the frequency domain is the solution to the following scattering problem
\begin{align}
\ds\Delta_x G(x,y) +{k}^2_\pm G(x,y)&=-\delta(x,y)&& \text{in}\quad\Rb^2_\pm, \label{eq:0.3} \\
\ds[G(x,y)]=0,\;\;\left[{\partial G(x,y)}/{\partial\nu(x)}\right]&=0&&\text{on}\quad\Gamma_0,\label{eq60}\\
\ds\lim_{\vert x \vert\rightarrow +\infty}\sqrt{\vert x \vert}\left(\frac{\partial G(x,y)}{\partial \vert x \vert}-ik_\pm G(x,y)\right)&=0& & \textrm{uniformly for all}~ \hat{x} \in \mathbb{S}^1_\pm, \label{eq:0.4}
\end{align}
where $\delta(x,y)$ denotes the Dirac delta distribution, $k_\pm = \omega/c_{\pm} > 0$ are two different
wave numbers in $\mathbb R^2_{\pm}$, respectively, $\nu$ denotes the unit normal on $\Gamma_0$ pointing
into $\Rb^2_+$ and $[\cdot]$ denotes the jump across the interface $\Gamma_0$. Here, (\ref{eq:0.4}) is called
the Sommerfeld radiation condition, $\omega$ is the wave frequency and $c_\pm$ are the wave speeds in the
half-spaces $\Rb^2_\pm$, respectively. For the derivation of the explicit form of $G(x,y)$,
we refer to \cite{Cheney_1995, AIL05, Lpj, Car17}.

It is well-known from \cite{DK13} that the solution to the scattering problem (\ref{eq:0.3})--(\ref{eq:0.4})
in the limiting case $k:=k_+=k_-$ (i.e. the fundamental solution of the homogeneous Helmholtz equation
$\Delta w + k^2 w = 0$ in $\Rb^2$ with wave number $k>0$)
is given by $\Phi(x,y):=\frac{i}{4}H^{(1)}_0(k|x-y|)$ with $H^{(1)}_0$ denoting the Hankel function of the
first kind of order $0$ and $\Phi(x,y)$ satisfies the far-field asymptotic behavior: for any $y\in\Rb^2$,
\be\label{NNe6}
\Phi(x,y)=\frac{i}{4}H^{(1)}_0(k|x-y|)=\frac{e^{ik|x|}}{\sqrt{|x|}}\frac{e^{i\frac{\pi}4}}{\sqrt{8\pi k}}
\left(e^{-ik \hat x\cdot y}+O\left(|x|^{-1}\right)\right),\\ \nonumber
\qquad\qquad\qquad\qquad\qquad\qquad |x|\rightarrow+\infty, \quad x\in\Rb^2,
\en
uniformly for all angles $\theta_{\hat{x}}\in\left[\left.0,2\pi\right)\right.$.
Here, $\theta_{\hat{x}}\in\left[\left.0,2\pi\right)\right.$ denotes the angle of the observation
direction $\hat{x}\in\mathbb{S}^1:=\{x\in\Rb^2:|x|=1\}$.
However, to the best of our knowledge, there are only a few works concerning the far-field asymptotic
properties of the two-layered Green function $G(x,y)$,
which is much more challenging since the explicit form of $G(x,y)$ is more complicated due to
the presence of the two-layered medium.
In \cite{AIL05}, it was proved that for any $y\in\Rb^2_-$, $G(x,y)$ has the far-field asymptotic behavior:
\be\label{eq44}
G(x,y) = \frac{e^{ik_{+}|x|}}{\sqrt{|x|}}G^{\infty}(\hat x,y) + G_{Res}(x,y),
\quad |x|\rightarrow+\infty,\quad x\in\Rb^2_+,
\en
where $G^{\infty}(\hat x, y)$ is the so-called far-field pattern of $G(x,y)$ in $x\in\Rb^2_+$
(see \cite[formula (9)]{AIL05}) and the remainder term $G_{Res}(x,y)=o(|x|^{-1/2})$
as $|x|\rightarrow+\infty$
for any angle $\theta_{\hat{x}}\in(0,\pi)$; see \cite[formula (2.59)]{Car17} or (\ref{eq:t7}) below for
the complete definition of $G^{\infty}(\hat x,y)$ with $\hat{x}\in\Sp^1_+$ and $y\in\Rb^2_+\cup\Rb^2_-$.
Moreover, it was deduced in \cite{Car17} (see also Appendix B of the supplementary materials in \cite{BL16})
that for any $y\in\Rb^2_+\cup\Rb^2_-$, $G(x,y)$ has the far-field asymptotic behavior (\ref{eq44})
with $G_{Res}(x,y)$ satisfying that
\be\label{eq22}
G_{Res}(x,y)=O(\vert x \vert^{-3/2})\quad \textrm{as}~ |x|\rightarrow+\infty
\en
{\color{hw}for all angles $\theta_{\hat{x}}\in(0,\pi)$ in the case $k_+<k_-$ and for all angles
$\theta_{\hat{x}}\in(0,\pi)\ba\{\theta_c,\pi-\theta_c\}$ in the case $k_+>k_-$ with $\theta_c:=\arccos(k_-/k_+)\in(0,\pi/2)$ (see \eqref{eq43} below).
Note that $\theta_c$ and $\pi-\theta_c$ are called the critical angles of $G(x,y)$ in the case $k_+>k_-$;}
see Remarks \ref{re3} and \ref{re4} for discussions on the critical angles.
Further, it can be deduced from \cite[formula (2.27) and pages 31--32]{Car17} that
for any $y\in\Rb^2_+\cup\Rb^2_-$,
$G(x,y)=O(|x|^{-3/2})$ as $|x|\rightarrow+\infty$ for the angle $\theta_{\hat{x}}\in\{0,\pi\}$.
{\color{hw}
Note that \cite{Car17} only considered the pointwise (rather than uniform) estimates of the far-field
asymptotics of $G(x,y)$ with respect to the angle $\theta_{\hat{x}}$
and these pointwise estimates can be dealt with
in the same way for both cases $k_+<k_-$ and $k_+>k_-$, as seen in \cite{Car17}.
Moreover, it should be pointed out that for the case $k_+>k_-$, the far-field asymptotic estimates of
$G(x,y)$ obtained in \cite{Car17} actually do not hold uniformly for the angles $\theta_{\hat{x}}$
in the vicinity of the critical angles $\theta_c$ and $\pi-\theta_c$;
see Section \ref{sec:2} below for further discussions.}

In this paper, we establish new results for the uniform far-field asymptotic properties of $G(x,y)$
for all angles $\theta_{\hat{x}}\in(0,\pi)\cup(\pi,2\pi)$
{\color{hw-a}
and for both cases $k_+<k_-$ and $k_+>k_-$}.
The main contributions of our results are twofold.
First, our results show that for any $y\in\Rb^2_+\cup\Rb^2_-$,
$G(x,y)$ has the far-field asymptotic properties (\ref{eq44}) and
\be\label{eq50}
G(x,y)=\frac{e^{ik_{-}\vert x \vert}}{\sqrt{\vert x \vert}}G^{\infty}(\hat x, y) + G_{Res}(x,y),
\quad \vert x \vert\rightarrow+\infty,\quad x\in\Rb^2_-,
\en
where $G^{\infty}(\hat x,y)$ in (\ref{eq50}) is the far-field pattern of $G(x,y)$ in
$x\in\Rb^2_-$ (see (\ref{eq:3.1}) below) and the remainder term $G_{Res}(x,y)$ in (\ref{eq44})
and (\ref{eq50}) satisfies that
\be\label{eq53}
G_{Res}(x,y)=O(|x|^{-3/4}),\quad |x|\rightarrow+\infty,\quad x\in\Rb^2_+\cup\Rb^2_-,
\en
uniformly for all angles $\theta_{\hat{x}}\in(0,\pi)\cup(\pi,2\pi)$
(including the critical angles) and all $y\in B^+_{R_0}\cup B^-_{R_0}$ with arbitrarily
fixed $R_0>0$ (see Remark \ref{re5}). Here, $B^\pm_{R_0}:=\{y\in\Rb^2_\pm:|y|< R_0\}$ for $R_0>0$.
We also prove that the uniform asymptotic property (\ref{eq53}) is essentially sharp for the
angles $\theta_{\hat{x}}$ lying in the vicinity of any one of the critical angles and that
the asymptotic property $G_{Res}(x,y)=O(|x|^{-3/2})$ does not hold for the critical angles
(see Remarks \ref{re3} and \ref{re4}).
Secondly, for all angles $\theta_{\hat{x}}\in(0,\pi)\cup(\pi,2\pi)$ except for the critical angles,
we obtain uniform upper bounds for the remainder term $G_{Res}(x,y)$ in (\ref{eq44}) and (\ref{eq50})
with large enough $|x|$. The uniform upper bounds obtained show that for any critical angle $\theta$,
$G_{Res}(x,y)$ in (\ref{eq44}) and (\ref{eq50}) satisfies that
\ben
G_{Res}(x,y)=O({{\left|\theta-\theta_{\hat x}\right|}^{-\frac 32}}|x|^{-\frac 32}),\quad
|x|\rightarrow+\infty,
\enn
uniformly for all angles $\theta_{\hat{x}}$ in a punctured neighborhood of $\theta$ and all
$y\in B^+_{R_0}\cup B^-_{R_0}$ with arbitrarily fixed $R_0>0$
(see Theorems \ref{NLe:2} and \ref{le:3.2}).
Moreover, similarly to our analysis for $G(x,y)$, we also derive uniform far-field asymptotic
properties for $\nabla_y G(x,y)$.
To the best of our knowledge, our results on the uniform far-field asymptotics of $G(x,y)$
and its derivatives are the sharpest yet obtained.
Note that the steepest descent method plays an important role in the proofs of our results
(see, e.g., \cite{B84,BH86,bg_2,CH95,D92,O70} for applications of the steepest descent method).
The crucial part of our proofs is the singularity analysis of the relevant integrals for $G(x,y)$
in the case when the angle $\theta_{\hat{x}}$ is very close to any one of the critical angles.
Further, as an application of our new results, we derive the uniform far-field asymptotics for
the solution to the acoustic scattering problem by buried obstacles in a two-layered medium
with a locally rough interface.

It is worth noting that this work is motivated by our recent work \cite{LYZZ2},
where we have considered the inverse acoustic scattering in a two-layered medium with
a locally rough interface and direct imaging methods have been proposed to numerically reconstruct
the locally rough interface from phaseless total-field data or phased far-field data.
In fact, the results obtained in the present paper provide a theoretical foundation for the
direct imaging methods given in \cite{LYZZ2}.

It should be remarked that some related works have been done for the asymptotics of relevant
Green functions. High-frequency asymptotic properties of the two-layered Green function in
two dimensions have been studied in \cite{B84}.
For the two-layered Green function in three dimensions, we refer to \cite{BLE1966,bg_2} for
the analysis of the high-frequency asymptotic properties, and \cite{BLE1966,bg_2,D92} for
the analysis of their far-field asymptotics.
Moreover, \cite{CH95} has derived the uniform far-field asymptotic expansion of
the Green function for the Helmholtz equation in a half-plane with an impedance boundary condition.

The remaining part of the paper is organized as follows.
Sections \ref{sec:2} and \ref{sec:3} are devoted to the theoretical analysis of
the uniform far-field asymptotic property of the two-layered Green function $G(x,y)$
for the cases $x\in\Rb^2_+$ and $x\in\Rb^2_-$, respectively.
In Section \ref{sec:4}, as an application of the results obtained in Sections \ref{sec:2} and \ref{sec:3},
we study the uniform far-field asymptotic property of the solution to the acoustic scattering problem
by buried obstacles in a two-layered medium with a locally rough interface.
Some concluding remarks are given in Section \ref{sec:5}.
We derive the formula (\ref{neq:4.2.33}) in Appendix \ref{sec:a}
and prove Lemmas \ref{Nle:1} and \ref{Le:3.2.1} in Appendix \ref{s3}.

\section{Uniform far-field asymptotic analysis of $G(x,y)$ with $x\in \Rb^2_+$} \label{sec:2}

This section is devoted to studying the uniform far-field asymptotic properties of $G(x,y)$
with $x\in\Rb^2_+$. To this end, we first introduce the following notations, which will be used
throughout the paper. Define $n:={k_{-}}/{k_{+}}$. Denote the angle $\theta_c$ by
\begin{align}\label{eq43}
\theta_c := \begin{cases}
\arccos (n) \in (0,\pi/2), \; & k_+>k_-,\\
\arccos (1/n)\in (0,\pi/2), \; & k_+<k_-.
\end{cases}
\end{align}
For any $x\in\Rb^2$ with $\vert x \vert\neq0$, let $x=(x_1,x_2)$
and $\hat{x}=x/\vert x \vert=(\cos\theta_{\hat{x}},\sin\theta_{\hat{x}})$ with $\theta_{\hat{x}}\in [0,2\pi)$.
For any $d\in\Sp^1:=\{d\in\mathbb{R}^2:\vert d\vert =1\}$, let $d=(d_1,d_2)=(\cos\theta_d,\sin\theta_d)$ with $\theta_d\in[0,2\pi)$. Define $\Sp^1_{\pm}:=\{x\in \Rb^2_\pm :\vert x \vert=1\}$. For any $y=(y_1,y_2)\in\mathbb{R}^2$,
let $y':=(y_1,-y_2)$. For any $R_0>0$, define $B^\pm_{R_0}:=\{y\in\mathbb{R}^2_\pm:\vert y \vert< R_0\}$.
Let $\Cb_0$, $\Cb_1$ and $\Cb_2$ be given by
\begin{align*}
  &\Cb_0: = \Cb \backslash \{z \in \Cb: \Ima(z) =0, \Rt(z) \le 0\},\\
  &\Cb_1: = \Cb \backslash \{z \in \Cb: \Ima(z) \geq 0, \Rt(z) = 0\},\\
  &\Cb_2: = \Cb \backslash \{z \in \Cb: \Ima(z) \leq 0, \Rt(z) = 0\},
\end{align*}
respectively. Then we introduce the functions $z^{\beta}$, $\mathcal S_1(z)$ and $\mathcal S_2(z)$
as follows.
\begin{enumerate}
  \item
  For $\beta\in\Rb$ and $z \in  \Cb_0$ with $z= \vert z\vert e^{i\theta_0}$ and $\theta_0\in (-\pi,\pi)$,
  define $z^{\beta}: = \vert z\vert^{\beta}e^{i\beta{\theta_0}}$.
  For $\beta>0$ and $z=0$, define $z^\beta:=0$.
  For simplicity, the function $z^{\beta}$ with $\beta = 1/2$ is also denoted by $\sqrt{z}$ as usual.
  \item
  For $z \in \Cb_1$ with $z= \vert z\vert e^{i\theta_1}$ and $\theta_1\in (-\frac{3\pi}2, \frac{\pi}2)$,
  let $\mathcal S_{1}(z): =\sqrt{\vert z\vert }e^{i{\theta_1}/2}$.
  For $z=0$ define $\mathcal S_{1}(z):=0$.
  \item
  For $z \in \Cb_2$ with $z= \vert z\vert e^{i\theta_2}$ and $\theta_2\in (-\frac{\pi}2, \frac{3\pi}2)$,
  let $\mathcal S_{2}(z): =\sqrt{\vert z\vert }e^{i{\theta_2}/2}$.
  For $z=0$ define $\mathcal S_{2}(z):=0$.
\end{enumerate}
Moreover, for $z\in\Cb$ and $a>0$ such that $z-a\in\Cb_1\cup\{0\}$ and $z+a\in\Cb_2\cup\{0\}$,
define $\mathcal S(z,a):=\mathcal S_{1}(z-a)\mathcal S_{2}(z+a)$ (note that
$\mathcal S(z,a) =0$ for $z=\pm a$). For $z\in\Cb$ and $a>0$ such that $z-a,z+a\in\Cb_2\cup\{0\}$,
define $\widetilde{\mathcal S}(z,a):=\mathcal S_{2}(z-a)\mathcal S_{2}(z+a)$.
{\color{hw}The branch cuts of $\mathcal S(z,a)$ and $\widetilde{\mathcal S}(z,a)$ with $a>0$
are depicted in Figure \ref{fig2}.}
For $z\in\{z\in\Cb:\Rt(z)=a,\I(z)\ge 0\}$ and $a>0$, it can be seen that
$\lim_{h\in\mathbb{R}, h\rightarrow +0} \mathcal{S}(z\pm h,a)$ exist, which are denoted
as $\mathcal{S}_{\pm}(z,a)$, respectively. It is clear that
\begin{align}\label{eq8}
\mathcal{S}_{+}(z,a)=-\mathcal{S}_{-}(z,a)=\widetilde{\mathcal{S}}(z,a)\quad
\text{for}~z\in \{z\in \Cb: \Rt(z) = a, \I(z)\ge 0\}.
\end{align}
For $c_1,c_2\in\Rb$ with $c_1<c_2$, define the strip
$L_{c_1,c_2}:=\left\{{z\in\mathbb C}:c_1<\I(z)<c_2\right\}$.
In particular, the strip $L_{-c, c}$ with $c>0$ is denoted as $L_{c}$.
For $\theta \in \mathbb{R}$, we define
\be\label{eq:0}
\mathcal R(\theta):=\frac{{i\sin\theta+\mathcal S(\cos\theta,n)}}{{i\sin\theta
-\mathcal S{(\cos\theta,n)}}},\quad \Tc(\theta):={\RR(\theta)}+1.
\en
Throughout this paper, the constants may be different at different places.
If not stated otherwise, for any $a,b\in\mathbb{R}\cup\{\pm\infty\}$ and any function $f$ defined
in the complex plane, $\int_a^b f(s)ds$ is always considered as an integral from $a$ to $b$ along the real axis.

\begin{figure}
\centering
\includegraphics[width=0.95\textwidth]{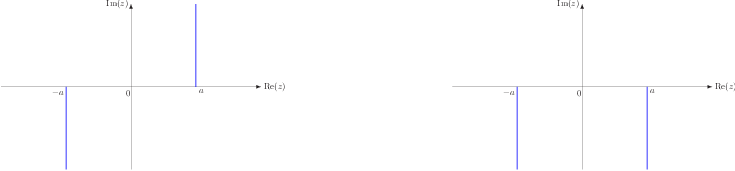}
\caption{{\color{hw}The blue lines are the branch cuts of $\mathcal S(z,a)$ (left) and
$\widetilde{\mathcal S}(z,a)$ (right) with $a>0$.}
}\label{fig2}
\end{figure}

Now we present the explicit formula for the two-layered Green function $G(x,y)$ with $x\in \Rb^2_+$.
For $x=(x_1,x_2)\in \Rb^2_+$ and $y=(y_1,y_2)\in  \Rb^{2}_+ \cup  \Rb^{2}_-$, $G(x,y)$ has the
following form (see, e.g., \cite[formula (2.27)]{Car17}
{\color{hw-a}and \cite[Appendix A]{Lpj}})
\begin{align} \label{e:1.2}
G(x,y) =\left\{\begin{aligned}
&\frac{i}{4}H^{(1)}_0(k_{+}\vert x- y\vert ) + G_{\mathcal R}(x,y),
&&x\in\Rb^2_+, &&y\in\Rb^2_+, \\
&G_{\mathcal T}(x,y),
&&x\in\Rb^2_+, &&y\in\Rb^2_-,
\end{aligned}\right.
\end{align}
where $G_{\mathcal R}(x,y)$ and $G_{\mathcal T}(x,y)$ are given by the so-called Sommerfeld integrals, that is,
\begin{align}
\label{eq4}
&G_{\mathcal R}(x,y): =  \frac{1}{4\pi}\int_{-\infty}^{+\infty}\frac{\mathcal S(\xi,{k_{+}})
-\mathcal S(\xi,{k_{-}})}{\mathcal S(\xi,{k_{+}})+\mathcal S(\xi,{k_{-}})}
\frac{e^{-\mathcal S(\xi,{k_{+}})\vert x_2+y_2\vert}}{\mathcal S(\xi,{k_{+}})}e^{i\xi(x_1-y_1)}d\xi,\nonumber\\
&\qquad\qquad\qquad\qquad\qquad\qquad\qquad\qquad\qquad\qquad\qquad\qquad
x\in\Rb^2_+,~y\in\Rb^2_+,\\ \label{eq5}
&G_{\mathcal T}(x,y): = \frac{1}{2\pi} \int_{-\infty}^{+\infty}\frac{{e}^{\mathcal S(\xi,{k_{-}})y_2
-\mathcal S(\xi,{k_{+}})x_2}}{\mathcal S(\xi,{k_{+}})+\mathcal S(\xi,{k_{-}})}e^{i\xi(x_1-y_1)}d\xi,
~x\in\Rb^2_+,~y\in\Rb^2_-.
\end{align}
Moreover, by introducing the variable $z=\xi/k_{+}$, (\ref{eq4}) and (\ref{eq5}) can be rewritten as
\begin{align}\label{eq:3.2.2}
&G_{\mathcal R}(x,y)=\frac{1}{4\pi}\int^{+\infty}_{-\infty}\frac{\mathcal S(z,1)
-\mathcal S(z,n)}{\mathcal S(z,1)+\mathcal S(z,n)}\frac{e^{-i k_{+}\vert y \vert(z\cos\theta_{\hat y}
-i\mathcal S(z,1)\sin\theta_{\hat y})}}{\mathcal S(z,1)}e^{i k_{+}\vert x \vert p(z,\theta_{\hat x})}dz,\\
&G_{\mathcal T}(x,y) = \frac{1}{2\pi}\int^{+\infty}_{-\infty} \frac{e^{-i k_{+}\vert y \vert
(z\cos\theta_{\hat y}+ i\mathcal S{(z,n)}\sin\theta_{\hat y})}}{\mathcal S{(z,1)}
+\mathcal S{(z,n)}}e^{i k_{+}\vert x \vert p(z,\theta_{\hat x})}dz,
\end{align}
where $x=\vert x \vert(\cos\theta_{\hat x},\sin\theta_{\hat x})$ with $\theta_{\hat x}\in (0,\pi)$,
$y = \vert y \vert(\cos\theta_{\hat y},\sin\theta_{\hat y})$ with $\theta_{\hat y}\in (0,\pi)\cup(\pi,2\pi)$
and $p(z,\theta_{\hat x}):= z\cos\theta_{\hat x}+ i\mathcal S{(z,1)}\sin\theta_{\hat x}$.

{\color{hw}
As mentioned in Section \ref{sec:1}, \cite{Car17} derived the pointwise estimates of the far-field asymptotics
of $G(x,y)$ for all angles $\theta_{\hat{x}}\in(0,\pi)$ in the case $k_+<k_-$ and for all angles
$\theta_{\hat{x}}\in(0,\pi)\ba\{\theta_c,\pi-\theta_c\}$ in the case $k_+>k_-$ (see Section 2.3.4
in \cite{Car17}).
These pointwise estimates were established in \cite{Car17} by using the steepest descent method and
were essentially based on the pointwise far-field asymptotic estimates of $G_{\mathcal R}(x,y)$
and $G_{\mathcal T}(x,y)$ given in \cite[(2.57a) and (2.57b)]{Car17}.}

In the following two subsections, we will study the uniform far-field asymptotic properties of
the two-layered Green function $G(x,y)$ for the case $k_+<k_-$ and for the case $k_+>k_-$, respectively.

{\color{hw}
For the case $k_+<k_-$, it is easy to prove that the far-field asymptotic estimates of $G_{\mathcal R}(x,y)$
and $G_{\mathcal T}(x,y)$ given in \cite[(2.57a) and (2.57b)]{Car17}
hold uniformly for all angles $\theta_{\hat{x}}\in(0,\pi)$ by using the same approach as given
in \cite[Section 2.3.4]{Car17}.
This is essentially due to the fact that in the case $k_+<k_-$, since $n=k_-/k_+>1$ then
the distances between the saddle point $\cos\theta_{\hat{x}}$ and the branch points $\pm n$ involved in
relevant integrals given in \cite[Section 2.3.4]{Car17} have a uniform positive lower bound for all $\theta_{\hat{x}}\in(0,\pi)$.
Therefore, one can easily obtain the uniform far-field asymptotic properties of $G(x,y)$ for all
angles $\theta_{\hat{x}}\in(0,\pi)$ in the case $k_+<k_-$; see Section \ref{sec:21} below.

However, for the case $k_+>k_-$,
the far-field asymptotic estimates of $G_{\mathcal R}(x,y)$ and $G_{\mathcal T}(x,y)$ given
in \cite[(2.57a) and (2.57b)]{Car17} do not hold uniformly for the angles $\theta_{\hat{x}}$
in the vicinity of the critical angles $\theta_c$ and $\pi-\theta_c$. In fact,
when the angle $\theta_{\hat{x}}$ is in the vicinity of the critical angle $\theta_c$, this
nonuniform property can be seen from the fact that the factor $|\sin(\theta_c-\theta)|^{-3/2}$
involved in \cite[formulas (2.54a) and (2.54b)]{Car17} will blow up if $\theta$ tends to $\theta_c$,
where the formulas (2.54a) and (2.54b) in \cite{Car17} are related to $G_{\mathcal R}(x,y)$ and
$G_{\mathcal T}(x,y)$, respectively, and where the notations $\theta$ and $\theta_c$ used
in \cite[formulas (2.54a) and (2.54b)]{Car17} correspond to $\theta_{\hat{x}}$ and $\theta_c$, respectively,
used in the present paper. Moreover, when the angle $\theta_{\hat{x}}$ is in the vicinity of the critical
angle $\pi-\theta_c$, this nonuniform property can be easily seen from the above discussion on
\cite[formulas (2.54a) and (2.54b)]{Car17} as well as a symmetry relation of $G(x,y)$ given in \eqref{eq48}.
Furthermore, it is worth noting that this nonuniform property is essentially due to the fact that,
in the case $k_+>k_-$, since $n=k_-/k_+<1$ then $\cos\theta_{\hat{x}}$ will be very close to $n$ (resp. $-n$)
when $\theta_{\hat{x}}$ approaches $\theta_c$ (resp. $\pi-\theta_c$), where $\cos\theta_{\hat{x}}$ and $\pm n$
are the saddle point and the branch points, respectively, as mentioned above.
Therefore, the method for deriving the far-field asymptotic estimates in \cite[(2.57a) and (2.57b)]{Car17}
does not work for the derivation of the uniform far-field asymptotics of $G(x,y)$ for all
angles $\theta_{\hat{x}}\in(0,\pi)$ in the case $k_+>k_-$.

In Remark \ref{re6} below, we give a detailed explanation on the main difficulties in deriving the
uniform far-field asymptotic estimates of $G(x,y)$ for all angles $\theta_{\hat{x}}\in(0,\pi)$
in the case $k_+>k_-$. We will also propose a method to address these difficulties after
Remark \ref{re6} in Section \ref{sec:22}.
}

\subsection{The case $k_{+}<k_{-}$}\label{sec:21}

In this case, we have the following theorem on the uniform far-field asymptotic properties of $G(x,y)$,
which is mainly based on the results in \cite{Car17}.

\begin{theorem}\label{NLe:4}
Assume that $k_+<k_-$ and let $R_0>0$ be an arbitrary fixed number. Suppose that
$y=(y_1,y_2)\in B^+_{R_0}\cup B^-_{R_0}$ and $x=\vert x \vert\hat x=\vert x \vert(\cos\theta_{\hat x},\sin\theta_{\hat x})\in\Rb^2_+$
with $\theta_{\hat{x}}\in(0,\pi)$, then we have the asymptotic behavior
\begin{align}
& G(x,y) = \frac{e^{ik_{+}\vert x \vert}}{\sqrt{\vert x \vert}}G^{\infty}(\hat x, y) + G_{Res}(x,y), \label{eq:3.2.1}\\
& \nabla_y G(x,y) =  \frac{e^{ik_{+}\vert x \vert}}{\sqrt{\vert x \vert}}H^{\infty}(\hat x,y ) +  H_{Res}(x,y), \label{eq:4.2.17}
\end{align}
where $G^{\infty}$, $H^{\infty}$ are defined by
\begin{align} \label{eq:t7}
G^{\infty}(\hat x,y):= \frac{e^{i\frac{\pi}4}}{\sqrt{8\pi k_{+}}}
\left\{\begin{aligned}
&e^{-i k_{+}{\hat x} \cdot y} + \mathcal R(\theta_{\hat x}) e^{-ik_{+}{\hat x}\cdot y{'}},
&&\hat{x}\in\mathbb{S}^1_+,~y\in\Rb^2_+,\\
&\mathcal T(\theta_{\hat x})e^{-ik_{+}(y_1\cos\theta_{\hat x}+iy_2\mathcal S{(\cos\theta_{\hat x},n)})},
&&\hat{x}\in\mathbb{S}^1_+,~y\in\Rb^2_-,\\
\end{aligned}\right.
\end{align}
\begin{align} \label{eq:4.2.18}
H^{\infty}(\hat x,y ) := {e^{-i\frac{\pi}4}}\sqrt{\frac{k_{+}}{8\pi}}
\left\{\begin{aligned} &e^{-i k_{+}\hat x \cdot y} \begin{pmatrix}
 \cos \theta_{\hat x} \\
 \sin \theta_{\hat x}
\end{pmatrix}^T  +  \mathcal R(\theta_{\hat x})e^{-ik_{+}{\hat x}\cdot y{'}} \begin{pmatrix}
\cos \theta_{\hat x} \\
 -\sin \theta_{\hat x}
\end{pmatrix}^T,\\
&\qquad\qquad\qquad\qquad\qquad\qquad\qquad\qquad
\hat{x}\in\mathbb{S}^1_+,~y\in\Rb^2_+, \\
& \mathcal T(\theta_{\hat x})e^{-ik_{+}(y_1\cos\theta_{\hat x}+iy_2\mathcal S{(\cos\theta_{\hat x},n)})}
\begin{pmatrix}
 \cos \theta_{\hat x}\\
 i \mathcal S{(\cos\theta_{\hat x},n)}
\end{pmatrix}^T,\\
&\qquad\qquad\qquad\qquad\qquad\qquad\qquad\qquad
\hat{x}\in\mathbb{S}^1_+,~y\in\Rb^2_-,
\end{aligned}\right.
\end{align}
and $G_{Res}$ and $H_{Res}$ satisfy the estimates
\ben
\vert G_{Res}(x,y)\vert , ~\vert H_{Res}(x,y)\vert  \le {C_{R_0}}{\vert x \vert^{-3/2}},\quad \vert x \vert\rightarrow+\infty,
\enn
uniformly for all $\theta_{\hat{x}}\in(0,\pi)$ and $y\in B^+_{R_0}\cup B^-_{R_0}$.
Here, the constant $C_{R_0}>0$ is independent of $x$ and $y$ but dependent of $R_0$.
\end{theorem}

\begin{proof}
By using the steepest decent method,
it was proved in \cite{Car17} that for fixed $\theta_{\hat{x}}\in(0,\pi)$ and
$y\in\mathbb{R}^2_+\cup\Rb^2_-$, $G(x,y)$ and $\nabla_y G(x,y)$
satisfy (\ref{eq:3.2.1}) and (\ref{eq:4.2.17}), respectively, with $G_{Res}(x,y)=O(\vert x \vert^{-3/2})$
and $H_{Res}(x,y)=O(\vert x \vert^{-3/2})$ (see pages 39--41 in \cite{Car17}).
By using the same arguments as in \cite{Car17}, it can be proved that the statement of
this theorem holds true.
\end{proof}

\subsection{The case $k_{+}>k_{-}$}\label{sec:22}

In view of (\ref{NNe6}) and (\ref{e:1.2}), we need to investigate the uniform far-field
asymptotic properties of $G_{\mathcal R}(x,y)$ and $G_{\mathcal T}(x,y)$ with $\theta_{\hat{x}}\in(0,\pi)$.
In what follows, we will give the detailed analysis of $G_{\mathcal R}(x,y)$ with $\theta_{\hat{x}}\in(0,\pi/2)$ {\color{hw-a}(see Lemmas \ref{Le:n4.2.1}, \ref{NLe:3.2.4} and \ref{Le:n4.2.2} below)}.
For $G_{\mathcal R}(x,y)$ with $\theta_{\hat{x}}\in\left[\left.\pi/2,\pi\right)\right.$ and
$G_{\mathcal T}(x,y)$ with $\theta_{\hat{x}}\in(0,\pi)$, the theoretical analyses are similar.
{\color{hw-a}The main results of this subsection on the uniform far-field asymptotics of $G(x,y)$
for all angles $\theta_{\hat{x}}\in(0,\pi)$ will be given in Theorem \ref{NLe:2}.}

Motivated by the ideas in \cite{B84,bg_2,CH95}, we will use the steepest descent method to study
$G_{\mathcal R}(x,y)$. For this, we introduce some paths and curves (see Figure \ref{nf1}).
{\color{hw}
Define the function $h_{D}(z):=-\sgn(z)\arccos[\sech(z)]$ for $z\in\mathbb{R}$,
where $\sgn$ denotes the sign function. It is clear that $h_{D}(z)$ is an odd and strictly monotonically
decreasing function in $z\in\mathbb{R}$.
}
Let the path $\mathcal D$ denote the curve
\ben{\color{hw}
{\mathcal I}_D:=\big\{\zeta\in\mathbb C:\Rt(\zeta)=\theta_{\hat{x}}+h_{D}(\I(\zeta))
\big\}}
\enn
with the orientation from $-\pi/2 + \theta_{\hat x}+ i\infty $ to $\pi/2 + \theta_{\hat x} - i\infty$
(see Figure \ref{nf1}(a) for the case
$\theta_c\leq\theta_{\hat{x}}<\frac{\pi}{2}$ and
Figure \ref{nf1}(b) for the case $0<\theta_{\hat{x}}<\theta_c$).
{\color{hw}
Note that at infinity the path $\mathcal D$ asymptotically approaches the lines $\Rt(\zeta)=-\pi/2+\theta_{\hat{x}}$
and $\Rt(\zeta)=\pi/2+\theta_{\hat{x}}$.}
See Remark \ref{re1} for the properties of the path $\mathcal D$.
{\color{hw}
Moreover, define the function $h_B(z):=\arccos[\cos(\theta_c)\sech(z)]$ for $z\in\mathbb{R}$.
It is easy to see that $h_B(z)$ is an even function and is strictly monotonically increasing in $z\in\left.\left[0,+\infty\right.\right)$.
}
Let ${\mathcal I}_m:= {\mathcal I}^{+}_m \cup  {\mathcal I}^{-}_m$ with $m \in \mathbb Z$, where
\begin{align*}
{\mathcal I}^{+}_m&:= \big\{\zeta \in \Cb: \Rt(\zeta)
= {\color{hw}h_B(\I(\zeta))+m\pi}, \Ima(\zeta) \le 0 \big\}
\end{align*}
is a curve connecting $\theta_c+m\pi$ and $\pi/2+m\pi-i\infty$, and
\begin{align*}
{\mathcal I}^{-}_m&:= \big\{\zeta \in \Cb: \Rt(\zeta)
= {\color{hw}- h_B(\I(\zeta))+m\pi}, \Ima(\zeta) \ge 0 \big\}
\end{align*}
is a curve connecting $-\theta_c+m\pi$ and $-\pi/2+m\pi+i\infty$.
{\color{hw}
Note that for $m\in\mathbb Z$, both the curves ${\mathcal I}^{+}_m$ and ${\mathcal I}^{-}_{m+1}$ asymptotically
approach the line $\Rt(\zeta)=\pi/2+m\pi$ at infinity.
}
Let $\mathcal{L}_o$ denote the loop around the curve $\mathcal{I}^+_0$
with the orientation indicated in Figure \ref{nf1}(b).
For $l\in\mathbb Z$, let ${\mathcal{\widetilde I}}^{\pm}_{2l}
:=\left\{\Rt(z)-i\Ima(z):z\in\mathcal I^{\pm}_{2l}\right\}$
{\color{hw}
denote the reflections of $\mathcal{I}^{\pm}_{2l}$, respectively, about the real axis} and define
${\mathcal {\widetilde{I}}}_{2l}:={\mathcal {\widetilde I}}^{+}_{2l}\cup{\mathcal{\widetilde I}}^{-}_{2l}$.
Some properties of the curves ${\mathcal {{I}}}_{2l}$, ${\mathcal {{I}}}_{2l+1}$ and
${\mathcal{\widetilde{I}}}_{2l}$ with $l\in\mathbb Z$ are presented in Remark \ref{re2} below.


\begin{remark}\label{re1} {\rm
It is easily seen that the path $\mathcal{D}$ will cross the real axis
with only one intersection point $\theta_{\hat{x}}$.
Moreover,
by patient but straightforward calculations, we have
\be \label{eq:12}
\sqrt{2k_{+}}e^{i\frac{\pi}4}\sin{\left(\frac{\zeta-\theta_{\hat x}}{2}\right)}
=-\sqrt{k_+}\frac{\sinh(\I(\zeta))}{\sqrt{\cosh(\I(\zeta))}}~{\rm for~}\zeta\in\mathcal D,
\en
which is needed for later use.
From this, it can be easily derived that $\I(ik_{+}\cos(\zeta - \theta_{\hat x}))= k_{+}$
for $\zeta\in\mathcal D$ and $\Rt(ik_+\cos(\zeta-\theta_{\hat x}))=-k_+\sinh^2(\I(\zeta))/\cosh(\I(\zeta))<0$
for $\zeta \in \mathcal D \backslash \{\theta_{\hat x}\}$.
Actually, we note that $\mathcal{D}$ is the steepest descent path of the function
$ik_{+}\cos(\zeta - \theta_{\hat x})$ through the point $\theta_{\hat x}$.
{\color{hw}
Here, for any analytic function $W(\zeta)$ of the complex variable $\zeta$,
an integration path $\mathcal{Q}$ in the complex plane is called a steepest descent path of $W(\zeta)$
if $\I(W(\zeta))$ is constant on $\mathcal{Q}$ (see \cite{O70}).
Furthermore,
it is worth noting that $ik_{+}\cos(\zeta - \theta_{\hat x})$ has
only one saddle point $\zeta=\theta_{\hat{x}}$
on the steepest descent path $\mathcal D$.
Thus, it can be seen from \cite[Section 7.2]{B84} that
$\Rt(ik_+\cos(\zeta-\theta_{\hat x}))$ decreases most rapidly
from $\theta_{\hat{x}}$ to $-\pi/2 + \theta_{\hat x}+ i\infty $ and
from $\theta_{\hat{x}}$ to $\pi/2 + \theta_{\hat x} - i\infty$ along the curve $\mathcal{I}_D$.}
For more introductions and properties of steepest descent paths,
we refer to \cite{B84,BH86,bg_2,O70}.}
\end{remark}

\begin{remark}\label{re2} {\rm
Let $\mathcal A_1 := \bigcup_{l\in \mathbb Z} {\mathcal I}_{2l}$,
$\mathcal A_2:=\bigcup_{l\in \mathbb Z}{\mathcal I}_{2l+1}$ and
$\mathcal A_3:=\bigcup_{l\in\mathbb Z} \widetilde{\mathcal I}_{2l}$.
By straightforward calculations, it is easily seen that
\ben
&&\mathcal{A}_1=\{\zeta\in \Cb: \Rt(\cos \zeta) = n,~\I(\cos \zeta)\ge 0\},\\
&&\mathcal{A}_2=\{\zeta\in \Cb: \Rt(\cos \zeta) = -n,~\I(\cos \zeta)\le 0\},\\
&&\mathcal{A}_3=\{\zeta\in \Cb: \Rt(\cos \zeta) = n,~\I(\cos \zeta)\leq 0\},
\enn
which imply that $\mathcal S(\cos \zeta, n)$ is analytic in
$\zeta\in\Cb\backslash\left(\mathcal A_1\cup\mathcal A_2\right)$ and
$\widetilde {\mathcal S}(\cos \zeta, n)$ is analytic in
$\zeta\in\Cb \backslash \left(\mathcal A_2 \cup \mathcal A_3\right)$. Thus
the curves in $\mathcal A_1 \cup \mathcal A_2$ are the branch cuts of $\mathcal S(\cos \zeta, n)$
and the curves in $\mathcal A_2 \cup \mathcal A_3$ are the branch cuts of
$\widetilde {\mathcal S}(\cos \zeta, n)$.
Moreover, it can be deduced by direct calculations that the path $\mathcal D$ lies in
$\Cb\backslash \left(\mathcal A_1\cup \mathcal A_2\right)$ for the case
$\theta_{\hat{x}}\in(\theta_c,\pi/2)$, the path $\mathcal D$ will cross $\mathcal A_1\cup\mathcal A_2$
with only one intersection point $\theta_c$ for the case $\theta_{\hat{x}}=\theta_c$
and the path $\mathcal D$ lies in $\Cb \backslash \left(\mathcal A_2 \cup \mathcal A_3\right)$
for the case $\theta_{\hat{x}}\in(0,\theta_c)$ (see Figure \ref{nf1}).
{\color{hw}Here, we note that $\zeta=\theta_c$ is a branch point of $\mathcal S(\cos \zeta, n)$
and $\widetilde {\mathcal S}(\cos \zeta, n)$.}
}
\end{remark}

\begin{figure}
\centering
\includegraphics[width=0.9\textwidth]{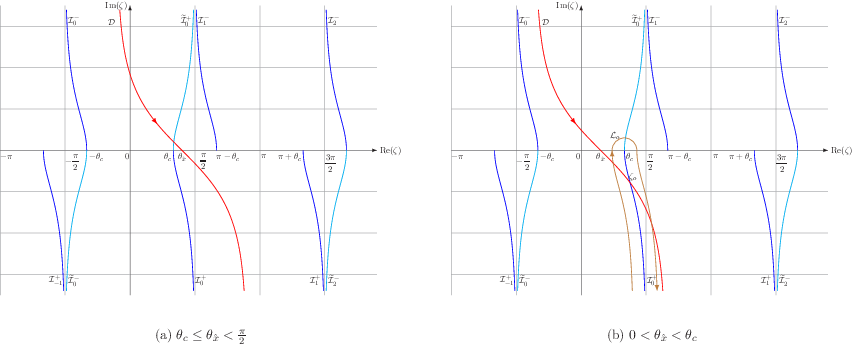}
\caption{The steepest descent path $\mathcal D$, the path $\mathcal{L}_o$, the branch cuts
$\mathcal{I}^\pm_m$ ($m\in\mathbb{Z}$) and the branch cuts ${\mathcal {\widetilde I}}^{\pm}_{2l}$
($l\in\mathbb{Z}$). $\zeta_o$ denotes the intersection point of $\mathcal D$ and $\mathcal I^+_0$
for the case $0<\theta_{\hat{x}}<\theta_c$.}\label{nf1}
\end{figure}

Now we have the following lemma, {\color{hw}which shows that the function $G_{\mathcal R}(x,y)$
with $\theta_{\hat x}\in (0,{\pi}/2)$
can be rewritten as the sum of integrals over the steepest descent path $\mathcal{D}$
and (possibly) an integral over the loop $\mathcal{L}_o$}.

\begin{lemma}\label{Le:n4.2.1}
Assume that $k_+>k_-$. Let $y=\vert y \vert(\cos \theta_{\hat y}, \sin \theta_{\hat y})\in\mathbb{R}^2_+$
with $\theta_{\hat y}\in (0,\pi)$
and $x = \vert x \vert(\cos \theta_{\hat x}, \sin \theta_{\hat x})\in\mathbb{R}^2_+$ with
$\theta_{\hat x}\in (0,{\pi}/2)$. Then the following statements hold.
\begin{enumerate}
\item
For $\theta_{\hat{x}}\in\left[\left.\theta_c,\pi/2\right)\right.$ and $\vert x \vert>\vert y \vert$, we have
\begin{align}\label{neq:4.2.7}
G_{\mathcal R}(x,y) = G^{(1)}_{\mathcal R}(x,y) + G^{(2)}_{\mathcal R}(x,y),
\end{align}
where $G^{(1)}_{\mathcal R}$ and $G^{(2)}_{\mathcal R}$ are given by
\begin{align}
\label{eq7}
&G^{(1)}_{\mathcal R}(x,y) := \frac{i}{4\pi}\int_{\mathcal D}\frac{\cos(2\zeta)-n^2}{n^2-1}e^{ik_{+}
\left(-\vert y \vert\cos(\zeta+\theta_{\hat y})+ \vert x \vert\cos(\zeta-\theta_{\hat x})\right)} d\zeta,\\
\label{eq55}
&G^{(2)}_{\mathcal R}(x,y) := \frac{i}{4\pi}\int_{\mathcal D}
\frac{2i\sin\zeta\mathcal S(\cos\zeta,n)}{n^2-1}e^{ik_{+}\left(-\vert y \vert\cos(\zeta+\theta_{\hat y})
+ \vert x \vert\cos(\zeta-\theta_{\hat x})\right)} d\zeta.
\end{align}

\item
For $\theta_{\hat x}\in\left(0, \theta_c\right)$ and $\vert x \vert>\vert y \vert/\cos(\theta_c)$, we have
\begin{align} \label{neq:4.2.8}
G_{\mathcal R}(x,y)& =G^{(1)}_{\mathcal R}(x,y)+G^{(3)}_{\mathcal R}(x,y)+G^{(4)}_{\mathcal R}(x,y),
\end{align}
where $G^{(1)}_{\mathcal R}$ is given by (\ref{eq7}), and $G^{(3)}_{\mathcal R}$ and
$G^{(4)}_{\mathcal R}$ are given by
\begin{align}
\label{eq56}
G^{(3)}_{\mathcal R}(x,y) := & \frac{i}{4\pi}\int_{\mathcal D}\frac{2i\sin\zeta
\widetilde{{\mathcal S}}(\cos\zeta,n)}{n^2-1} {e^{ik_{+}\left(-\vert y \vert\cos(\zeta+\theta_{\hat y})
+ \vert x \vert\cos(\zeta-\theta_{\hat x})\right)}} d\zeta,\\
\label{eq59}
G^{(4)}_{\mathcal R}(x,y) := & \frac{i}{4\pi}\int_{\mathcal{L}_o}
\frac{2i\sin\zeta\mathcal S{(\cos\zeta,n)}}{n^2-1} {e^{ik_{+}
\left(-\vert y \vert\cos(\zeta+\theta_{\hat y})+ \vert x \vert\cos(\zeta-\theta_{\hat x})\right)}} d\zeta.
\end{align}
\end{enumerate}
\end{lemma}

\begin{proof}
Let $x$ and $y$ be fixed throughout the proof.
Motivated by \cite{AIL05,CH95}, we introduce the change of variable
\begin{align}\label{eq57}
z= \left\{\begin{aligned}
&\cos \eta, &&0 <\eta <\pi, &&\textrm{for}~ -1 < z < 1,  \\
&\cos i\eta, &&\eta\ge 0,&&\textrm{for}~  z\ge 1,  \\
&\cos (\pi+i\eta), &&\eta\le 0,&&\textrm{for}~ z\le -1.
\end{aligned}\right.
\end{align}
Then the formula (\ref{eq:3.2.2}) can be rewritten as
\begin{align}
G_{\mathcal R}(x,y)= &\frac{i}{4\pi}\bigg\{\int^{\pi}_{0} F(\eta,x,y) d\eta
+ \int_{+\infty}^{0} F(i\eta, x,y) i d\eta  + \int^{-\infty}_{0} F(i\eta+\pi, x,y)id\eta\bigg\}\notag \\
= & \frac{i}{4\pi}\int_{\mathcal L} F(\zeta, x, y) d\zeta,\label{eq58}
\end{align}
where $F(\zeta,x,y):=e^{ik_{+}\left[-\vert y \vert\cos(\zeta+\theta_{\hat y})+\vert x \vert\cos(\zeta-\theta_{\hat x})\right]} (-i\sin\zeta-\mathcal{S}{(\cos\zeta,n)})/(-i\sin\zeta+\mathcal{S}{(\cos\zeta,n)})$
and $\mathcal L$ denotes the piecewise linear path with the orientation from
$0+i\infty\rightarrow 0\rightarrow\pi\rightarrow\pi-i\infty$ in the complex plane (see Figure \ref{fig4.2.2}).
{\color{hw-a}Here, we note that the term $\mathcal S(z,1)$ involved in (\ref{eq:3.2.2}) has
two branch cuts and two zeros (see Figure \ref{fig2}). By the change of variable \eqref{eq57},
such function $\mathcal S(z,1)$ is transformed to the entire function $-i\sin\zeta$ and its reciprocal $(\mathcal S(z,1))^{-1}$ involved in (\ref{eq:3.2.2}) is eliminated
due to the chain rule. This indeed leads to a simplified integral expression \eqref{eq58} of $G_{\mathcal R}(x,y)$,
which makes it more convenient to investigate the asymptotic properties of $G_{\mathcal R}(x,y)$ later on in the present subsection.}

Next, we rewrite $G_{\mathcal R}(x,y)$ with the aid of Cauchy integral theorem.
To this end, we distinguish between the following two cases.

\textbf{Case 1:} $\theta_{\hat{x}}\in\left[\left.\theta_c,\pi/2\right)\right.$.
By a straightforward calculation, we have that for $\zeta\in\Cb$,
\begin{align}
&\Ima\left(-\vert y \vert\cos(\zeta+\theta_{\hat y}) + \vert x \vert\cos(\zeta-\theta_{\hat x})\right)\nonumber\\
&\qquad\qquad\qquad
= \sinh(\I(\zeta))\left(\vert y \vert\sin(\Rt(\zeta)+\theta_{\hat y})-\vert x \vert\sin(\Rt(\zeta)-\theta_{\hat x})\right).\label{neq:4.2.3}
\end{align}
Let ${\Rmnum{1}}_{+}:=\{\zeta\in\Cb:\Ima(\zeta)>0,-\frac{\pi}2+\theta_{\hat x}\le\Rt(\zeta)\le 0\}$.
It is easy to verify that
$\sin(\Rt(\zeta)-\theta_{\hat x})\leq\min(-\sin\theta_{\hat{x}},\sin(\Rt(\zeta)+\theta_{\hat y}))$
for $\zeta \in {\Rmnum{1}}_{+}$.
This, together with (\ref{neq:4.2.3}), implies that for $\zeta\in{\Rmnum{1}}_{+}$ and $\vert x \vert>\vert y \vert$,
\begin{align*}
\left\vert e^{ik_{+}\left(-\vert y \vert\cos(\zeta+\theta_{\hat y}) + \vert x \vert\cos(\zeta-\theta_{\hat x})\right)}\right\vert
\le e^{-k_+ \sinh(\I(\zeta)) (\vert x \vert-\vert y \vert)\sin\theta_{\hat x}}.
\end{align*}
Further, choose $\varepsilon>0$ small enough so that
$\varepsilon<\min(\pi/2-\theta_{\hat{x}},(\theta_{\hat{x}}+\theta_{\hat{y}})/2)$
and let ${\Rmnum{1}}_{-,\varepsilon}:= \{\zeta\in\Cb:\Ima(\zeta)<0,
\frac{\pi}2+\theta_{\hat x}-\varepsilon\le\Rt(\zeta)\le \pi \}$.
Note that $\theta_{\hat x}<\pi/2-\varepsilon\leq \Rt(\zeta)-\theta_{\hat x}\leq\pi-\theta_{\hat x}$
and $\pi/2+\varepsilon<\Rt(\zeta)+\theta_{\hat y}<2\pi$ for
$\zeta \in {\Rmnum{1}}_{-,\varepsilon}$.
Then we can easily obtain that $\sin(\Rt(\zeta)-\theta_{\hat x})\geq\max(\sin\theta_{\hat{x}},
\sin(\Rt(\zeta) + \theta_{\hat y}))$ for $\zeta\in{\Rmnum{1}}_{-,\varepsilon}$.
Thus it follows from (\ref{neq:4.2.3}) that: for $\zeta \in {\Rmnum{1}}_{-,\varepsilon}$ and $\vert x \vert>\vert y \vert$,
\ben
\left\vert e^{ik_{+}\left(-\vert y \vert\cos(\zeta+\theta_{\hat y}) + \vert x \vert\cos(\zeta-\theta_{\hat x})\right)}\right\vert
\le e^{k_+ \sinh(\I(\zeta)) (\vert x \vert-\vert y \vert)\sin\theta_{\hat x}}.
\enn
It is easily seen that
\be\label{eq9}
\vert \sin (\zeta)\vert \le \vert\sinh(\I(\zeta))\vert + \cosh(\I(\zeta))~{\rm for~}\zeta\in\mathbb{C},
\en
thus we obtain
\begin{align*}
\left\vert\frac{-i\sin\zeta-\mathcal S{(\cos\zeta,n)}}{-i\sin\zeta+\mathcal S{(\cos\zeta,n)}}\right\vert
&=\left\vert\frac{\cos(2\zeta)-n^2}{n^2-1} + \frac{2i\sin\zeta\mathcal S(\cos\zeta,n)}{n^2-1}\right\vert\\
&\le C \left\{1+\left[\vert\sinh(\I(\zeta))\vert + \cosh(\I(\zeta))\right]^2\right\}~{\rm for~}\zeta\in\mathbb{C},
\end{align*}
where $C>0$ is a constant independent of $\zeta$. Hence, for
$\zeta\in {\Rmnum{1}}_{-,\varepsilon}\cup {\Rmnum{1}}_+$ and fixed $x,y$ with $\vert x \vert>\vert y \vert$,
the function $F(\zeta, x, y)$ is exponentially decreasing as $\vert\I(\zeta)\vert \rightarrow +\infty$.

Define the domain ${\Rmnum{2}}:={\Rmnum{2}}_1 \cup {\Rmnum{2}}_2\cup\mathcal{I}_{\theta_c,\pi-\theta_c}$
in the complex plane, where
$\mathcal{I}_{\theta_c,\pi-\theta_c}:=\{s\in\mathbb{R}:s\in(\theta_c,\pi-\theta_c)\}$,
${\Rmnum{2}}_1$ is the domain enclosed by ${\mathcal I}^-_{0}$, ${\mathcal I}^-_{1}$ and the curve $\{s\in\mathbb{R}:s\in(-\theta_c,\pi-\theta_c)\}$, and ${\Rmnum{2}}_2$ is the domain enclosed by
${\mathcal I}^+_{0}$, ${\mathcal I}^+_{1}$ and the curve $\{s\in\mathbb{R}:s\in(\theta_c,\pi+\theta_c)\}$.
Here, it should be mentioned that $-\theta_c$, $\theta_c$, $\pi-\theta_c$ and $\pi+\theta_c$
are the endpoints of $\mathcal{I}^-_0$, $\mathcal{I}^+_0$, $\mathcal{I}^-_1$ and
$\mathcal{I}^+_1$, respectively.
Note that the paths $\mathcal D$ and $\mathcal L$ lie in $\ov{II}$ (see Figure \ref{fig4.2.2}(a)),
and $F(\zeta, x, y)$ is analytic for $\zeta\in II$ (see Remark \ref{re2}).
Therefore, using the above arguments and applying Cauchy integral theorem, we obtain (\ref{neq:4.2.7}).

\begin{figure}
\centering
\includegraphics[width=0.9\textwidth]{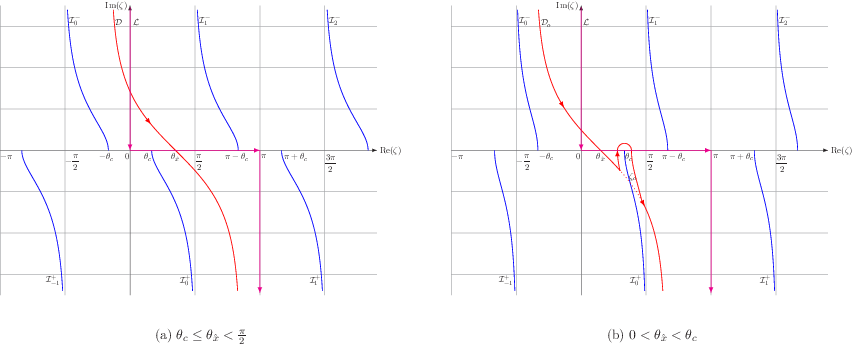}
\caption{The original path $\mathcal L$, the steepest descent path $\mathcal D$ for the case
 $\theta_c\leq \theta_{\hat{x}}<\frac{\pi}{2}$ and the path $\mathcal{D}_o$ for the case $0<\theta_{\hat{x}}<\theta_c$.}\label{fig4.2.2}
\end{figure}

\textbf{Case 2}: $\theta_{\hat{x}}\in\left(0,\theta_c\right)$. In this case,
it can be deduced that the path $\mathcal D$ will cross the branch cut $\mathcal I^+_0$ with
only one intersection point, which is denoted as $\zeta_o$ (see Figure \ref{nf1}(b)).
Thus, using the similar arguments as in Case 1, we have that for $\vert x \vert>\vert y \vert$,
\ben
G_{\mathcal R}(x,y)= G^{(1)}_{\mathcal R}(x,y)+\frac{i}{4\pi}\int_{\mathcal D_o}
\frac{2i\sin\zeta\mathcal S(\cos\zeta,n)}{n^2-1}{e^{ik_{+}\left(-\vert y \vert\cos(\zeta+\theta_{\hat y})
+ \vert x \vert\cos(\zeta-\theta_{\hat x})\right)}} d\zeta,
\enn
where the path $\mathcal{D}_o$ is depicted in Figure \ref{fig4.2.2}(b).
We note that, compared with the path $\mathcal D$,
the path $\mathcal D_o$ has an additional finite-length path around the branch cut $\mathcal{I}^+_0$.

If we shrink the additional path for $\mathcal D_o$, then it follows from
(\ref{eq8}) and the property of $\mathcal{I}^+_0$ given in Remark \ref{re2} that for $\vert x \vert>\vert y \vert$,
\be\label{eq:10}
G_{\mathcal R}(x,y)= G^{(1)}_{\mathcal R}(x,y)+G^{(2)}_{\mathcal R}(x,y)+G_{\mathcal R,\zeta_o,1}(x,y),
\en
where
\begin{align}\label{eq41}
G_{\mathcal R,\zeta_o,1}(x,y):=\frac{i}{2\pi}\int_{\mathcal I_{\theta_c,\zeta_o}}
\frac{2i\sin\zeta\mathcal S_{-}(\cos\zeta,n)}{n^2-1}{e^{ik_{+}\left(-\vert y \vert\cos(\zeta+\theta_{\hat y})
+ \vert x \vert\cos(\zeta-\theta_{\hat x})\right)}}d\zeta.
\end{align}
Here, the path $\mathcal I_{\theta_c, \zeta_o}$ denotes the curve
$\{\zeta\in\mathcal I^+_0:\Rt(\zeta)\in\left[\theta_c,\Rt(\zeta_o)\right]\}$
with its orientation from $\theta_c$ to $\zeta_o$.

From the definition of the path $\mathcal D$, it is easily seen that
\ben
\Rt(\cos \zeta)=\cos(\Rt(\zeta))/\cos(\Rt(\zeta)-\theta_{\hat x}),\quad
\Rt(\zeta)-\theta_{\hat x}\in(-\pi/2,\pi/2)\quad
\textrm{for}~\zeta\in\mathcal D.
\enn
Then it easily follows that: for $\zeta\in\mathcal D$,
$\Rt(\cos\zeta)$ is monotonously decreasing as $\Rt(\zeta)$ increases.
This, together with the fact that $\Rt(\cos\zeta_o)=n$ (see Remark \ref{re2}), implies that
$\Rt(\cos\zeta)<n$ for $\zeta\in \mathcal D_{\zeta_o,\frac\pi2+\theta_{\hat{x}}-i\infty}$ and
$\Rt(\cos\zeta)>n$ for $\zeta\in \mathcal D_{-\frac\pi2+\theta_{\hat{x}}+i\infty,\zeta_o}$,
where the path $\mathcal D_{\zeta_o,\frac\pi2+\theta_{\hat{x}}-i\infty}$ denotes the part of
the path $\mathcal{D}$ starting from $\zeta_o$ and ending at $\frac\pi2+\theta_{\hat{x}}-i\infty$,
and the path $\mathcal D_{-\frac\pi2+\theta_{\hat{x}}+i\infty,\zeta_o}$ denotes the part
of the path $\mathcal{D}$ starting from
$-\frac\pi2+\theta_{\hat{x}}+i\infty$ and ending at $\zeta_o$. Thus, it follows from the
definitions of the functions $\widetilde{\mathcal S}$ and $\mathcal S$ that
\begin{align}
&\widetilde{\mathcal S}(\cos\zeta,n) =  \mathcal S(\cos\zeta,n), && \zeta \in \mathcal D_{-\frac\pi2+\theta_{\hat{x}}+i\infty,\zeta_o}, \notag \\
&\widetilde{\mathcal S}(\cos\zeta,n)=-\mathcal S(\cos\zeta,n), && \zeta \in
\mathcal{D}_{\zeta_o,\frac\pi2+\theta_{\hat{x}}-i\infty}. \label{eq:20}
\end{align}
From this, it is deduced that
\be\label{eq:7}
G^{(2)}_{\mathcal R}(x,y) = G^{(3)}_{\mathcal R}(x,y) + G_{\mathcal R, \zeta_o, 2}(x,y),
\en
where
\begin{align}\label{eq:19}
&G_{\mathcal R, \zeta_o, 2}(x,y) :=
\frac{i}{2\pi}\int_{\mathcal D_{\zeta_o,\frac\pi2+\theta_{\hat{x}}-i\infty}} \frac{2i\sin\zeta\mathcal S(\cos\zeta,n)}{n^2-1} {e^{ik_{+}\left(-\vert y \vert\cos(\zeta+\theta_{\hat y})+ \vert x \vert\cos(\zeta-\theta_{\hat x})\right)}} d\zeta.
\end{align}
Using (\ref{neq:4.2.3}) and (\ref{eq9}), we have that for $\zeta \in \{\zeta\in\Cb:\Ima(\zeta)< 0, \frac{\pi}2-(\theta_c-\theta_{\hat x}) \le \Rt(\zeta) \le \frac{\pi}2 + \theta_{\hat x}\}$,
\begin{align*}
&\left\vert\frac{2i\sin\zeta\mathcal S(\cos\zeta,n)}{n^2-1} {e^{ik_{+}\left(-\vert y \vert\cos(\zeta+\theta_{\hat y})
+ \vert x \vert\cos(\zeta-\theta_{\hat x})\right)}}\right\vert \\
&\le C \left[1+\left(\vert\sinh(\I(\zeta))\vert + \cosh(\I(\zeta))\right)^2\right]
 e^{k_+ \sinh(\I(\zeta)) \left(\vert x \vert\cos({\theta_{c}})-\vert y \vert\right)},
\end{align*}
where $C>0$ is a constant independent of $\zeta$. From this, (\ref{eq:19}) and the fact that
$\mathcal S(\cos\zeta,n)$ is analytic for $\zeta\in II_2$ with $II_2$ defined as in Case 1,
we can apply Cauchy integral theorem to obtain that: for $\vert x \vert> \vert y \vert/\cos(\theta_c)$,
\begin{align} \label{eq:9}
&G_{\mathcal R,\zeta_o, 2}(x,y) = \frac{i}{2\pi}\int_{\mathcal I_{\zeta_o,\frac{\pi}2-i\infty}}
\frac{2i\sin\zeta\mathcal S_{-}(\cos\zeta,n)}{n^2-1} {e^{ik_{+}\left(-\vert y \vert\cos(\zeta+\theta_{\hat y})
+ \vert x \vert\cos(\zeta-\theta_{\hat x})\right)}} d\zeta,
\end{align}
where the path $\mathcal I_{\zeta_o,\frac{\pi}2-i\infty}$ denotes the curve
$\{\zeta \in \mathcal I^+_0: \Rt(\zeta)\in \left(\Rt(\zeta_o), \frac\pi2\right)\}$
with the orientation from $\zeta_o$ to $\frac{\pi}2-i\infty$.
Thus it easily follows from (\ref{eq8}), (\ref{eq41}) and (\ref{eq:9}) that: for $\vert x \vert>\vert y \vert/\cos(\theta_c)$,
\begin{align}\label{eq:11}
G^{(4)}_{\mathcal R}(x,y) = G_{\mathcal R,\zeta_o, 1}(x,y) + G_{\mathcal R,\zeta_o, 2}(x,y).
\end{align}
This, together with (\ref{eq:10}) and (\ref{eq:7}), implies that the formula (\ref{neq:4.2.8}) holds.
The proof is thus complete.
\end{proof}

\begin{remark}\label{re6}{\rm {\color{hw}
As mentioned in Remark \ref{re1}, $ik_{+}\cos(\zeta - \theta_{\hat x})$ has
only one saddle point $\zeta=\theta_{\hat{x}}$ on the steepest descent path $\mathcal D$.
Note that $ik_{+}\cos(\zeta - \theta_{\hat x})$ appears in the exponential functions
in both \eqref{eq55} and \eqref{eq56}.
Thus, by the steepest descent method, the asymptotic expansions of $G^{(2)}_{\mathcal R}(x,y)$ and $G^{(3)}_{\mathcal R}(x,y)$ for large $|x|$ are expected to depend on the series expansions of
$\mathcal S(\cos\zeta,n)$ and $\widetilde{{\mathcal S}}(\cos\zeta,n)$, respectively, at the saddle point
$\zeta=\theta_{\hat{x}}$ (see \cite[Section 7.3]{B84}), where $\mathcal S(\cos\zeta,n)$ and
$\widetilde{{\mathcal S}}(\cos\zeta,n)$ appear in the amplitude functions in \eqref{eq55} and \eqref{eq56}, respectively.
However, since $\zeta=\theta_c$ is a branch point of $\mathcal S(\cos \zeta, n)$
and $\widetilde {\mathcal S}(\cos \zeta, n)$ (see Remark \ref{re2}), the modulus of
every derivative of ${\mathcal S}(\cos\zeta,n)$ and $\widetilde{{\mathcal S}}(\cos\zeta,n)$
at the saddle point $\zeta=\theta_{\hat{x}}$ will blow up when the angle $\theta_{\hat{x}}$ tends to $\theta_c$.
This leads to difficulties in the
investigation of the uniform far-field asymptotics of $G^{(2)}_{\mathcal R}(x,y)$ and $G^{(3)}_{\mathcal R}(x,y)$
for the angles $\theta_{\hat{x}}$ very close to the branch point $\theta_c$.
On the other hand, because of the presence of the branch cuts, we need to analyze not only the asymptotics
of the integrals $G^{(j)}_{\mathcal R}(x,y)$ $(j=1,2,3)$ over the steepest descent path $\mathcal{D}$
but also the asymptotics of the integral $G^{(4)}_{\mathcal R}(x,y)$ over the loop $\mathcal{L}_o$.
This leads to another difficulty in investigating the uniform far-field asymptotics of $G_{\mathcal{R}}(x,y)$,
since the loop $\mathcal{L}_o$ is around the branch cut $\mathcal{I}^+_0$ of ${\mathcal S}(\cos(\cdot),n)$
appearing in the amplitude function in \eqref{eq59}.
}}
\end{remark}

\begin{remark}\label{re7}{\rm{\color{hw}
The steepest descent path $\mathcal{D}$ was also used in \cite[formula (9)]{CH95} for deriving the uniform
far-field asymptotic expansion of the Green function for the Helmholtz equation in a half-plane with impedance
boundary condition. See \cite[Section 2]{CH95} for an equivalent definition of $\mathcal{D}$.
The integrand in \cite[formula (9)]{CH95} is analytic in the complex plane except for an infinite countable number of
poles (see the discussion following \cite[formula (7)]{CH95}). In our case, however, due to the analyticity of
${\mathcal S}(\cos(\cdot),n)$ and $\widetilde{{\mathcal S}}(\cos(\cdot),n)$ (see Remark \ref{re2}),
the integrands in \eqref{eq55} and \eqref{eq56} are analytic in the complex plane except for an infinite countable number
of branch cuts. Therefore, the asymptotic analysis for \cite[formula (9)]{CH95} does not work for
the study of the uniform far-field asymptotics of $G^{(2)}_{\mathcal R}(x,y)$ and $G^{(3)}_{\mathcal R}(x,y)$
for the angles $\theta_{\hat{x}}$ very close to the branch point $\theta_c$,
which is difficult as discussed in Remark \ref{re6}.}}
\end{remark}

{\color{hw}
To overcome the difficulties mentioned in Remark \ref{re6}, in what follows we will introduce some useful integral identities (see Lemma \ref{NLL1} below) and give a rigorous analysis of the singularities
of ${\mathcal S}(\cos(\cdot),n)$ and $\widetilde{{\mathcal S}}(\cos(\cdot),n)$ (see Lemmas \ref{Nle:1}, \ref{Le:3.2.1} and \ref{le:4} below).}

For the subsequent use, we need some special functions. Let $D_{\beta}(z)$ denote the parabolic
cylinder function and  let $\Gamma(z)$ denote the Gamma function. We refer to \cite{W_27} for
the definitions and properties of $D_{\beta}(z)$ and $\Gamma(z)$.
In particular, it is known that $\Gamma(z)$ is analytic except at the points $z=0,-1,-2,\ldots$
(see, e.g., Section 12.1 in \cite{W_27})
and $D_\beta(z)$ is analytic in $z\in\mathbb{C}$ for any $\beta\in\mathbb{R}$ (see, e.g.,
Section 16.5 in \cite{W_27}).
The following lemma presents some useful results relevant to $D_{\beta}(z)$ and $\Gamma(z)$,
which are given in \cite{bg_2}.

\begin{lemma}[see formulas (A.3.25) and (A.3.26) in \cite{bg_2}]\label{NLL1}
Assume $\beta\in\mathbb{R}$ with $\beta>-1$ and $\rho,b\in\mathbb{C}$ with $\I(\rho)>0$, $\I(b)\ne 0$.
Consider the integrals  $F_2(\rho,b,\beta):= \int^{+\infty}_{-\infty}(s-b)^{\beta}e^{i\rho s^2}ds$
and $F_3(\rho,b,\beta):= \int_{\gamma_2}(s-b)^{\beta}e^{i\rho s^2}ds $, where $\gamma_2$ is a loop
around the branch cut $\{z\in \Cb: \I(z)= \I(b), \Rt(z)\le \Rt(b)\}$ depicted in Figure \ref{f3.1}.
Then, we have
%
\begin{empheq}[left={F_2(\rho,b,\beta)= \empheqlbrace}]{align}
&e^{i\rho \frac{b^2}2}\sqrt{2\pi}{\left(\frac{1}{2\rho}\right)}^{\frac{\beta+1}2}e^{\frac{i\pi(3\beta+1)}4}D_{\beta}(\sqrt{2\rho}b e^{\frac{i\pi}4}),  & \I(b)<0, \label{neq:4.2.33} \\
&e^{i\rho \frac{b^2}2}\sqrt{2\pi}{\left(\frac{1}{2\rho}\right)}^{\frac{\beta+1}2}e^{\frac{i\pi(-\beta+1)}4}D_{\beta}(\sqrt{2\rho}b e^{\frac{-i3\pi}4}), &  \I (b)>0, \label{j:e11}
\end{empheq}
and
\begin{align} \label{eq:j.3}
{F_3(\rho,b,\beta)= \sgn{(\I(b))}e^{i\rho \frac{b^2}2}{\left(\frac{1}{2\rho}\right)}^{\frac{\beta+1}2}
e^{\frac{i\pi(\beta+3)}4}\frac{2\pi}{\Gamma(-\beta)}D_{-\beta-1}(\sqrt{2\rho}b e^{\frac{i3\pi}4})},
\end{align}
where $\sgn$ denotes the sign function.
\end{lemma}


\begin{figure}
\centering
\includegraphics[width=0.9\textwidth]{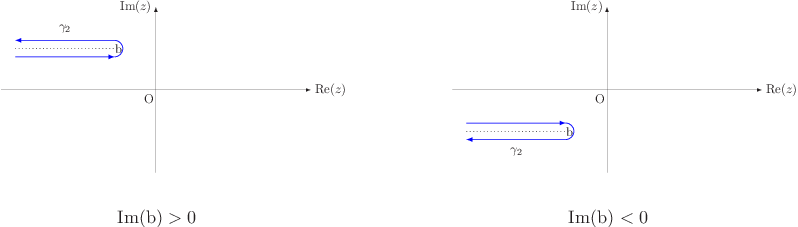}
\caption{The path $\gamma_2$ for $F_3(\rho,b,\beta)$.}\label{f3.1}
\end{figure}

\begin{remark}{\rm
There was a typo in \cite[formula (A.3.25)]{bg_2} for the integral $F_2(\rho,b,\beta)$ with $\I(b)<0$
and \cite[formula (A.3.25)]{bg_2} was given without a proof.
In formula (\ref{neq:4.2.33}), we give the correct expression for this typo.
For the proof of formula (\ref{neq:4.2.33}), see Appendix \ref{sec:a}.
}
\end{remark}

To proceed further, we need the following Lemmas \ref{Nle:1} and \ref{Le:3.2.1}.
The proofs of these two lemmas will be given in Appendices \ref{sec:b1} and \ref{sec:b2}, respectively.

\begin{lemma}\label{Nle:1}
Let $f(s): = a_0 + a_1 P(s) + a_2Q(s)$, $s\in\mathbb{C}$, where the constants
$a_0\ge 0, a_1>0, a_2\in\Rb$, and the functions $P(s)$ and $Q(s)$ are given by
\ben
P(s):=\sqrt{1-\frac{s^2}{2k_{+}i}},\quad  Q(s):=\frac{se^{-i\frac{\pi}4}}{\sqrt{2k_{+}}}.
\enn
Define $\mathsf w_0: = {a_1\sqrt{k_{+}}}/{\sqrt{a_1^2+a_2^2}}$.
Then the following statements hold.
\begin{enumerate}
\item\label{se3}
When $a_2\geq0 $, we have $\{f(s):s\in L_{-\mathsf w_0,\sqrt{k_{+}}}\}\cap\{s\in\Rb:s\le 0\}=\emptyset$.
\item\label{se4}
When $a_2<0$, we have $\{f(s):s\in L_{ -\sqrt{k_{+}},\mathsf w_0}\}\cap \{s\in \Rb: s\le 0\}=\emptyset$.
\end{enumerate}
\end{lemma}

\begin{lemma} \label{Le:3.2.1}
Let $A(z)$ and $B(z)$ be two analytic functions in a strip  $ L_{\mathsf  w} $ with
$ \mathsf w>0 $. If $ A^2(z) = B^2(z) $ for $ z\in  L_{\mathsf w} $ and $A(0)= B(0)\ne 0$.
Then we have $A(z)= B(z)$ for $z\in  L_{\mathsf w}$.
\end{lemma}

Let $\arcsin(z):=\int^z_0 1/\sqrt{1-t^2}dt$ denote the principal value of the inverse of the
sine function in the complex plane, where the path of integration must not cross the branch cuts
$\{z\in\Rb:\vert z\vert\ge 1\}$ (see, e.g., Section 4.23 in \cite{OLBC10}). It is known that $\arcsin (z)$
is analytic in the domain $\Cb\backslash\{z\in\Rb: \vert z\vert\ge 1\}$.
Then, employing Lemmas \ref{Nle:1} and \ref{Le:3.2.1} we get the following lemma.

\begin{lemma} \label{le:4}
Let $P (s)$ and $Q (s)$ be the functions defined in Lemma \ref{Nle:1}.
Assume $\theta_{\hat x} \in (0,\pi/2)$ and define
${\zeta}(s) := 2 {\arcsin} (Q(s)) + \theta_{\hat x}$. Then the following statements hold.
\begin{enumerate}
  \item \label{e1}
  $\zeta(s)$ is analytic in the strip $L_{\sqrt{k_{+}}}$ and satisfies
  \begin{align} \label{neq:4.2.10}
 P(s)=\cos{\left(\frac{\zeta(s)-\theta_{\hat x}}{2}\right)},\;\;
 Q(s)=\sin{\left(\frac{\zeta(s)-\theta_{\hat x}}{2}\right)}\quad\textrm{for}~s\in L_{\sqrt{k_+}}.
 \end{align}
Moreover, ${\zeta}(s)$ is a bijection from $\mathbb R$ onto $\mathcal D$
and ${\zeta}(s)$ travels from $-\pi/2+\theta_{\hat x}+i\infty$ to $\pi/2 +\theta_{\hat x}-i\infty$
on the path $\mathcal D$ as $s$ travels from $-\infty$ to $+\infty$ along the real axis.

\item
\label{e2} In the case $\theta_{\hat x}\in(\theta_c,\pi/2)$, we have
\begin{align}
&\sqrt {2/k_{+}}e^{-i\frac{\pi}{4}}H_{\theta_c}(s) H_{\pi-\theta_c}(s)\sqrt{s-s_b^{*}}\sqrt{s-s_b}
= \mathcal S(\cos(\zeta(s)),n)\quad\textrm{for}~ s\in\Rb\label{neq:4.2.45}
\end{align}
and in the case $\theta_{\hat x}\in(0,\theta_c)$, we have
\begin{align}
&-\sqrt {2/k_{+}}e^{-i\frac{\pi}{4}}H_{\theta_c}(s) H_{\pi-\theta_c}(s)\sqrt{s-s_b^{*}}\sqrt{s-s_b}
= \widetilde{\mathcal S}(\cos(\zeta(s)),n) \notag \\
&\qquad\qquad\qquad\qquad\qquad\qquad\qquad\qquad\qquad\qquad \quad \textrm{for}~ s\in L_{\vert\I(s_b)\vert}, \label{eq:1}
\end{align}
where
\ben
s_b:=\sqrt{2k_{+}}e^{i\frac{\pi}4}\sin {\left(\frac{\theta_c-\theta_{\hat x}}{2}\right)},\;
s_b^{*}:=\sqrt{2k_{+}}e^{i\frac{\pi}4}\sin{\left(\frac{\pi-\theta_c-\theta_{\hat x}}{2}\right)}
\enn
and $H_{\theta}(s):={\sqrt{F^{(1)}_{\theta}(s)}\sqrt{F^{(2)}_{\theta}(s)}}\Big/{\sqrt{F^{(3)}_{\theta}(s)}}$
for $\theta \in (0,\pi)$ with $F^{(j)}_{\theta}(s)$ $(j=1,2,3)$ given by
\begin{align*}
& F^{(1)}_{\theta}(s) :=  1 + P(s)\cos\left(\frac{\theta-\theta_{\hat x}}{2}\right)
+ Q(s)\sin\left(\frac{\theta-\theta_{\hat x}}{2}\right),  \\
& F^{(2)}_{\theta}(s):=P(s)\sin\left(\frac{\theta+\theta_{\hat x}}{2}\right)
+Q (s)\cos\left(\frac{\theta+\theta_{\hat x}}{2}\right),  \\
& F^{(3)}_{\theta}(s) :=  \cos\left(\frac{\theta-\theta_{\hat x}}{2}\right)+P(s).
\end{align*}
Moreover, $H_{\theta_c}(s)$ and $H_{\pi-\theta_c}(s)$ are analytic in the strip
$L_{\sigma_{\theta_{\hat x}}}$ with $\sigma_{\theta_{\hat x}}:=\sqrt{k_+}
\min(\sin\left(({\theta_c+\theta_{\hat x}})/2\right),\cos{\left(({\theta_c-\theta_{\hat x}})/2\right)})$.

\item\label{e3}
For any $\theta_{\hat{x}}\in(0,\pi/2)$,
the function $H_{\theta_c}(s) H_{\pi-\theta_c}(s)\sqrt{s-s_b^{*}}$ is analytic in
$s\in L_{\sigma^{(1)}_{\theta_{\hat{x}}}}$ with
$\sigma^{(1)}_{\theta_{\hat{x}}}:=\min(\sigma_{\theta_{\hat{x}}},
\sqrt{k_+}\sin((\pi-\theta_c-\theta_{\hat{x}})/2))$
and the function $H_{\theta_c}(s) H_{\pi-\theta_c}(s)\sqrt{s-s_b^{*}}\sqrt{s-s_b}$
is analytic in $s\in L_{\sigma^{(2)}_{\theta_{\hat{x}}}}$ with
$\sigma^{(2)}_{\theta_{\hat{x}}}:=\sqrt{k_+}\vert\sin((\theta_c-\theta_{\hat{x}})/2)\vert$.
Moreover, we have
\be\label{eq31}
0<\sigma^{(1)}_{\min}\leq\sigma^{(1)}_{\theta_{\hat{x}}}\leq\sigma^{(1)}_{\max}<\sqrt{k_+}
\en
with $\sigma^{(1)}_{\min}:=\min(\sqrt{k_+}\sin(\theta_c/2),\sqrt{k_+}\sin(\pi/4-\theta_c/2))$
and $\sigma^{(1)}_{\max}:=\sqrt{k_+/2}$.
\end{enumerate}
\end{lemma}

\begin{proof}
(\ref{e1}) From Section 4.15 in \cite{OLBC10}, it can be seen that $\sin (z)$ is a conformal
mapping from the domain $\{z\in\Cb:-\pi/2<\Rt(z)<\pi/2\}$ to the domain $\Cb\backslash\{z\in\Rb:\vert z\vert \ge 1\}$
and its inverse is given by $\arcsin(z)$. Further, it is clear that $\Rt(\sqrt{z})>0$ for $z\in\Cb_0$
and $\Rt(\cos(z))>0$ for $z\in\{z\in\Cb:-\pi/2<\Rt(z)<\pi/2\}$. These facts imply that
$\zeta(s)$ is analytic in $ L_{\sqrt{k_{+}}}$ and the formula (\ref{neq:4.2.10}) holds true.
Moreover, from the formula (\ref{eq:12}), it follows that the mapping
$F(z):= \sqrt{2k_{+}}e^{i\frac{\pi}4}\sin{\left({(z-\theta_{\hat x})}/{2}\right)}$
is a bijection from $\mathcal D$ to $\mathbb{R}$
and $F(z)$ travels from $-\infty$ to $+\infty$ along the real axis as
$z$ travels from $-\pi/2 + \theta_{\hat x}+ i\infty $ to $\pi/2+\theta_{\hat x}-i\infty$
on the path $\mathcal D$. Hence, the proof of statement (\ref{e1}) is now completed
from the fact that $s= \sqrt{2k_{+}}e^{i\frac{\pi}4}\sin{\left({(\zeta(s)-\theta_{\hat x})}/{2}\right)}$
for $s\in\mathbb{R}$ (see the second formula in (\ref{neq:4.2.10})).

({\ref{e2}}) From Lemma \ref{Nle:1}, we have $F^{(3)}_{\theta_c}(s)\neq0$ for $s\in L_{\sqrt{k_{+}}}$.
Thus it follows from (\ref{neq:4.2.10}) that for $s\in L_{\sqrt{k_{+}}}$,
\begin{align*}
&\left(\sqrt{2k_{+}}e^{i{\pi}/{4}}\right)^{-1}\left(\cos \zeta(s)- \cos\theta_c\right)\\
&= \frac{4\sin\left(\frac{\theta_c- \zeta(s)}{4}\right)\cos\left(\frac{\theta_c
-2\theta_{\hat x}+ \zeta(s)}{4}\right)\cos^2\left(\frac{\theta_c- \zeta(s)}{4}\right)
\sin\left(\frac{\zeta(s) + \theta_{c}}{2}\right)}{\sqrt{2k_{+}}e^{i{\pi}/{4}}
\cos\left(\frac{\theta_c-2\theta_{\hat x}+\zeta(s)}{4}\right)\cos\left(\frac{\theta_c-\zeta(s)}{4}\right)}\\
& = \left(\sin \left(\frac{\zeta(s)- \theta_{\hat x}}{2}\right)
- \sin \left(\frac{\theta_c- \theta_{\hat x}}{2}\right)\right)\frac{\sqrt 2 e^{i\frac{3\pi}4}\left(1+\cos\left(\frac{\theta_{c}- \zeta(s)}{2}\right)\right)
\sin\left(\frac{\theta_c+ \zeta(s) }{2}\right)}{\sqrt{k_{+}}
\left(\cos\left(\frac{\theta_c- \theta_{\hat x}}{2}\right)
+\cos\left(\frac{\zeta(s)- \theta_{\hat x}}{2}\right)\right)} \\
&= i(s-s_b)\frac{F^{(1)}_{\theta_c}(s)F^{(2)}_{\theta_c}(s)}{F^{(3)}_{\theta_c}(s)}
k_{+}^{-1}.
\end{align*}
Similarly as above, for $s\in L_{\sqrt{k_{+}}}$, we have $F^{(3)}_{\pi-\theta_c}(s)\neq0$ and
\begin{align*}
\cos\zeta(s) + \cos\theta_c = (s-s^{*}_b)\frac{F^{(1)}_{\pi-\theta_c}(s)F^{(2)}_{\pi-\theta_c}(s)}
{F^{(3)}_{\pi-\theta_c}(s)}\frac{\sqrt2 e^{\frac{i3\pi}4}}{\sqrt{k_{+}}}.
\end{align*}
Using Lemma \ref{Nle:1} again, we obtain that $\sqrt{F^{(j)}_{\theta_c}(s)}$ and
$\sqrt{F^{(j)}_{\pi-\theta_c}(s)}$ $(j =1,2,3)$ are analytic functions in the strip
$L_{\sigma_{\theta_{\hat x}}}$. Thus, it follows that $H_{\theta_c}(s)$ and
$H_{\pi-\theta_c}(s)$ are analytic in $L_{\sigma_{\theta_{\hat x}}}$.

Now we prove that the formula (\ref{neq:4.2.45}) holds for the case $\theta_{\hat x}\in(\theta_c,\pi/2)$.
From Lemma \ref{Nle:1} and (\ref{neq:4.2.10}), it follows that $\zeta'(s)$ is
analytic in $s \in  L_{\sqrt{k_+}}$ and is given by
\be\label{eq10}
\zeta'(s)= {\sqrt {2/k_+}e^{-i\pi/4}}/P(s),~s\in L_{\sqrt{k_+}}.
\en
This, together with the definition of $P(s)$, implies that
$\vert\zeta'(s)\vert$ is uniformly bounded for $s \in  L_{\sqrt{k_+}/2}$.
Further, for any fixed $\theta_{\hat{x}}\in(\theta_c,\pi/2)$, it can be seen that
$\mathcal D=\{\zeta(s):s\in \mathbb{R}\}\subset II$ and the distance between
$\mathcal D$ and the boundary of $II$ has a positive lower bound, where $II$ is defined as in
the proof of Lemma \ref{Le:n4.2.1}.
Thus, it easily follows from Taylor formula and Cauchy inequality (see,
e.g., \cite[Corollary 4.3]{Stein}) that, for any fixed $\theta_{\hat{x}}\in(\theta_c,\pi/2)$,
there exists small enough $\varepsilon_{\theta_{\hat x}}>0$
such that
$\{\zeta(s):s\in L_{\varepsilon_{\theta_{\hat x}}}\}\subset II$.
By this and Remark \ref{re2}, we obtain that $\mathcal S(\cos(\zeta(s)),n)$ with
$\theta_{\hat{x}}\in(\theta_c,\pi/2)$ is analytic in $s\in L_{\varepsilon_{\theta_{\hat x}}}$.
Clearly, $\sqrt{s-s_b}$ and $\sqrt{s-s_b^{*}}$ are analytic in $s\in L_{\vert\I(s_b)\vert}$ and
in $s\in L_{\vert\I(s^{*}_b)\vert}$, respectively. Moreover, by a straightforward calculation, we have
\begin{align}
&\frac{\sqrt{2}}{\sqrt{k_+}e^{i\frac{\pi}{4}}}\sqrt{-s_b}\sqrt{-s_b^{*}}H_{\theta_c}(0)H_{\pi-\theta_c}(0)
\notag\\
&=\mathcal S(\cos(\zeta(0)),n)=-i\sqrt{n^2-\cos^2\theta_{\hat x}}
\ne 0 \quad \textrm{for}~\theta_{\hat x} \in (\theta_c, \pi/2). \label{neq:4.2.40}
\end{align}
Hence, using the above arguments and applying Lemma \ref{Le:3.2.1}, we obtain that the
formula (\ref{neq:4.2.45}) holds for the case $\theta_{\hat x}\in(\theta_c,\pi/2)$.

Next, we prove that the formula (\ref{eq:1}) holds for the case $\theta_{\hat x}\in(0,\theta_c)$.
Let $\theta_{\hat x}\in(0,\theta_c)$ in the rest proof of statement (\ref{e2}).
Denote by $q_1(z):=\sin\left(({z-\theta_{\hat x}})/{2}\right)$  and
$q_2(z):=\cos\left(({z-\theta_{\hat x}})/{2}\right)$. A straightforward calculation gives that
\begin{align}\label{eq:16}
&\Ima \bigg(\sqrt{2k_{+}}e^{i\frac\pi 4}q_1(z)\bigg)=\notag\\
&\qquad \qquad
\sqrt{k_+}\left(q_1(\Rt(z))
\cosh\left(\frac{\Ima(z)}{2}\right) + q_2(\Rt(z))\sinh\left(\frac{\Ima(z)}{2}\right)\right).
\end{align}
It is easily seen that $q_1(\Rt(z))\cosh\left({\Ima(z)}/{2}\right)$
and $q_2(\Rt(z))\sinh\left({\Ima(z)}/{2}\right)$ are both nonnegative or both
nonpositive in $z\in{\mathcal A_2}\cup{\mathcal A_3}$, where $\mathcal A_2$
and $\mathcal A_3$ are defined in Remark \ref{re2}. This, together with  (\ref{eq:16}), implies that
\begin{align}
&\bigg\vert\Ima \bigg(\sqrt{2k_{+}}e^{i\frac\pi 4}q_1(z)\bigg)\bigg\vert \ge \sqrt{k_+}\left\vert q_1(\Rt(z))\cosh\left(\frac{\Ima(z)}{2}\right)\right\vert \notag \\
&\ge\sqrt{k_+} \min_{z \in {\mathcal {\widetilde I}}^{+}_{0}}\left\vert q_1(\Rt(z))\right\vert
\min_{z \in {\mathcal {\widetilde I}}^{+}_{0}}\left\vert \cosh\left(\frac{\Ima(z)}{2}\right) \right\vert
\ge \vert\I(s_b)\vert\quad\textrm{for}~z \in \mathcal A_2\cup \mathcal A_3.\label{eq:17}
\end{align}
Thus, combining (\ref{eq:17}) and the second formula in (\ref{neq:4.2.10}), we deduce that
\ben
\left\{\zeta(s):s\in L_{\vert\I(s_b)\vert}\right\}\cap\left(\mathcal A_2\cup \mathcal A_3\right)=\emptyset.
\enn
Hence, it follows from Remark \ref{re2} that $\widetilde{\mathcal S}(\cos(\zeta(s)),n)$
is analytic in $s\in L_{\vert\I(s_b)\vert}$. Further, it is easy to verify that
\begin{align*}
&-\sqrt{2/k_{+}}e^{-i\frac{\pi}{4}}\sqrt{-s_b}\sqrt{-s_b^{*}}H_{\theta_c}(0)H_{\pi-\theta_c}(0)
= \widetilde{\mathcal S}(\cos(\zeta(0)),n)=\sqrt{\cos^2\theta_{\hat x}-n^2}
\ne 0
\end{align*}
for $\theta_{\hat x} \in (0,\theta_c)$. Therefore, employing Lemma \ref{Le:3.2.1} again, we obtain that the formula (\ref{eq:1})
holds for the case $\theta_{\hat x} \in(0,\theta_c)$.

(\ref{e3}) For $\theta_{\hat{x}}\in(0,\pi/2)$, it is clear
that $\sigma^{(2)}_{\theta_{\hat{x}}}\leq \sigma^{(1)}_{\theta_{\hat{x}}}$
and $\sigma^{(1)}_{\theta_{\hat{x}}}=\sqrt{k_+}\min(\sin((\theta_c+\theta_{\hat{x}})/2),
\cos((\theta_c+\theta_{\hat{x}})/2))$. Thus by statement (\ref{e2}), the analyticity of $\sqrt{s-s_b}$
and $\sqrt{s-s^*_b}$, and a direct calculation, we can obtain that statement (\ref{e3}) holds.
\end{proof}

With the aid of Lemmas  \ref{NLL1} and \ref{le:4},
we are now ready to study the uniform far-field asymptotic estimates of the functions
$G^{(j)}_{\mathcal R}(x,y)$ $(j=1,2,3,4)$ defined in Lemma \ref{Le:n4.2.1}.
{\color{hw}
We mention that in the proofs of Lemmas \ref{NLe:3.2.4} and \ref{Le:n4.2.2} below,
an important role is played by the technically involved singularity analysis of
the relevant integrals
for $G_{\mathcal{R}}(x,y)$ when the angle $\theta_{\hat{x}}$ is very close to the branch point $\theta_c$.
}

\begin{lemma} \label{NLe:3.2.4}
Assume that $k_+>k_-$ and let $R_0>0$ be an arbitrary fixed number.
Suppose that $y = \vert y \vert(\cos \theta_{\hat y}, \sin \theta_{\hat y})\in B^+_{R_0}$ with
$\theta_{\hat y}\in (0,\pi)$ and $x=\vert x \vert(\cos\theta_{\hat x}, \sin \theta_{\hat x})\in\Rb^2_+$
with  $\theta_{\hat x}\in \left[\left.\theta_c,{\pi}/2\right)\right.$, then
we have the asymptotic behavior
\begin{align}\label{j:e25}
G^{(1)}_{\mathcal R}(x,y) = 
\frac{e^{ik_{+}\vert x \vert}}{\sqrt{\vert x \vert}}\frac{e^{i\frac{\pi}4}}
{\sqrt{8\pi k_{+}}}\left(\frac{2\cos^2\theta_{\hat x}-1-n^2}{n^2-1}\right)
e^{-i k_{+}\vert y \vert\cos(\theta_{\hat x}{+}\theta_{\hat y})}+ G^{(1)}_{\mathcal R,Res}(x,y)
\end{align}
and
\begin{align}\label{eq:3.2.11}
 G^{(2)}_{\mathcal R}(x,y) =
 \frac{e^{ik_{+}\vert x \vert}}{\sqrt{\vert x \vert}}\frac{e^{i\frac{\pi}4}}
{\sqrt{8\pi k_{+}}}\left(\frac{2i\sin\theta_{\hat x}\mathcal S(\cos\theta_{\hat x},n)}{n^2-1}\right)
e^{-ik_{+}\vert y \vert\cos(\theta_{\hat x}+\theta_{\hat y})} + G^{(2)}_{\mathcal R,Res}(x,y),
\end{align}
where $G^{(1)}_{\mathcal R,Res}(x,y)$ and $G^{(2)}_{\mathcal R,Res}(x,y)$ satisfy that
\begin{align} \label{neq:4.2.32}
&\vert G^{(1)}_{\mathcal R,Res}(x,y)\vert\le {C_{R_0}}{\vert x \vert^{- 3/2}},&&\vert x \vert\rightarrow+\infty,\\ \label{neq:4.2.23}
&\vert G^{(2)}_{\mathcal R,Res}(x,y)\vert\le {C_{R_0}}{\vert x \vert^{- 3/4}}, &&\vert x \vert\rightarrow+\infty,
\end{align}
uniformly for all $\theta_{\hat x}\in \left[\left.\theta_c,{\pi}/2\right)\right.$
and $y\in B^+_{R_0}$, and that
\begin{align}
&\vert G^{(2)}_{\mathcal R,Res}(x,y)\vert \le {C_{R_0}}{{{\left\vert
\theta_c-\theta_{\hat x}
\right\vert}^{-\frac 32}}\vert x \vert^{-\frac 32}}, &&\vert x \vert\rightarrow+\infty, \label{neq:4.2.22}
\end{align}
uniformly for all $\theta_{\hat x}\in \left(\theta_c,{\pi}/2\right)$  and $y\in B^+_{R_0}$.
Here, the constant $C_{R_0}>0$ is independent of $x$ and $y$ but dependent of $R_0$.
\end{lemma}

\begin{proof}
Let $\vert x \vert$ be sufficiently large throughout the proof.
As mentioned in Remark \ref{re2}, the steepest descent path $\mathcal{D}$ will not cross
the branch cuts of $\mathcal S(\cos \zeta, n)$
for the case $\theta_{\hat{x}}\in(\theta_c,\pi/2)$, and will cross the
branch cuts of $\mathcal S(\cos \zeta, n)$
with only one intersection point $\theta_c$ for the case $\theta_{\hat{x}}= \theta_c$.
Consequently, we will apply the steepest descent method (see, e.g., \cite{B84,bg_2,CH95})
to rewrite $G^{(j)}_{\mathcal R}(x,y)$ $(j=1,2)$ as integrals over the real axis.
By statement (\ref{e1}) of Lemma \ref{le:4}, we introduce the change of variable $\zeta = \zeta (s)$
to rewrite $G^{(j)}_{\mathcal R}(x,y)$ $(j = 1,2)$ as
\begin{align}\label{eq16}
G^{(j)}_{\mathcal R}(x,y)= \frac{ie^{i k_{+}\vert x \vert}}{4\pi}\int^{+\infty}_{-\infty} v_j(s)F(s)\frac{d\zeta(s)}{ds}e^{-\vert x \vert s^2}ds,
\end{align}
where $\zeta(s)$ is defined as in Lemma \ref{le:4} and
\begin{align}
&v_1(s):=\frac{\cos(2\zeta(s))-n^2}{n^2-1},  ~v_2(s)
:=\frac{2i\sin\zeta(s)\mathcal S(\cos(\zeta(s)),n)}{n^2-1},\notag\\
&F(s):= e^{-i k_{+}\vert y \vert\cos(\zeta(s)+\theta_{\hat y})}.\label{eq:3.2.3}
\end{align}

Next, the proof is divided into two steps. The first step is to estimate
$G^{(2)}_{\mathcal R}(x,y)$ and the second step is to estimate $G^{(1)}_{\mathcal R}(x,y)$.

\textbf{Step 1}: Estimates of $G^{(2)}_{\mathcal R}(x,y)$.
From Lemma \ref{le:4} and formula (\ref{eq16}), it is easily seen that
$G^{(2)}_{\mathcal R}(x,y)$ can be rewritten as
\be\label{Ne5}
G^{(2)}_{\mathcal R}(x,y)=\frac{ie^{ik_{+}\vert x \vert}}{4\pi}\int^{+\infty}_{-\infty}\sqrt{s-s_b}g(s)e^{-\vert x \vert s^2}ds,
\en
where $g(s)$ is given by
\begin{align} \label{eq:18}
g(s):= \sqrt {2/k_{+}}e^{-i\frac{\pi}{4}}H_{\theta_c}(s) H_{\pi-\theta_c}(s)\sqrt{s-s_b^{*}}
F(s)\frac{d\zeta(s)}{ds} \frac{2i\sin\zeta(s)}{n^2-1}
\end{align}
and is an analytic function in the strip $L_{\sigma^{(1)}_{\theta_{\hat{x}}}}$ with
$\sigma^{(1)}_{\theta_{\hat{x}}}$ defined in Lemma \ref{le:4}.
For $s\in L_{\sqrt{k_+}}$, with the aid of (\ref{neq:4.2.10}), we have
\ben
\sin(\zeta(s)-\theta_{\hat x})=2P(s)Q(s),\;\;\cos(\zeta(s)-\theta_{\hat x})=2P^2(s)-1,
\enn
which implies that
\begin{align*}
&\sin(\zeta(s)) =  2P(s)Q(s) \cos\theta_{\hat x} + (2P^2(s)-1) \sin\theta_{\hat x}, \\
&\cos(\zeta(s)+\theta_y) = {(2P^2(s)-1)\cos(\theta_{\hat x}+ \theta_{\hat y})
- 2P(s)Q(s) \sin(\theta_{\hat x}+ \theta_{\hat y})}.
\end{align*}
Note that $\Rt(P(s))>0$ in $s\in L_{\sqrt{k_+}}$,
$\min\left[\cos\left({(\theta_c-\theta_{\hat x})}/{2}\right),
\cos\left({(\pi-\theta_c-\theta_{\hat x})}/{2}\right)
\right]\geq\sin(\theta_c/2)$
and $\sigma^{(1)}_{\theta_{\hat{x}}}\leq\sigma^{(1)}_{\max}<\sqrt{k_+}$ (see statement (\ref{e3}) of Lemma \ref{le:4}).
Thus we have
$\vert F^{(3)}_{\theta_c}(s)\vert\geq C(1+\vert s\vert )$,
$\vert F^{(3)}_{\pi-\theta_c}(s)\vert \geq C(1+\vert s\vert )$ and
$C(1+\vert s\vert )\le\vert P(s)\vert \le 2(1+\vert s\vert /\sqrt{2k_{+}})$
in $s\in  L_{\sigma^{(1)}_{\theta_{\hat{x}}}}$ for some constant $C>0$.
Hence, by using the formula (\ref{eq10}), we have: for any $ \vert y \vert < R_0 $,
\be\label{Ne3}
\vert g(s)\vert \le C_{R_0}{(1+ \vert s \vert)}^{{5}/{2}}e^{\widetilde{C}_{R_0}{\vert s\vert}^2},\quad
s\in  L_{\sigma^{(1)}_{\theta_{\hat{x}}}},
\en
where the positive constants $C_{R_0}$ and $\widetilde{C}_{R_0}$ are independent of $\theta_{\hat x}$.
Then it follows from  Cauchy inequality (see, e.g., \cite[Corollary 4.3]{Stein}), (\ref{eq31})
and the estimate (\ref{Ne3}) that for $n=0,1,2,\ldots$,
\begin{align}\label{Ne16}
\left\vert\frac{d^n g}{d s^n}(s)\right\vert&\le \frac{C_{R_0} n!{(1+\sigma^{(1)}_{\theta_{\hat{x}}}
+ \vert s\vert )}^{5/2}e^{\widetilde{C}_{R_0}(\sigma^{(1)}_{\theta_{\hat{x}}}+\vert s\vert )^2}}
{(\sigma^{(1)}_{\theta_{\hat{x}}})^{n}}\notag\\
&\leq \frac{C_{R_0} n!{(1+ \vert s\vert )}^{5/2}
e^{\widetilde{C}_{R_0}\vert s\vert^2}}{(\sigma^{(1)}_{\theta_{\hat{x}}})^{n}},
\qquad s \in\Rb.
\end{align}

The rest proof of this step is divided into the following two parts.

\textbf{Part I}: We prove that $G^{(2)}_{\mathcal R}(x,y)$ has the form (\ref{eq:3.2.11})
with $G^{(2)}_{\mathcal R,Res}(x,y)$ satisfying (\ref{neq:4.2.22}) uniformly for all
$\theta_{\hat{x}}\in(\theta_c,\pi/2)$ and $y\in B^+_{R_0}$.

Assume $\theta_{\hat{x}}\in\left(\theta_c,\pi/2\right)$.
Let $G(s):=\sqrt{s-s_b}g(s)$ and choose $\sigma=\sigma^{(2)}_{\theta_{\hat{x}}}$ with
$\sigma^{(2)}_{\theta_{\hat{x}}}$ defined in Lemma \ref{le:4}. By statement (\ref{e3}) of Lemma \ref{le:4},
$G(s)$ is an analytic function in the strip $L_\sigma$. Define the function $J(s):= G(s)-G(0)-G{'}(0)s$.
Then it easily follows from (\ref{neq:4.2.40}) and (\ref{Ne5}) that
\begin{align}
&G^{(2)}_{\mathcal R}(x,y)= \frac{ie^{ik_{+}\vert x \vert}}{4\pi}\left(\int^{+\infty}_{-\infty}
G(0)e^{-\vert x \vert s^2}ds + \int^{+\infty}_{-\infty} J(s)e^{-\vert x \vert s^2}ds\right)\nonumber\\
&=\frac{e^{ik_{+}\vert x \vert}}{\sqrt{\vert x \vert}}\frac{e^{i\frac{\pi}4}}{\sqrt{8\pi k_{+}}}
\frac{2i\mathcal S(\cos\theta_{\hat x},n)}{n^2-1}\sin\theta_{\hat x} e^{-i k_{+}\vert y \vert
\cos(\theta_{\hat x}+\theta_{\hat y})}+
\frac{ie^{ik_{+}\vert x \vert}}{4\pi} \left(I_1+I_2\right),\label{eq11}
\end{align}
where $I_1$ and $I_2$ are given by
\ben
I_1:=\int_{-\sqrt{\sigma}}^{\sqrt{\sigma}}J(s)e^{-\vert x \vert s^2}ds,\quad
I_2:=\int_{\sqrt \sigma}^{+\infty}\left(J(s)+J(-s)\right)e^{-\vert x \vert s^2}ds.
\enn

Next, we estimate $I_1$. By mean-value theorem, we have: for any $s\in\Rb$, there exist
$\alpha_1,\alpha_2\in[0,1]$ such that
\ben
\Rt(J(s))=\Rt\bigg(\sqrt{\alpha_1 s-s_b}g''(\alpha_1 s)+ \frac{g{'}(\alpha_1 s)}
{\sqrt{\alpha_1 s-s_b}}-\frac{g(\alpha_1 s)}{4(\alpha_1 s -s_b)^{\frac 32}}\bigg)s^2
\enn
and
\ben
\I(J(s))=\I\bigg(\sqrt{\alpha_2 s-s_b}g{''}(\alpha_2 s)+ \frac{g{'}(\alpha_2 s)}
{\sqrt{\alpha_2 s-s_b}}-\frac{g(\alpha_2 s)}{4(\alpha_2 s -s_b)^{\frac 32}}\bigg)s^2.
\enn
Here, we note that $\sqrt{\alpha_j s-s_b}\neq 0$ $(j=1,2)$ due to $\I (s_b)<0$.
Further, it is easy to verify that $\vert s-s_b \vert \ge \sigma$ for $s \in \Rb$ and that $\sigma<\sqrt{k_+}$.
These, together with (\ref{eq31}), (\ref{Ne3}) and (\ref{Ne16}), yield that
\ben
\vert J(s) \vert \le {C_{R_0}}\left(1+{{\sigma}^{-1/2}}+{{\sigma}^{-3/2}}\right) s^2
\le {C_{R_0}}{{\sigma}^{-3/2}} s^2,\; s\in(-\sqrt\sigma,\sqrt\sigma).
\enn
Hence, we get
\be \label{neq:4.2.21}
\vert I_1\vert \le  \frac{C_{R_0}}{{\sigma}^{3/2}} \int^{+\infty}_{-\infty} s^2e^{-\vert x \vert s^2}ds
= \frac{C_{R_0}{\Gamma(\frac 3 2)}}{{{2}\sigma}^{3/2}\vert x \vert^{\frac 32}}.
\en

For the estimate of $I_2$, we introduce the following function
\ben
\Psi(t):= \int^{t}_{\sqrt{\sigma}}(J(s)+J(-s))e^{-r_0s^2}ds,
\enn
where $r_0$ is a fixed positive constant satisfying $r_0> 2\widetilde{C}_{R_0}$ with
$\widetilde{C}_{R_0}$ given in (\ref{Ne3}). By (\ref{Ne3}), it is deduced that
\begin{align*}
\sup_{t>\sqrt{\sigma}}\vert\Psi(t)\vert &\le 2\int^{+\infty}_{-\infty}\left(\vert G(s)\vert +\vert G(0)\vert \right)e^{-r_0s^2}ds\\
&\leq\int^{+\infty}_{-\infty}{C_{R_0}}{(1+\vert s\vert )}^{3}e^{-(r_0-\widetilde{C}_{R_0})s^2}ds
\le {C_{R_0}}.
\end{align*}
From this it follows that for $\vert x \vert$ large enough,
\begin{align}
\vert I_2\vert &=\left\vert \int^{+\infty}_{\sqrt{\sigma}}{\Psi{'}(s)}e^{-(\vert x \vert-r_0)s^2}ds\right\vert
 = \left\vert 2\int^{+\infty}_{\sqrt{\sigma}}(\vert x \vert-r_0)s \Psi(s) e^{-(\vert x \vert-r_0)s^2}ds\right\vert \notag \\
& \le {C_{R_0}}\int^{+\infty}_{\sqrt{\sigma}}2(\vert x \vert-r_0)s e^{-(\vert x \vert-r_0)s^2}ds
\leq {C_{R_0}}e^{{-(\vert x \vert-r_0)}{\sigma}}\leq {C_{R_0}}e^{-\vert x \vert{\sigma}}.\label{neq:4.2.20}
\end{align}
Thus, in terms of (\ref{eq:3.2.11}) and (\ref{eq11}), we can apply (\ref{neq:4.2.21})
and (\ref{neq:4.2.20}) to obtain that
\begin{align}\label{eq32}
\left\vert G^{(2)}_{\mathcal R,Res}(x,y)\right\vert &=\left\vert \left(I_1+I_2\right)/(4\pi)\right\vert
\leq C_{R_0}\left((\sigma\vert x \vert)^{-3/2}+e^{-\vert x \vert\sigma}\right)\notag\\
&\leq
\frac{C_{R_0}}{{{\left\vert \sin\left({(\theta_c-\theta_{\hat x})}/2\right)\right\vert}^{\frac 32}}\vert x \vert^{\frac 32}}
\end{align}
for large enough $\vert x \vert$. By using this and (\ref{eq11}) it is easily obtained that
$G^{(2)}_{\mathcal R}(x,y)$ has the form (\ref{eq:3.2.11}) with $G^{(2)}_{\mathcal R,Res}(x,y)$
satisfying (\ref{neq:4.2.22}) uniformly for all $\theta_{\hat{x}}\in(\theta_c,\pi/2)$ and $y\in B^+_{R_0}$.

\textbf{Part II}: We prove that $G^{(2)}_{\mathcal R}(x,y)$ has the form (\ref{eq:3.2.11})
with $G^{(2)}_{\mathcal R,Res}(x,y)$ satisfying (\ref{neq:4.2.23}) uniformly for all
$\theta_{\hat x}\in \left[\left.\theta_c,{\pi}/2\right)\right.$ and $y\in B^+_{R_0}$.
To do this, we consider the following three cases.

\textbf{Case 1}: $\theta_{\hat{x}}\in(\theta_c, \pi/2)$ with
$\vert \sin\left(({\theta_c-\theta_{\hat x}})/2\right)\vert \geq  {(2\sqrt{k_{+} \vert x \vert})}^{-1}$.
The proof of this case can be easily obtained by using (\ref{eq11}) and (\ref{eq32}).

\textbf{Case 2}: $\theta_{\hat{x}}\in(\theta_c, \pi/2)$ with
$\vert \sin \left(({\theta_c-\theta_{\hat x}})/2\right)\vert <  {(2\sqrt{k_{+} \vert x \vert})}^{-1}$.
Define $g_1(s):= {\left[g(s)-g(s_b)-(g(s_b)-g(0))(s-s_b)/{s_b}\right]}/{\left[s(s-s_b)\right]}$.
Due to the analyticity of $g(s)$, it is easily seen that $g_1(s)$
 is analytic
in $s\in L_{\sigma^{(1)}_{\theta_{\hat{x}}}}$.
Then following the idea in \cite[Section 2]{BLE1966} and \cite[Section A.3.3]{bg_2},
we can employ (\ref{Ne16}) and integration by parts to rewrite (\ref{Ne5}) as
\begin{align}\label{eq35}
&G^{(2)}_{\mathcal R}(x,y)= \notag\\
&\frac{ie^{ik_{+}\vert x \vert}}{4\pi}\left\{\int^{+\infty}_{-\infty}
\left[g(s_b)\sqrt{s-s_b} + \frac{g(s_b)-g(0)}{s_b}(s-s_b)^{\frac 32}\right ]e^{-\vert x \vert s^2}ds
+ G^{(2)}_{\mathcal R,1}(x,y)\right\},
\end{align}
where
\ben
G^{(2)}_{\mathcal R,1}(x,y):=\frac{1}{2\vert x \vert}\int^{+\infty}_{-\infty}\sqrt{s-s_b}
\left((s-s_b)g'_1(s)+\frac 32 g_1(s)\right)e^{-\vert x \vert s^2}ds.
\enn
It can be seen that $g(s_b)=h_c(\theta_{\hat{x}})$
and $\I(s_b) < 0$, where the function $h_c(\theta)$ is defined by
\begin{align}\label{eq14}
h_c(\theta):=\frac{{2}^{\frac{9}{4}}e^{{5\pi i}/{8}}}{k^{\frac{3}{4}}_{+}
\left[\cos\left(\frac{\theta_c-\theta}2\right)\right]^{3/2}\left(\tan\theta_c\right)^{1/2}}
e^{-i k_{+}\vert y \vert\cos(\theta_c+\theta_{\hat y})},\quad \theta\in(0,\pi/2).
\end{align}
Thus it follows from (\ref{neq:4.2.33}) that
\begin{align}{\label{j:e26}}
&G^{(2)}_{\mathcal R} (x,y)\notag\\
=&\frac{ie^{ik_+\vert x \vert}}{4\pi}\left\{g(s_b)F_2(i{\vert x \vert},s_b,1/2)
+\frac{g(s_b)-g(0)}{s_b}F_2(i{\vert x \vert}, s_b,3/2)+ G^{(2)}_{\mathcal R,1}(x,y)\right\}\notag\\
=&{\frac{e^{{11\pi i}/{8}+ik_{+}[\vert x \vert\cos^2({(\theta_c-\theta_{\hat x})}/2)
-\vert y \vert\cos(\theta_c+\theta_{\hat{y}})]}}{\sqrt{\pi}\vert x \vert^{\frac34} k^{\frac34}_{+}
\left[\cos\left(\frac{\theta_c-\theta_{\hat x}}2\right)\right]^{3/2}
{\left(\tan\theta_c\right)^{1/2}}}}\left(D_{\frac 12}(u)+\frac{A-1}{u}D_{\frac 32}(u)\right)\notag\\
&+\frac{ie^{ik_+\vert x \vert}}{4\pi}G^{(2)}_{\mathcal R,1}(x,y),
\end{align}
where $A:=g(0)/g(s_b)$ and $u:=2\sqrt{k_{+}\vert x \vert}e^{{3\pi i}/4}\sin({(\theta_c-\theta_{\hat x})}/2)$.

In the remaining part of this case, we further assume that $\vert x \vert$ is large enough so that
$\vert x \vert\geq 8\left(\sigma^{(1)}_{\min}\right)^{-2}$ with $\sigma^{(1)}_{\min}$ given in
Lemma \ref{le:4}. Then, due to the assumption
$\left\vert \sin\left({(\theta_c-\theta_{\hat x})}/2\right)\right\vert <({2\sqrt{k_{+} \vert x \vert}})^{-1}$
and Lemma \ref{le:4}, we have $\vert s_b\vert <\sigma^{(1)}_{\theta_{\hat{x}}}/4$.
Hence, we can use the analyticity of $g(s)$ in $L_{\sigma^{(1)}_{\theta_{\hat{x}}}}$ and
the Taylor formula to obtain that
\begin{align*}
g_1(s)&= \frac{g(s)-g(0)-(g(s_b)-g(0))s/s_b}{s(s-s_b)}\\
&=\frac{\sum^{+\infty}_{n=1}{\frac{d^n g}{d s^n}(0)}s^{n-1}/{n!}
-\sum^{+\infty}_{n=1}{\frac{d^n g}{d s^n}(0)}s_b^{n-1}/{n!}}{s-s_b}
\end{align*}
for $\vert s\vert <{\sigma^{(1)}_{\theta_{\hat{x}}}}$.
These, together with (\ref{eq31}) and (\ref{Ne16}), imply that
\begin{align}\label{eq12}
\vert g_1(s)\vert &\le \sum^{+\infty}_{n=2}{(n-1)}\left\vert \frac{\frac{d^n g}{d s^n}(0)}{n!}
\frac{{\left(\sigma^{(1)}_{\theta_{\hat{x}}}\right)}^{n-2}}{2^{n-2}}\right\vert \notag\\
&\le {C_{R_0}}{\left(\sigma^{(1)}_{\theta_{\hat{x}}}\right)^{-2}}
\sum^{+\infty}_{n=2}{\frac{n-1}{2^{n-2}}}\le C_{R_0} \quad \textrm{for}\;\vert s\vert
\leq {\sigma^{(1)}_{\theta_{\hat{x}}}}/2.
\end{align}
Further, similarly as above, we can apply the Taylor formula and (\ref{Ne16}) to obtain that
\begin{align} \label{eq:14}
\vert g(0)-g(s_b)\vert /\vert s_b\vert  \le \left\vert \sum^{+\infty}_{n=1}\frac{d^n g}{d s^n}(0) s^{n-1}_b/n!\right\vert
\le C_{R_0}\sum_{n=1}^{+\infty}\left(\sigma^{(1)}_{\theta_{\hat{x}}}\right)^{-n}\vert s_b\vert^{n-1}
\le C_{R_0}.
\end{align}
From (\ref{eq31}), (\ref{Ne3}), (\ref{eq:14}) and $\vert s_b\vert <\sigma^{(1)}_{\theta_{\hat{x}}}/4$, we have
\begin{align}\label{eq13}
\vert g_1(s)\vert \leq \vert g(s)-g(0)\vert /(\vert s\vert \vert s-s_b\vert ) + C_{R_0}/\vert s-s_b\vert
\le C_{R_0}{(1+\vert s\vert )}^{\frac 52}e^{\widetilde{C}_{R_0}{\vert s\vert }^2}
\end{align}
for $s \in L_{\sigma^{(1)}_{\theta_{\hat{x}}}}$ with $\vert s\vert >\sigma^{(1)}_{\theta_{\hat{x}}}/2$.
Thus it follows from (\ref{eq12}) and (\ref{eq13}) that
\begin{align} \label{eq:25}
\vert g_1(s)\vert  \le  C_{R_0}{(1+\vert s\vert )}^{\frac 52}e^{\widetilde{C}_{R_0}{\vert s\vert }^2}\;
\textrm{for}\;s\in L_{\sigma^{(1)}_{\theta_{\hat{x}}}}.
\end{align}
Then similarly as in deriving the estimate (\ref{Ne16}), we can use (\ref{eq31}), (\ref{eq:25}) and
the Cauchy inequality to obtain
\begin{align*}
\vert g'_1(s)\vert  \le  C_{R_0}{(1+\vert s\vert )}^{\frac 52}e^{\widetilde{C}_{R_0}{\vert s\vert }^2},~s\in \Rb.
\end{align*}
Hence, by this and (\ref{eq:25}) we arrive at the result
\begin{align} \label{eq:3.2.14}
\vert G^{(2)}_{\mathcal R,1}(x,y)\vert &\le C_{R_0}\vert x \vert^{-1}\int^{+\infty}_{-\infty}  (1+\vert s\vert^ {4})e^{\left(\widetilde{C}_{R_0}-\vert x \vert\right){\vert s\vert }^2}ds\nonumber\\
&\le C_{R_0}\vert x \vert^{-1}\left((\vert x \vert-\widetilde{C}_{R_0})^{-1/2}+(\vert x \vert-\widetilde{C}_{R_0})^{-5/2}\right)
\le {C_{R_0}}{\vert x \vert}^{-3/2}.
\end{align}
From the explicit expression of $g(s_b)$ and (\ref{eq:14}), we have
\begin{align}\label{eq15}
\vert ({A-1})/{u}\vert =\left\vert\left({g(0)-g(s_b)}\right)/\left({i\sqrt{2\vert x \vert}s_bg(s_b)}\right)\right\vert
\le C_{R_0}\vert x \vert^{-1/2}.
\end{align}
Moreover, since $\vert u\vert <1$ due to the assumption
$\vert \sin\left(({\theta_c-\theta_{\hat x}})/2\right)\vert <{(2\sqrt{k_{+} \vert x \vert})}^{-1}$, we can apply the
fact that $D_{ 1/2}(z)$ and $D_{3/2}(z)$ are analytic for any
$z\in\mathbb{C}$ to obtain that $\vert D_{ 1/2}(u)\vert \leq C$ and $\vert D_{3/2}(u)\vert \leq C$ for some constant $C>0$.
This, together with (\ref{j:e26}), (\ref{eq:3.2.14}), (\ref{eq15}) and the fact that $\vert {\cos\left({(\theta_c-\theta)}/2\right)}\vert \geq \cos(\pi/4-\theta_c/2)$ for any
$\theta\in\left(\theta_c,{\pi}/2\right)$, implies that
\be \label{eq:3.2.15}
\vert G^{(2)}_{\mathcal R}(x,y)\vert \le {C_{R_0}}{\vert x \vert^{-3/4}}(1+\vert x \vert^{-1/2}) + {C_{R_0}}{\vert x \vert}^{-3/2}
\leq {C_{R_0}}{\vert x \vert^{-3/4}}.
\en
Further, using the assumption $\vert \sin\left(({\theta_c-\theta_{\hat x}})/2\right)\vert <{(2\sqrt{k_+\vert x \vert})}^{-1}$
again, we have
\begin{align*}
\vert \cos\theta_{\hat x}-\cos \theta_c\vert = \left\vert 2\sin \left(({\theta_c-\theta_{\hat x}})/2\right)\sin \left(({\theta_c+\theta_{\hat x}})/2\right)\right\vert\le  k^{-1/2}_+ \vert x \vert^{-1/2},
\end{align*}
which yields that
\be \label{eq:3.2.16}
\left\vert \frac{e^{ik_{+}\vert x \vert}}{\sqrt{\vert x \vert}}\frac{e^{i\frac{\pi}4}}{\sqrt{8\pi k_{+}}}
\frac{2i\mathcal S{(\cos\theta_{\hat x},n)}}{n^2-1}\sin\theta_{\hat x}e^{-ik_{+}
\vert y \vert\cos(\theta_{\hat x}+\theta_{\hat y})}\right\vert\le C \vert x \vert^{- 3/4}
\en
for some constant $C>0$.
Therefore, it follows from (\ref{eq:3.2.15}) and (\ref{eq:3.2.16}) that $G^{(2)}_{\mathcal R}(x,y)$
has the form (\ref{eq:3.2.11}) with  $G^{(2)}_{\mathcal R,Res}(x,y)$ satisfying (\ref{neq:4.2.23})
uniformly for all $\theta_{\hat{x}}\in(\theta_c, \pi/2)$ with
$\vert \sin\left(({\theta_c-\theta_{\hat x}})/2\right)\vert <{(2\sqrt{k_{+}\vert x \vert})}^{-1}$
and $y\in B^+_{R_0}$.

\textbf{Case 3}: $\theta_{\hat x}=\theta_c$.
Taking the limit $\theta_{\hat{x}}\rightarrow +\theta_c$ along the real axis in (\ref{neq:4.2.45}),
we obtain that for $\theta_{\hat{x}}=\theta_c$,
\begin{align*}
\mathcal S(\cos(\zeta(s)),n)=
\left\{
\begin{array}{ll}
  \sqrt {2/k_{+}}e^{-i\frac{\pi}{4}}H_{\theta_c}(s) H_{\pi-\theta_c}(s)\sqrt{s-s_b^{*}}
  \sqrt{s}, & s\geq 0,\\
  \sqrt {2/k_{+}}e^{-i\frac{\pi}{4}}H_{\theta_c}(s) H_{\pi-\theta_c}(s)\sqrt{s-s_b^{*}}
  \sqrt{-s}i, & s<0.
\end{array}
\right.
\end{align*}
Then, by (\ref{eq16}), $G^{(2)}_{\RR}(x,y)$ can be written as
\begin{align}
G^{(2)}_{\RR}(x,y) &= \frac{ie^{ik_{+}\vert x \vert}}{4\pi}\int^{+\infty}_{0}
\left(g(s)+i g(-s)\right)\sqrt{s}e^{-\vert x \vert s^2}ds\nonumber\\
& = \frac{ie^{ik_{+}\vert x \vert}}{4\pi}\int^{+\infty}_{0}\bigg\{\left(g(0)+i g(0)\right)\sqrt{s}e^{-\vert x \vert s^2} \notag \\
&\quad +\left[(g(s)-g(0))+i (g(-s)-g(0))\right]\sqrt{s}e^{-\vert x \vert s^2}\bigg\}ds\nonumber\\ \label{eq45}
&= \frac{ie^{ik_{+}\vert x \vert}}{4\pi}\bigg\{2^{-{1}/{2}}e^{{\pi i}/{4}}g(0)\Gamma({3}/{4})\vert x \vert^{-3/4}
\notag\\
&\quad+\int^{+\infty}_{0}\left[(g(s)-g(0))+i (g(-s)-g(0))\right]\sqrt{s}e^{-\vert x \vert s^2}ds\bigg\},
\end{align}
where $g(s)$ is given by (\ref{eq:18}).
Similarly as in the discussion in Part I of this step, we can apply (\ref{eq31}), (\ref{Ne16})
and the mean-value theorem for $\Rt(g(s)-g(0))$ and $\I (g(s)-g(0))$ to obtain that
\ben
\vert (g(s)-g(0))/s\vert \leq C_{R_0}(1+\vert s\vert )^{5/2} e^{\widetilde{C}_{R_0}\vert s\vert^2}~\textrm{for}~s\in\mathbb{R}.
\enn
Thus it follows that
\begin{align}
&\left\vert\int^{+\infty}_{0}\left[(g(s)-g(0))+i (g(-s)-g(0))\right]\sqrt{s}e^{-\vert x \vert s^2}ds\right\vert\notag\\
&\leq C_{R_0}\int^{+\infty}_{0}(s^{3/2}+s^4)e^{-(\vert x \vert-\widetilde{C}_{R_0})s^2}ds\nonumber\\
&\leq C_{R_0}\left[ (\vert x \vert-\widetilde{C}_{R_0})^{-5/4}
+(\vert x \vert-\widetilde{C}_{R_0})^{-5/2} \right]\leq C_{R_0}\vert x \vert^{-5/4}\label{eq46}
\end{align}
for large enough $\vert x \vert$. Hence, using (\ref{eq45}) and (\ref{eq46}), we have
$\vert G^{(2)}_{\mathcal{R}}(x,y)\vert \leq C_{R_0}\vert x \vert^{-3/4}$.
This, together with the fact that $\mathcal S{(\cos\theta_{c},n)}=0$, implies that
$G^{(2)}_{\mathcal R}(x,y)$ has the form
(\ref{eq:3.2.11}) with $G^{(2)}_{\mathcal R,Res}(x,y)$ satisfying (\ref{neq:4.2.23}) uniformly
for $\theta_{\hat{x}}=\theta_c$ and all $y\in B^+_{R_0}$.

Based on the discussions for the above three cases, we obtain that $G^{(2)}_{\mathcal R}(x,y)$
has the form (\ref{eq:3.2.11}) with $G^{(2)}_{\mathcal R,Res}(x,y)$ satisfying (\ref{neq:4.2.23})
uniformly for all $\theta_{\hat x}\in \left[\left.\theta_c,{\pi}/2\right)\right.$ and $y\in B^+_{R_0}$.

\textbf{Step 2}: Estimate of $G^{(1)}_{\mathcal R}(x,y)$. From (\ref{eq16}) and statement (\ref{e1}) of
Lemma \ref{le:4}, it is easily seen that
\begin{align*}
G^{(1)}_{\RR}(x,y) = \frac{ie^{ik_{+}\vert x \vert}}{4\pi} \int^{+\infty}_{-\infty}g_2(s) e^{-\vert x \vert s^2}ds,
\end{align*}
where $g_2(s): = v_1(s)F(s){d\zeta(s)}/{ds}$ is analytic in $ L_{\sqrt{k_+}}$.
A straightforward calculation gives that
\begin{align}
&G^{(1)}_{\mathcal R}(x,y)=\frac{ie^{ik_{+}\vert x \vert}}{4\pi}\left(\int^{+\infty}_{-\infty}g_2(0)
e^{-\vert x \vert s^2}ds + \int^{+\infty}_{-\infty}g'_{2}(0) se^{-\vert x \vert s^2}ds
+ G^{(1)}_{\mathcal R, Res}(x,y)\right) \notag \\
& = \frac{e^{ik_{+}\vert x \vert}}{\sqrt{\vert x \vert}}\frac{e^{i\frac{\pi}4}}{\sqrt{8\pi k_{+}}}
\left(\frac{2\cos^2\theta_{\hat x}-1-n^2}{n^2-1}\right) e^{-i k_{+}\vert y \vert
\cos(\theta_{\hat x}{+}\theta_{\hat y})} + \frac{ie^{ik_{+}\vert x \vert}}{4\pi}
G^{(1)}_{\mathcal R, Res}(x,y), \label{eq:26}
\end{align}
where $G^{(1)}_{\mathcal R,Res}(x,y):=\int^{+\infty}_{-\infty}g_{2,Res}(s)e^{-\vert x \vert s^2}ds$
with $g_{2,Res}(s):= g_2(s)-g_2(0)-g'_2(0)s$.

Let ${\widetilde{\sigma}}$ be a fixed number in $(0,\sqrt{k_+})$.
Note that $\vert P(s)\vert \neq 0$ in $s\in L_{\sqrt{k_+}}$.
Then similarly to the derivation of (\ref{Ne3}), we can get that
$\left\vert g_2(s)\right\vert \le C_{R_0,{\widetilde{\sigma}}}{(1+\vert s\vert )}^{3}e^{\widetilde{C}_{R_0}{\vert s\vert }^2}$
for any $s\in  L_{{\widetilde{\sigma}}}$,
where the constant $C_{R_0,{\widetilde{\sigma}}}>0$ depends on $R_0$ and ${\widetilde{\sigma}}$.
Thus it follows from Cauchy inequality that
\ben
\left\vert\frac{d^2 g_2}{ds^2}(s)\right\vert \le {2 C_{R_0,{\widetilde{\sigma}}}(1+\vert s\vert +{\widetilde{\sigma}})^3
e^{\widetilde{C}_{R_0}(\vert s\vert +{\widetilde{\sigma}})^2}}/\widetilde{\sigma}^2
\leq C_{R_0,{\widetilde{\sigma}}}(1+\vert s\vert )^3e^{\widetilde{C}_{R_0}\vert s\vert^2}
\enn
for $s\in\Rb$.
Hence, similarly as in the discussion in Part I of Step 1, we can apply the mean-value theorem
for $\Rt(g_{2,Res})$ and $\I (g_{2,Res})$ to obtain that
\begin{align*}
\left\vert\frac{g_{2,Res}(s)}{s^2}\right\vert
\le C_{R_0,{\widetilde{\sigma}}}{(1+\vert s\vert )}^{3}e^{\widetilde{C}_{R_0}{\vert s\vert }^2}~\textrm{for}~s\in\Rb.
\end{align*}
For simplicity, we choose ${\widetilde{\sigma}}=\sqrt{k_+}/2$ and so it follows that
\begin{align*}
\vert G^{(1)}_{\mathcal R,Res}(x,y)\vert &\le C_{R_0,\sqrt{k_+}/2}\int_0^{+\infty}
\left({s}^2+s^{5}\right)e^{-(\vert x \vert-\widetilde{C}_{R_0})s^2}ds\\
&\le C_{R_0,\sqrt{k_+}/2}\left[ (\vert x \vert-\widetilde{C}_{R_0})^{-3/2}
+ (\vert x \vert-\widetilde{C}_{R_0})^{-3} \right]\leq C_{R_0}\vert x \vert^{-3/2}
\end{align*}
for $\vert x \vert$ large enough.
This, together with (\ref{eq:26}), implies that $G^{(1)}_{\mathcal R}(x,y)$ has the form (\ref{j:e25})
with $G^{(1)}_{\mathcal R,Res}(x,y)$ satisfying (\ref{neq:4.2.32}) uniformly for all
$\theta_{\hat x}\in \left[\left.\theta_c,{\pi}/2\right)\right.$ and $y\in B^+_{R_0}$.
\end{proof}

\begin{lemma} \label{Le:n4.2.2}
Assume that $k_+>k_-$ and let $R_0>0$ be an arbitrary fixed number.
Suppose that $y=\vert y \vert(\cos \theta_{\hat y}, \sin \theta_{\hat y})\in B^+_{R_0}$
with $\theta_{\hat{y}}\in(0,\pi)$ and $x=\vert x \vert(\cos\theta_{\hat x},\sin\theta_{\hat x})\in\Rb^2_+$
with $\theta_{\hat x}\in (0,\theta_c)$, then the following statements hold.
\begin{enumerate}
\item\label{s1}
$G^{(1)}_{\mathcal R}(x,y)$ has the asymptotic behavior (\ref{j:e25}) with
$G^{(1)}_{\mathcal R,Res}(x,y)$ satisfying (\ref{neq:4.2.32}) uniformly for all
$\theta_{\hat x}\in (0,\theta_c)$ and $y\in B^+_{R_0}$.
\item \label{s2}
$G^{(3)}_{\mathcal R}(x,y)$ has the asymptotic behavior
\begin{align*}
& G^{(3)}_{\mathcal R}(x,y)=\\
&\frac{e^{ik_{+}\vert x \vert}}{\sqrt{\vert x \vert}}\frac{e^{i\frac{\pi}4}}{\sqrt{8\pi k_{+}}}
\left(\frac{2i\sin\theta_{\hat x}\widetilde{\mathcal S}(\cos\theta_{\hat x},n)}{n^2-1}\right) e^{-ik_{+}\vert y \vert\cos(\theta_{\hat x}+\theta_{\hat y})} + G^{(3)}_{\mathcal R,Res}(x,y)
\end{align*}
with $G^{(3)}_{\mathcal R,Res}(x,y)$ satisfying that
\begin{align*}
&\vert G^{(3)}_{\mathcal R,Res}(x,y)\vert \le{C_{R_0}}{\vert x \vert^{- 3/4}}, && \vert x \vert\rightarrow+\infty,\\
&\vert G^{(3)}_{\mathcal R,Res}(x,y)\vert  \le {C_{R_0}}{{{\left\vert \theta_c-\theta_{\hat x}
\right\vert}^{-\frac 32}}\vert x \vert^{-\frac 32}}, && \vert x \vert\rightarrow+\infty,
\end{align*}
uniformly for all $\theta_{\hat x}\in (0,\theta_c)$ and $y\in B^+_{R_0}$.
\item \label{s4}
$G^{(4)}_{\mathcal R}(x,y)$ satisfies that
\begin{align}
&\vert G^{(4)}_{\mathcal R}(x,y)\vert \le
 {C_{R_0}}{\vert x \vert^{- 3/4}},&& \vert x \vert\rightarrow+\infty,\label{neq:4.2.36}\\ \label{neq:4.2.37}
&\vert G^{(4)}_{\mathcal R}(x,y)\vert  \le {C_{R_0}}{{{\left\vert \theta_c-\theta_{\hat x}\right\vert}^{-\frac 32}}
\vert x \vert^{-\frac 32}},&& \vert x \vert\rightarrow+\infty,
\end{align}
uniformly for all $\theta_{\hat x}\in (0,\theta_c)$ and $y\in B^+_{R_0}$.
\end{enumerate}
Here, the constant $C_{R_0}>0$ is independent of $x$ and $y$ but dependent of $R_0$.
\end{lemma}

\begin{proof}
Let $\vert x \vert$ be sufficiently large throughout the proof.
Statement (\ref{s1}) can be obtained in a same way as in Step 2 of the proof of Lemma \ref{NLe:3.2.4}.
Moreover, as mentioned in Remark \ref{re2}, the steepest descent path $\mathcal{D}$ will not
cross the branch cuts of $\widetilde{\mathcal S}(\cos\zeta, n)$
for the case $\theta_{\hat{x}}\in(0,\theta_c)$.
Thus statement (\ref{s2}) can be proved in a similar way as in Step 1 of the proof of Lemma \ref{NLe:3.2.4}.
Note that formula (\ref{j:e11}) is needed for deducing statement (\ref{s2}) since $\I(s_b)>0$
for $\theta_{\hat{x}}\in(0,\theta_c)$ (compare the derivation of (\ref{j:e26})).

Now we only need to prove statement (\ref{s4}). From  (\ref{eq:20}), (\ref{eq:19}), (\ref{eq:11})
and the fact that ${\mathcal S}_{-}(\cos\zeta,n) =  -\widetilde{\mathcal S}(\cos\zeta,n)$
for $\zeta\in\mathcal I_{\theta_c,\zeta_o}$ (see (\ref{eq8})), it can be seen that
\begin{align*}
&G^{(4)}_{\mathcal R}(x,y)=\\
&-\frac{i}{2\pi}\left[\int_{\mathcal I_{\theta_c,\zeta_o}}
+ \int_{\mathcal D_{\zeta_o,\frac\pi2+\theta_{\hat{x}}-i\infty}} \right]
\frac{2i\sin\zeta \widetilde{\mathcal S}(\cos\zeta,n)}{n^2-1} {e^{ik_{+}
\left(-\vert y \vert\cos(\zeta+\theta_{\hat y})+ \vert x \vert\cos(\zeta-\theta_{\hat x})\right)}}d\zeta.
\end{align*}
Here, the paths $\mathcal I_{\theta_c,\zeta_o}$ and $\mathcal D_{\zeta_o,\frac\pi2+\theta_{\hat{x}}-i\infty}$
are defined in the proof of Lemma \ref{Le:n4.2.1}.
Then it follows from Remark \ref{re2} and Cauchy integral theorem  that
\begin{align*}
& G^{(4)}_{\mathcal R}(x,y)=\\
&-\frac{i}{2\pi}\left[\int_{\theta_c}^{\theta_{\hat x}}
 + \int_{\mathcal D_{\theta_{\hat x}, \frac\pi2+\theta_{\hat x}-i\infty}}\right]
  \frac{2i\sin\zeta\widetilde{\mathcal S}(\cos\zeta,n)}{n^2-1}{e^{ik_{+}
  \left(-\vert y \vert\cos(\zeta+\theta_{\hat y})+ \vert x \vert\cos(\zeta-\theta_{\hat x})\right)}}d\zeta,
\end{align*}
where $\mathcal D_{\theta_{\hat x}, \pi/2+\theta_{\hat x}-i\infty}$ denotes the part of
the path $\mathcal{D}$ starting from $\theta_{\hat x}$ and ending at $\pi/2+\theta_{\hat{x}}-i\infty$.
Let $\zeta(s)$ be defined in Lemma \ref{le:4} and $\mathcal{I}_{s_b,0}$ be the path $\{\sqrt{2k_+}e^{i\pi/4}t:t\in\mathbb{R}~\textrm{s.t.}~0<t<\sin((\theta_c-\theta_{\hat{x}})/2)\}$
with the orientation from $s_b$ to $0$.
From the arguments in the proof of statement (\ref{e1}) of Lemma \ref{le:4}, it can be seen that
$\zeta(s)$ is a conformal mapping from $L_{\sqrt{k_+}}$ to $\{\zeta(s):s\in L_{\sqrt{k_+}}\}$.
Since $\mathcal{I}_{s_b,0}$ lies in $L_{\sqrt{k_+}}$,
it follows from the change of variable $\zeta = \zeta(s)$ and statement (\ref{e2}) of Lemma \ref{le:4} that
\begin{align}
G^{(4)}_{\mathcal R}(x,y)&=-\frac{ie^{i k_{+}\vert x \vert}}{2\pi}\left[\int_{\mathcal{I}_{s_b,0}}
+\int^{+\infty}_{0}\right]\frac{2i\sin\zeta(s)\widetilde{\mathcal S}(\cos{(\zeta(s))},n)}{n^2-1}
F(s)\frac{d\zeta(s)}{ds}e^{-\vert x \vert s^2}ds\nonumber\\ \label{eq30}
&=-\frac{ie^{i k_{+}\vert x \vert}}{2\pi}\left[\int_{\mathcal{I}_{s_b,0}}
+ \int^{+\infty}_{0}\right]\sqrt{s-s_b} f(s)e^{-\vert x \vert s^2}ds
\end{align}
with $F(s)$ defined in (\ref{eq:3.2.3}) and $f(s)$ given by
\begin{align*}
f(s):= -\sqrt {2/k_{+}}e^{-i\frac{\pi}{4}}H_{\theta_c}(s) H_{\pi-\theta_c}(s)
\sqrt{s-s_b^{*}}F(s)\frac{d\zeta(s)}{ds} \frac{2i\sin\zeta(s)}{n^2-1}.
\end{align*}
Further, it can be seen from statement (\ref{e3}) of Lemma \ref{le:4} that
$f(s)$ is an analytic function in the strip $L_{\sigma^{(1)}_{\theta_{\hat{x}}}}$ with
$\sigma^{(1)}_{\theta_{\hat{x}}}$ given as in Lemma \ref{le:4}.
Moreover, by same arguments as in the derivation of (\ref{Ne3}), we have
\begin{align}\label{eq:27}
\vert f(s)\vert \le C_{R_0}{(1+\vert s\vert )}^{\frac 52}e^{\widetilde{C}_{R_0}{\vert s\vert }^2},\quad
s\in  L_{\sigma^{(1)}_{\theta_{\hat{x}}}},
\end{align}
where the positive constants $C_{R_0}$ and $\widetilde{C}_{R_0}$ are independent of $\theta_{\hat x}$
but dependent of $R_0$.
Due to (\ref{eq31}), it follows from Cauchy inequality and (\ref{eq:27}) that for $n=0,1,2,\ldots$,
\begin{align}\label{eq34}
\left\vert\frac{d^n f}{d s^n}(s)\right\vert&\leq
\frac{C_{R_0}n!(1+\vert s\vert +\sigma^{(1)}_{\min}/2)^{5/2}
e^{\widetilde{C}_{R_0}(\vert s\vert +\sigma^{(1)}_{\min}/2)^2}}{(\sigma^{(1)}_{\min}/2)^{n}}
\leq\frac{C_{R_0}n!(1+\vert s\vert )^{5/2}
e^{\widetilde{C}_{R_0}\vert s\vert^2}}{(\sigma^{(1)}_{\min}/2)^{n}}\nonumber\\
&\leq\frac{C_{R_0}n!(1+\vert \Rt(s)\vert )^{5/2}
e^{\widetilde{C}_{R_0}\vert \Rt(s)\vert^2}}{(\sigma^{(1)}_{\min}/2)^{n}},
\quad s\in  L_{{\sigma^{(1)}_{\min}/2}}.
\end{align}

The rest of the proof is divided into the following two parts.

\textbf{Part I}: We prove that $G^{(4)}_{\mathcal R}(x,y)$ satisfies (\ref{neq:4.2.37}) uniformly
for all $\theta_{\hat x}\in (0,\theta_c)$ and $y\in B^+_{R_0}$.

Due to $\Rt(s_b)>0$ and $\I(s^2_b)>0$, we let the path $\mathcal{D}_{s_b}$ denote the curve $\{s\in\mathbb{C}:\Rt(s)\in(\Rt(s_b),+\infty),
\I(s)=(\Rt(s))^{-1}\I(s^2_b)/2\}$ with the orientation from $s_b$ to $+\infty$.
Clearly, $\Ima(s^2) =  \Ima(s_b^2)$ on $s \in \mathcal{D}_{s_b}$ and $\mathcal{D}_{s_b}\subset L_{\sigma^{(1)}_{\theta_{\hat{x}}}}$.
It is worth noting that $\mathcal{D}_{s_b}$ is a part of
the steepest descent path of the function $h(s)=-s^2$ crossing $s_b$.
Thus, using (\ref{eq30}), (\ref{eq:27}),  the analyticity of $f(s)$ and Cauchy integral theorem,
we can rewrite $G^{(4)}_{\RR}(x,y)$ as the integral along the path $\mathcal{D}_{s_b}$, that is,
\begin{align*}
G^{(4)}_{\RR}(x,y) =  -{ie^{i k_{+}\vert x \vert}}({2\pi})^{-1}\int_{\mathcal{D}_{s_b}}\sqrt{s-s_b}f(s)e^{-\vert x \vert s^2}ds.
\end{align*}
Define the function $\phi(t):=\left({s_b^2 + t}\right)^{1/2}$.
Since $\Rt(\phi(t))>\Rt(s_b)$ and $\I(\phi^2(t))=\I(s^2_b)$ for $t\in(0,+\infty)$, it is easily
seen that $\phi(t)$ travels from $s_b$ to
$+\infty$ on the path $\mathcal{D}_{s_b}$ as $t$ travels from $0$ to $+\infty$ along the real axis.
Thus introducing the change of variable $s = \phi(t)$ and applying the fact that
$\Ima(s^2)=\Ima(s_b^2)$ for $s \in \mathcal{D}_{s_b}$, we have
\be \label{neq:4.2.34}
G^{(4)}_{\RR}(x,y)=-{ie^{i k_{+}\vert x \vert}}({2\pi})^{-1}e^{-\vert x \vert s_b^2}\int^{+\infty}_0\sqrt t f_2(t)e^{-\vert x \vert t}dt,
\en
where $f_{2}(t) :=[2\left(\phi(t) + s_b\right)^{1/2}\phi(t)]^{-1}f(\phi(t))$. Note that
$\left\vert \big(\phi(t) + s_b\big)^{1/2}\right\vert  \ge \vert s_b\vert^{1/2}$ and
$\left\vert \phi(t)\right\vert \ge \vert s_b \vert $  for $t\in [0,+\infty)$. Hence, combining (\ref{eq:27}) and
(\ref{neq:4.2.34}), we have that for $\vert x \vert$ large enough,
\begin{align}
&\left\vert G^{(4)}_{\RR}(x,y)\right\vert \le {C_{R_0}}\vert s_b \vert^{-3/2}\int^{+\infty}_0
\sqrt t{(1+ \vert s_b\vert +\sqrt t)}^{\frac 52}e^{\widetilde{C}_{R_0}{(\vert s_b\vert^ 2+t)}}e^{-\vert x \vert t}dt\nonumber\\
&\le {C_{R_0}}\vert s_b\vert^{-3/2}\int^{+\infty}_{0}(\sqrt t + t^{7/4})e^{-(\vert x\vert -\widetilde{C}_{R_0})t}dt\nonumber\\
&\le {C_{R_0}}\vert s_b\vert^{-3/2}\left[(\vert x \vert-\widetilde{C}_{R_0})^{-3/2}
+(\vert x \vert-\widetilde{C}_{R_0})^{-11/4}\right]\nonumber\\ \label{eq33}
&\leq C_{R_0}\vert \sin((\theta_c-\theta_{\hat{x}})/2)\vert^{-3/2}\vert x \vert^{-3/2}.
\end{align}
From this, we easily obtain that $G^{(4)}_{\mathcal R}(x,y)$ satisfies (\ref{neq:4.2.37})
uniformly for all $\theta_{\hat x}\in (0,\theta_c)$ and $y\in B^+_{R_0}$.

\textbf{Part II}: We prove that $G^{(4)}_{\mathcal R}(x,y)$ satisfies (\ref{neq:4.2.36}) uniformly
for all $\theta_{\hat x}\in (0,\theta_c)$ and $y\in B^+_{R_0}$. For this aim, we distinguish
between the following two cases.

\textbf{Case 1}: $\theta_{\hat x}\in (0,\theta_c)$ with
$\left\vert \sin\left({(\theta_c-\theta_{\hat x})}/2\right)\right\vert \geq {(2\sqrt{k_{+} \vert x \vert})}^{-1}$.
The proof of this case can be easily obtained by applying (\ref{eq33}).

\textbf{Case 2}: $\theta_{\hat x}\in (0,\theta_c)$ with
$\left\vert \sin\left({(\theta_c-\theta_{\hat x})}/2\right)\right\vert < {(2\sqrt{k_{+} \vert x \vert})}^{-1}$.
In this case, assume that $\vert x \vert$ is large enough s.t. $\vert x \vert\geq 8\left(\sigma^{(1)}_{\min}\right)^{-2}$
and thus $\vert s_b\vert <\sigma^{(1)}_{\min}/4\leq\sigma^{(1)}_{\theta_{\hat{x}}}/4$.
Then from (\ref{eq30}), (\ref{eq:27}), the analyticity of $f(s)$, the fact that $\I(s_b)=\vert s_b\vert /\sqrt{2}<\sigma^{(1)}_{\min}/(4\sqrt{2})$ and Cauchy integral theorem,
it follows that $G^{(4)}_{\mathcal R}(x,y)$ can be rewritten as
\begin{align}
G^{(4)}_{\mathcal R}(x,y)= -{ie^{i k_{+}\vert x \vert}}{(2\pi)^{-1}}\int_{\mathcal{I}_{s_b,s_b+\infty}}
\sqrt{s-s_b}f(s)e^{-\vert x \vert s^2}ds, \label{eq:21}
\end{align}
where $\mathcal{I}_{s_b,s_b+\infty}$ denotes the path $\{s_b+t:0<t<+\infty)\}$
with the orientation from $s_b$ to $s_b+\infty$.
Define $f_1(s):= {\left[f(s)-f(s_b)-(f(s_b)-f(0))(s-s_b)/{s_b}\right]}/{\left[s(s-s_b)\right]}$.
Due to the  analyticity of $f(s)$,
it is clear that
$f_1(s)$
is analytic in $s\in L_{\sigma^{(1)}_{\theta_{\hat{x}}}}$.
Thus similarly to the derivation of (\ref{eq35}),
we can apply (\ref{eq34}) and (\ref{eq:21}) to get that
\begin{align}\label{eq36}
G^{(4)}_{\mathcal R}(x,y)=&-\frac{ie^{i k_{+}\vert x \vert}}{2\pi}\int_{\mathcal{I}_{s_b,s_b+\infty}}\left[f(s_b)\sqrt{s-s_b} + \frac{f(s_b)-f(0)}{s_b}(s-s_b)^{\frac 32}\right ]e^{-\vert x \vert s^2}ds \notag\\
&-\frac{ie^{i k_{+}\vert x \vert}}{2\pi} G^{(4)}_{\mathcal R,1}(x,y),
\end{align}
where
\begin{align}\label{eq38}
G^{(4)}_{\mathcal R,1}(x,y):=\frac{1}{2\vert x \vert}\int_{\mathcal{I}_{s_b,s_b+\infty}}\sqrt{s-s_b}
\left((s-s_b) f'_{1}(s)+\frac 32 f_1(s)\right)e^{-\vert x \vert s^2}ds.
\end{align}
Since $\I (-s_b)<0$, it follows from a change of variable $t=-s$ that
$\int_{\mathcal{I}_{s_b,s_b+\infty}}(s-s_b)^{\beta}e^{-\vert x \vert s^2}ds=2^{-1}e^{-i\beta\pi}F_3(i\vert x \vert,-s_b,\beta)$
for $\beta=1/2$ and $\beta=3/2$, where $F_3$ is defined in Lemma \ref{NLL1}.
This, together with (\ref{eq:j.3}) and (\ref{eq36}), implies that
\begin{align}
&G^{(4)}_{\mathcal R}(x,y)+\frac{ie^{ik_+\vert x \vert}}{2\pi} G^{(4)}_{\mathcal R,1}(x,y)\notag\\
&=-\frac{e^{ik_+\vert x \vert}}{4\pi}\left\{f(s_b)F_3(i{\vert x \vert},-s_b,1/2)
-\frac{f(s_b)-f(0)}{s_b}F_3(i{\vert x \vert},-s_b,3/2)\right\}
\notag\\ \label{eq54}
&=\frac{f(s_b)e^{i(k_+\vert x \vert\cos^2({(\theta_c-\theta_{\hat x})}/2)+\frac{\pi}{2})}}
{\vert x \vert^{\frac34}2^{\frac74}\Gamma(-1/2)}\left(D_{-\frac 32}(\widetilde{u})+\frac{\Gamma(-1/2)}{\Gamma(-3/2)}
\left(\frac{\widetilde{A}-1}{\widetilde{u}}\right)
D_{-\frac 52}(\widetilde{u})\right),
\end{align}
where $\widetilde{u}: =2\sqrt{k_{+}\vert x \vert}e^{\frac{i\pi}4}\sin({(\theta_c-\theta_{\hat x})}/2)$ and
$\widetilde{A}:= f(0)/f(s_b)$.
It is easily seen that $f(s_b)=-h_c(\theta_{\hat{x}})$ with the function $h_c$ defined in (\ref{eq14}).
Similarly to the derivation of (\ref{eq:14}), we can apply Taylor formula and (\ref{eq34}) to deduce that
\be\label{eq37}
\vert f(0)-f(s_b)\vert /\vert s_b\vert \leq C_{R_0},
\en
which implies that
\be\label{eq39}
\vert (\widetilde{A}-1)/\widetilde{u}\vert =\left\vert (f(0)-f(s_b))/\left(\sqrt{2\vert x \vert}s_bf(s_b)\right)\right\vert
\leq C_{R_0}\vert x \vert^{-1/2}.
\en
Similarly to the derivation of (\ref{eq:25}), we can employ the analyticity of $f(s)$, Taylor formula, the fact that $\vert s_b\vert <\sigma^{(1)}_{\theta_{\hat{x}}}/4$ and formulas (\ref{eq31}), (\ref{eq:27}), (\ref{eq34})
and (\ref{eq37}) to get that
\begin{align*}
\vert f_1(s)\vert \le C_{R_0}{(1+\vert s\vert )}^{\frac 52}e^{\widetilde{C}_{R_0}{\vert s\vert }^2},\quad
s\in  L_{\sigma^{(1)}_{\theta_{\hat{x}}}}.
\end{align*}
Then similarly to the derivation of (\ref{eq34}), we can use the above formula, Cauchy
inequality and (\ref{eq31}) to obtain that: for $n=0,1$,
\begin{align*}
\left\vert \frac{d^n f_1}{d s^n}(s)\right\vert
\leq\frac{C_{R_0}n!(1+\vert s\vert )^{5/2}
e^{\widetilde{C}_{R_0}\vert s\vert^2}}{(\sigma^{(1)}_{\min}/2)^{n}},\quad
s\in  L_{{\sigma^{(1)}_{\min}/2}}.
\end{align*}
This, together with the facts that $s-s_b>0$ for $s\in\mathcal{I}_{s_b,s_b+\infty}$,
$\Rt(s_b)>0$ and $\Rt(s^2_b)=0$, implies that for $n=0,1$,
\ben
\left\vert \frac{d^n f_1}{d s^n}(s)e^{-\vert x \vert s^2}\right\vert \leq
C_{R_0}\left(1+\vert s-s_b\vert \right)^{5/2}e^{-\left(\vert x \vert-\widetilde{C}_{R_0}\right)\vert s-s_b\vert^2}
\quad\textrm{for}~s\in\mathcal{I}_{s_b,s_b+\infty}.
\enn
Thus, combining this and (\ref{eq38}) and using a change of variable $t=s-s_b$, we obtain
\begin{align}
\vert G^{(4)}_{\mathcal R,1}(x,y)\vert &\leq C_{R_0}\vert x \vert^{-1}\int^{+\infty}_0
\left(t^{1/2}+t^4\right)e^{-\left(\vert x \vert-\widetilde{C}_{R_0}\right)t^2}dt\nonumber\\ \label{eq40}
&\leq C_{R_0}\vert x \vert^{-1}\left((\vert x \vert-\widetilde{C}_{R_0})^{-3/4}+(\vert x \vert-\widetilde{C}_{R_0})^{-5/2}\right)
\leq C_{R_0}\vert x \vert^{-7/4}
\end{align}
for $\vert x \vert$ large enough. Further, since $\vert \widetilde{u}\vert <1$ in this case, it follows from the
analyticity of the function $D_\beta$ with $\beta\in\mathbb{R}$ that $\vert D_{-3/2}(\widetilde{u})\vert \leq C$
and $\vert D_{-5/2}(\widetilde{u})\vert \leq C$ for some constant $C>0$. This, together with
(\ref{eq54}), (\ref{eq39}), (\ref{eq40}) and the fact that
$\vert \cos((\theta_c-\theta)/2)\vert \geq\cos(\theta_c/2)$ for any $\theta\in(0,\theta_c)$, implies that
\ben
\vert G^{(4)}_{\mathcal R}(x,y)\vert \le {C_{R_0}}{\vert x \vert^{-3/4}}(1+\vert x \vert^{-1/2})
+{C_{R_0}}{\vert x \vert}^{-7/4}\le{C_{R_0}}{\vert x \vert^{-3/4}}.
\enn

From the discussions in the above two cases, we obtain that $G^{(4)}_{\mathcal R}(x,y)$
satisfies (\ref{neq:4.2.36}) uniformly for all $\theta_{\hat x}\in (0,\theta_c)$ and
$y\in B^+_{R_0}$. Therefore, the proof is complete.
\end{proof}

Based on Lemmas \ref{Le:n4.2.1}, \ref{NLe:3.2.4} and \ref{Le:n4.2.2}, we get the following uniform
far-field asymptotic estimates of the two-layered Green function and its derivatives.

\begin{theorem}\label{NLe:2}
Assume that $k_+>k_-$ and let $R_0>0$ be an arbitrary fixed number. Suppose that
$y=(y_1,y_2)=\vert y \vert(\cos \theta_{\hat y}, \sin \theta_{\hat y}) \in B^+_{R_0}\cup B^-_{R_0}$
with $\theta_{\hat{y}}\in(0,\pi)\cup(\pi,2\pi)$ and
$x=(x_1,x_2)=\vert x \vert\hat x=\vert x \vert(\cos\theta_{\hat x},\sin\theta_{\hat x})\in {\Rb}_{+}^2 $
with $\theta_{\hat{x}}\in(0,\pi)$, then we have the asymptotic behaviors
\begin{align} \label{NNe:3.2.1}
G(x,y)&=  \frac{e^{ik_{+}\vert x \vert}}{\sqrt{\vert x \vert}}G^{\infty}(\hat{x},y) + G_{Res}(x,y),\\
\label{eq42}
\nabla_y G(x,y) &=  \frac{e^{ik_{+}\vert x \vert}}{\sqrt{\vert x \vert}}H^{\infty}(\hat x,y ) +  H_{Res}(x,y),
\end{align}
where $G^{\infty}$ and $H^{\infty}$ are given by (\ref{eq:t7}) and (\ref{eq:4.2.18}), respectively,
and $G_{Res}$ and $H_{Res}$ satisfy
\begin{align}\label{neq:4.2.29}
&\left\vert G_{Res}(x,y)\right\vert,~
\vert H_{Res}(x,y)\vert  \le{C_{R_0}}{\vert x \vert^{-3/4}},\quad \vert x \vert\rightarrow+\infty,
\end{align}
uniformly for all $\theta_{\hat{x}}\in(0,\pi)$ and $y\in B^+_{R_0}\cup B^-_{R_0}$,
\begin{align}\label{neq:4.2.30}
&\left\vert G_{Res}(x,y)\right\vert,~\vert H_{Res}(x,y)\vert  \le{C_{R_0}}{{{\left\vert\theta_c-\theta_{\hat x}
\right\vert}^{-\frac 32}}\vert x \vert^{-\frac 32}},\quad\vert x \vert\rightarrow+\infty,
\end{align}
uniformly for all $\theta_{\hat{x}}\in (0,\theta_c)\cup(\theta_c,\pi/2)$ and
$y\in B^+_{R_0}\cup B^-_{R_0}$, and
\begin{align}\label{neq:4.2.31}
&\left\vert G_{Res}(x,y)\right\vert,~\vert H_{Res}(x,y)\vert \le{C_{R_0}}{{\left\vert\pi-\theta_c-\theta_{\hat x}\right\vert}^{-\frac32}
\vert x \vert^{-\frac 32}},\quad\vert x \vert\rightarrow+\infty,
\end{align}
uniformly for all $\theta_{\hat{x}}\in \left.\left[\pi/2,\pi-\theta_c\right.\right)\cup(\pi-\theta_c,\pi)$
and $y\in B^+_{R_0}\cup B^-_{R_0}$.
Here, the constant $C_{R_0}>0$ is independent of $x$ and $y$ but dependent of $R_0$.
\end{theorem}

\begin{proof}
First, we consider the asymptotic behavior of $G(x,y)$.
To this end, we distinguish between the following four cases.

\textbf{Case 1}: $\theta_{\hat{x}}\in(0,\pi/2)$ and $\theta_{\hat y}\in (0,\pi)$.
Note that ${\mathcal S}(\cos\theta_{\hat x},n)= \widetilde{\mathcal S}(\cos\theta_{\hat x},n)$
for $\theta_{\hat x}\in (0,\theta_c)$. Thus, it follows from
(\ref{NNe6}), (\ref{e:1.2}) and Lemmas \ref{Le:n4.2.1}, \ref{NLe:3.2.4} and \ref{Le:n4.2.2}
that $G(x,y)$ has the asymptotic behavior
(\ref{NNe:3.2.1}) with $G_{Res}(x,y)$ satisfying (\ref{neq:4.2.29}) uniformly for all
$\theta_{\hat{x}}\in(0,\pi/2)$ and $y\in B^+_{R_0}$ and satisfying (\ref{neq:4.2.30})
uniformly for all $\theta_{\hat{x}}\in (0,\theta_c)\cup(\theta_c,\pi/2)$ and $y\in B^+_{R_0}$.

\textbf{Case 2}: $\theta_{\hat{x}}=\pi/2$ and $\theta_{\hat y}\in (0,\pi)$.
By using similar arguments as in the proof of Theorem \ref{NLe:4}, we can obtain that
$G(x,y)$ has the asymptotic behavior
(\ref{NNe:3.2.1}) with $G_{Res}(x,y)$ satisfying $\vert G_{Res}(x,y)\vert  \le {C_{R_0}}{\vert x \vert^{-3/2 }}$
uniformly for $\theta_{\hat{x}}=\pi/2$ and all $y\in B^+_{R_0}$.
This directly implies that $G_{Res}(x,y)$ in (\ref{NNe:3.2.1}) satisfies (\ref{neq:4.2.29})
and (\ref{neq:4.2.31}) uniformly for $\theta_{\hat{x}}=\pi/2$ and all $y\in B^+_{R_0}$.

\textbf{Case 3}: $\theta_{\hat{x}}\in\left(\pi/2,\pi\right)$ and $\theta_{\hat y}\in (0,\pi)$.
Since $\mathcal S(\xi,a)=\mathcal S(-\xi,a)$ for any $\xi\in\mathbb{R}$ and $a>0$, it easily follows
from (\ref{e:1.2}) and (\ref{eq4}) that
\be\label{eq48}
G(v,w)=G(x,y)
\en
with $v:=(-x_1,x_2)$ and $w:=(-y_1,y_2)$. Note that $\theta_{\hat{v}}=\pi-\theta_{\hat{x}}$.
Thus, with the aid of the results in Case 1, we deduce that
$G(x,y)$ has the asymptotic behavior
(\ref{NNe:3.2.1}) with $G_{Res}(x,y)$ satisfying (\ref{neq:4.2.29}) uniformly for all $\theta_{\hat{x}}\in\left(\pi/2,\pi\right)$ and $y\in B^+_{R_0}$
and satisfying (\ref{neq:4.2.31})
uniformly for all $\theta_{\hat{x}}\in \left(\pi/2,\pi-\theta_c\right)\cup(\pi-\theta_c,\pi)$
and $y\in B^+_{R_0}$.

\textbf{Case 4}: $\theta_{\hat{x}}\in\left(0,\pi\right)$ and $\theta_{\hat y}\in (\pi,2\pi)$.
Analogous to the above three cases, we can also study the asymptotic behavior of
$G_{\mathcal T}(x,y)$ by using similar arguments as in the proofs of Theorem \ref{NLe:4}
and Lemmas \ref{Le:n4.2.1}, \ref{NLe:3.2.4} and \ref{Le:n4.2.2}. Consequently, for $y\in B^-_{R_0}$,
we can obtain from (\ref{e:1.2}) that  $G(x,y)$ has the asymptotic behavior
(\ref{NNe:3.2.1}) with $G_{Res}(x,y)$ satisfying all its properties presented in this theorem.

Secondly, we consider the asymptotic behavior of $\nabla_y G(x,y)$. By (\ref{eq4}) and (\ref{eq5}) we have
\begin{align*}
&\nabla_y  G_{\mathcal R}(x,y) =\notag\\
&\frac{-1}{4\pi}\int_{-\infty}^{+\infty}\frac{\mathcal S(\xi,{k_{+}})
-\mathcal S(\xi,{k_{-}})}{\mathcal S(\xi,{k_{+}})+\mathcal S(\xi,{k_{-}})}
\frac{e^{-\mathcal S(\xi,{k_{+}})\vert x_2+y_2\vert }}{\mathcal S(\xi,{k_{+}})}e^{i\xi(x_1-y_1)}
\begin{pmatrix}
 i\xi \\
 \mathcal S(\xi,{k_{+}})
\end{pmatrix}^T
d\xi,&& y_2>0,\\
&\nabla_y  G_{\mathcal T}(x,y) = \frac{1}{2\pi} \int_{-\infty}^{+\infty}
\frac{{e}^{\mathcal S(\xi,{k_{-}})y_2-\mathcal S(\xi,{k_{+}})x_2}}{\mathcal S(\xi,{k_{+}})
+\mathcal S(\xi,{k_{-}})}e^{i\xi(x_1-y_1)}
\begin{pmatrix}
-i\xi \\
\mathcal S(\xi,{k_{-}})
\end{pmatrix}^T
d\xi, &&y_2<0.
\end{align*}
From \cite{DK13}, it is easy to see that the Hankel function $H^{(1)}_0$  satisfies
\begin{align*}
\nabla_y\left(\frac{i}{4}H^{(1)}_0(k_{+}\vert x-y\vert )\right) = \frac{e^{ik_{+}\vert x \vert}}{\sqrt{\vert x \vert}}e^{-i\frac{\pi}4}{\sqrt{\frac{k_+}{8\pi}}}
\left(\hat x e^{-ik_{+}\hat x\cdot y} +O\left({\vert x \vert^{-1}}\right)\right)
\end{align*}
as $\vert x \vert\rightarrow+\infty$ uniformly for all $\theta_{\hat{x}}\in(0,\pi)$ and $y\in\mathbb{R}^2$ with $\vert y \vert< R_0$.
Note that $H^\infty(\hat{x},y)=\nabla_y G^\infty(\hat{x},y)$.
Thus, using (\ref{e:1.2}) and similar arguments as in the derivations of the asymptotic behavior
of $G(x,y)$, we can obtain that $\nabla_y G(x,y)$
has the asymptotic behavior (\ref{eq42}) with $H_{Res}(x,y)$ satisfying all its properties presented
in this theorem.
\end{proof}

\begin{remark}\label{re3}{\rm
Let $y$ be an arbitrary fixed point in $\Rb^2_+\cup\Rb^2_-$ and $G_{Res}(x,y)$ be given as in
Theorem \ref{NLe:2}. It has been proved in \cite[Section 2.3.4]{Car17} that
\be\label{eq47}
G_{Res}(x,y)=O(\vert x \vert^{-3/2}),\quad \vert x \vert\rightarrow+\infty,
\en
for all $\theta_{\hat{x}}\in(0,\pi)\ba\{\theta_c,\pi-\theta_c\}$.
The above result is also a direct consequence of Theorem \ref{NLe:2}.
Moreover, we can further deduce that
\be\label{eq49}
\lim_{\vert x \vert\rightarrow+\infty}\vert G_{Res}(x,y)\vert \vert x \vert^{3/4}\neq0\quad
\textrm{for}~\theta_{\hat{x}}\in\{\theta_c,\pi-\theta_c\},
\en
which directly implies that (\ref{eq47}) does not hold for $\theta_{\hat{x}}\in\{\theta_c,\pi-\theta_c\}$.
In fact, for the case $y\in\Rb^2_+$, (\ref{eq49}) can be easily proved by Lemmas \ref{Le:n4.2.1}
and \ref{NLe:3.2.4} and formulas (\ref{NNe6}), (\ref{e:1.2}), (\ref{eq45}), (\ref{eq46}) and (\ref{eq48})
(note that $g(0)$ in (\ref{eq45}) is equal to non-zero value $h_c(\theta_c)$ with $h_c$ defined in (\ref{eq14})).
For the case $y\in\Rb^2_-$, (\ref{eq49}) can be proved by using similar arguments for $G_{\mathcal T}(x,y)$.
Furthermore, it can be seen from (\ref{eq49}) that the uniform asymptotic estimate of $G_{Res}(x,y)$
in (\ref{neq:4.2.29}) is essentially sharp.
Due to these results, we then call $\theta_c$ and $\pi-\theta_c$ the critical angles for the case $k_+>k_-$.
}
\end{remark}

\section{Uniform far-field asymptotic analysis of $G(x,y)$ with $x\in\Rb^2_-$}\label{sec:3}

In this section, we study the uniform far-field asymptotics of $G(x,y)$ with $x\in\Rb^2_-$.
Let $\theta_c$ be defined as in (\ref{eq43}).
{\color{hw-a}
Note that in the case $k_+<k_-$, there are two critical angles $\pi+\theta_c$ and $2\pi-\theta_c$ for $G(x,y)$.
However, the difficulties in the investigation of the uniform far-field asymptotics of $G(x,y)$ for the angles $\theta_{\hat{x}}$ in the vicinity of these two critical angles can be resolved by using a property
of $G(x,y)$ presented below.
}

Let the functions $\widetilde {\mathcal R}(\theta)$ and $\widetilde {\mathcal T}(\theta)$ be defined by
\begin{align*}
\widetilde {\mathcal R}(\theta):=\frac{{i\sin\theta-\mathcal S(\cos\theta,1/n)}}
{{i\sin\theta+\mathcal S{(\cos\theta,1/n)}}},\quad
\widetilde {\mathcal T}(\theta):={\widetilde{\RR}(\theta)}+1 \quad \textrm{for}~
\theta\in\mathbb{R}.
\end{align*}
By \cite[formula (2.27)]{Car17}, $G(x,y)$ with $x=(x_1,x_2)\in \Rb^{2}_-$ and
$y=(y_1,y_2)\in\Rb^{2}_+\cup\Rb^{2}_-$ has the following explicit formula
{\color{hw-a}(see also \cite[Appendix A]{Lpj})}
\begin{align}\label{eq:4.1}
G(x,y) =\left\{\begin{aligned}
& G_{\widetilde{\mathcal{T}}}(x,y),
&& x\in\Rb^2_-, && y\in\Rb^2_+, \\
&\frac{i}{4}H^{(1)}_0(k_{-}\vert x- y\vert ) + G_{\widetilde{\mathcal{R}}}(x,y),
&& x\in\Rb^2_-, && y\in\Rb^2_-,
\end{aligned}\right.
\end{align}
where $G_{\widetilde{\mathcal{T}}}$ and $G_{\widetilde{\mathcal{R}}}$ are given by
the Sommerfeld integrals as follows
\begin{align}
\label{eq61}
& G_{\widetilde{\mathcal{T}}}(x,y):= \frac{1}{2\pi}\int_{-\infty}^{+\infty}
\frac{e^{-\mathcal S(\xi,{k_{+}})y_2+\mathcal S(\xi,{k_{-}})x_2}}{\mathcal S(\xi,{k_{+}})
+\mathcal S(\xi,{k_{-}})}e^{i\xi(x_1-y_1)}d\xi ,\\
\label{eq62}
& G_{\widetilde{\mathcal{R}}}(x,y):=  \frac{1}{4\pi}\int_{-\infty}^{+\infty}
\frac{\mathcal S(\xi,{k_{-}})-\mathcal S(\xi,{k_{+}})}{\mathcal S(\xi,{k_{+}})
+\mathcal S(\xi,{k_{-}})}\frac{e^{-\mathcal S(\xi,{k_{-}})\vert x_2+y_2\vert }}
{\mathcal S(\xi,{k_{-}})}e^{i\xi(x_1-y_1)}d\xi.
\end{align}
Let $G^*(x,y):= G(x',y')$ for $x\in \Rb^2$ and $y \in \Rb^2_+\cup \Rb^2_-$ with $x'$ and $y'$
defined in Section \ref{sec:2}.
By formulas (\ref{e:1.2}) and (\ref{eq:4.1}), it is easy to verify that $G^*(x,y)$ is the
two-layered Green function satisfying (\ref{eq:0.3})--(\ref{eq:0.4}) with the wave numbers
$k_+$ and $k_-$ replaced by $k_-$ and $k_+$, respectively.
Note further that $x'\in\mathbb{R}^2_+$  and $\theta_{\hat{x'}}=2\pi-\theta_{\hat{x}}$
for $x\in\mathbb{R}^2_-$.
Therefore, applying Theorems \ref{NLe:4} and \ref{NLe:2} to $G^*(x,y)$, we can directly obtain
the following two theorems for the uniform far-field asymptotics of $G(x,y)$ with $x\in \Rb^2_-$.

\begin{theorem}\label{le:3.2}
Assume that $k_+<k_-$ and let $R_0>0$ be an arbitrary fixed number. Suppose that
$y=(y_1,y_2)\in B^+_{R_0}\cup B^-_{R_0}$ and
$x= (x_1,x_2)= \vert x \vert\hat x=\vert x \vert(\cos \theta_{\hat x},\sin\theta_{\hat x})\in {\Rb}_{-}^2$
with $\theta_{\hat{x}}\in(\pi,2\pi)$, then we have the asymptotic behaviors
\begin{align*}
G(x,y)&=  \frac{e^{ik_{-}\vert x \vert}}{\sqrt{\vert x \vert}}G^{\infty}(\hat{x},y) + G_{Res}(x,y),\\
\nabla_y G(x,y) &=  \frac{e^{ik_{-}\vert x \vert}}{\sqrt{\vert x \vert}}H^{\infty}(\hat x,y ) +  H_{Res}(x,y),\notag
\end{align*}
where $G^{\infty}$, $H^{\infty}$ are defined by
\begin{align}\label{eq:3.1}
G^{\infty}(\hat x,y):= \frac{e^{i\frac{\pi}4}}{\sqrt{8\pi k_{-}}}
\left\{\begin{aligned}
&\widetilde {\mathcal T}(\theta_{\hat x})e^{-ik_{-}(y_1\cos\theta_{\hat x}
-iy_2\mathcal S{(\cos\theta_{\hat x},1/n)})},
&&\hat{x}\in\mathbb{S}^1_-,&& y\in\Rb^2_+,\\
&e^{-i k_{-}{\hat x} \cdot y} + \widetilde {\mathcal R}(\theta_{\hat x}) e^{-ik_{-}{\hat x}\cdot y{'}},
&&\hat{x}\in\mathbb{S}^1_-,&& y\in\Rb^2_-,\\
\end{aligned}\right.
\end{align}
\begin{align}\label{eq:3.2}
H^{\infty}(\hat x,y ) :=
{e^{-i\frac{\pi}4}}\sqrt{\frac{k_{-}}{8\pi}}
\left\{
\begin{aligned}
&\widetilde{\mathcal T}(\theta_{\hat x})e^{-ik_{-}(y_1\cos\theta_{\hat x}
-iy_2\mathcal S{(\cos\theta_{\hat x},1/n)})}
\begin{pmatrix}
\cos \theta_{\hat x}\\
 -i \mathcal S{(\cos\theta_{\hat x},1/n)}
\end{pmatrix}^T,\\
&\qquad\qquad\qquad\qquad\qquad\qquad\qquad\qquad
\hat{x}\in\mathbb{S}^1_-,\quad y\in\Rb^2_+, \\
&e^{-i k_{-}\hat x \cdot y} \begin{pmatrix}
 \cos \theta_{\hat x} \\
 \sin \theta_{\hat x}
\end{pmatrix}^T
+  \widetilde{\mathcal R}(\theta_{\hat x})e^{-ik_{-}{\hat x}\cdot y{'}}
\begin{pmatrix}
\cos \theta_{\hat x} \\
 -\sin \theta_{\hat x}
\end{pmatrix}^T,\\
&\qquad\qquad\qquad\qquad\qquad\qquad\qquad\qquad
\hat{x}\in\mathbb{S}^1_-,\quad y\in\Rb^2_-,
\end{aligned}
\right.
\end{align}
and $G_{Res}$ and $H_{Res}$ satisfy the estimates
\begin{align}\label{eq52}
&\left\vert G_{Res}(x,y)\right\vert,~
\vert H_{Res}(x,y)\vert  \le{C_{R_0}}{\vert x \vert^{-3/4}},\quad \vert x \vert\rightarrow+\infty,
\end{align}
uniformly for all $\theta_{\hat{x}}\in(\pi,2\pi)$ and $y\in B^+_{R_0}\cup B^-_{R_0}$,
\begin{align*}
&\left\vert G_{Res}(x,y)\right\vert,~
\vert H_{Res}(x,y)\vert  \le{C_{R_0}}{{{\left\vert
\pi+\theta_c-\theta_{\hat x}
\right\vert}^{-\frac 32}}\vert x \vert^{-\frac 32}},\quad\vert x \vert\rightarrow+\infty,
\end{align*}
uniformly for all $\theta_{\hat{x}}\in(\pi,\pi+\theta_c)\cup\left.\left(\pi+\theta_c,3\pi/2\right.\right]$
and $y\in B^+_{R_0}\cup B^-_{R_0}$, and
\begin{align*}
&\left\vert G_{Res}(x,y)\right\vert,~
\vert H_{Res}(x,y)\vert  \le{C_{R_0}}{{{\left\vert
2\pi-\theta_c-\theta_{\hat x}
\right\vert}^{-\frac 32}}\vert x \vert^{-\frac 32}},\quad\vert x \vert\rightarrow+\infty,
\end{align*}
uniformly for all $\theta_{\hat{x}}\in  (3\pi/2,2\pi-\theta_c)\cup(2\pi-\theta_c,2\pi)$ and $y\in B^+_{R_0}\cup B^-_{R_0}$.
Here,
the constant $C_{R_0}>0$ is independent of $x$ and $y$ but dependent of $R_0$.
\end{theorem}

\begin{theorem}\label{le:3.1}
Assume that $k_+>k_-$ and let $R_0>0$ be an arbitrary fixed number.
Suppose that $y=(y_1,y_2)\in B^+_{R_0}\cup B^-_{R_0}$ and
$x=\vert x \vert\hat x=\vert x \vert(\cos \theta_{\hat x},\sin\theta_{\hat x} )\in\mathbb{R}^2_-$
with $\theta_{\hat{x}}\in(\pi,2\pi)$, then we have the asymptotic behaviors
\begin{align*}
 G(x,y) &= \frac{e^{ik_{-}\vert x \vert}}{\sqrt{\vert x \vert}}G^{\infty}(\hat x, y) + G_{Res}(x,y),\\
 \nabla_y G(x,y) &=  \frac{e^{ik_{-}\vert x \vert}}{\sqrt{\vert x \vert}}H^{\infty}(\hat x,y ) +  H_{Res}(x,y),
\end{align*}
where $G^{\infty}$ and $H^{\infty}$ are given by (\ref{eq:3.1}) and (\ref{eq:3.2}), respectively, and
$G_{Res}$ and $H_{Res}$ satisfy
\begin{align*}
\vert G_{Res}(x,y)\vert,~ \vert H_{Res}(x,y)\vert \le {C_{R_0}}{\vert x \vert^{-3/2}},\quad \vert x \vert\rightarrow+\infty,
\end{align*}
uniformly for all $\theta_{\hat{x}}\in(\pi,2\pi)$ and $y\in B^+_{R_0}\cup B^-_{R_0}$.
Here, the constant $C_{R_0}>0$ is independent of $x$ and $y$ but dependent of $R_0$.
\end{theorem}

\begin{remark}\label{re4} {\rm
In this remark, we restrict our attention to the case $k_+<k_-$.
Let $y$ be an arbitrary fixed point in $\Rb^2_+\cup\Rb^2_-$ and $G_{Res}(x,y)$ be given as in
Theorem \ref{le:3.2}. It easily follows from Theorem \ref{le:3.2} that $G_{Res}(x,y)$
satisfies (\ref{eq47}) for all $\theta_{\hat{x}}\in (\pi,2\pi)\ba\{\pi+\theta_c,2\pi-\theta_c\}$.
Moreover, combining Remark \ref{re3} and the arguments above Theorem \ref{le:3.2}, we can obtain that
\be\label{eq51}
\lim_{\vert x \vert\rightarrow+\infty}\vert G_{Res}(x,y)\vert\vert x\vert^{3/4}\neq0\quad
\textrm{for}~\theta_{\hat{x}}\in\{\pi+\theta_c,2\pi-\theta_c\},
\en
which directly implies that $G_{Res}(x,y)$ does not satisfy (\ref{eq47})
for $\theta_{\hat{x}}\in\{\pi+\theta_c,2\pi-\theta_c\}$.
Further, it is easily seen from (\ref{eq51}) that the uniform asymptotic estimate of
$G_{Res}(x,y)$ in (\ref{eq52}) is essentially sharp.
Therefore, we call $\pi+\theta_c$ and $2\pi-\theta_c$ the critical angles for the case $k_+<k_-$.
}
\end{remark}

\begin{remark}\label{re5} {\rm
Let $G_{Res}(x,y)$ and $H_{Res}(x,y)$ be given as in Theorems \ref{NLe:4} and \ref{le:3.2}
for the case $k_+<k_-$ and be given as in Theorems \ref{NLe:2} and \ref{le:3.1} for the case $k_+>k_-$.
Then it easily follows from the results in Section \ref{sec:2} and this section that,
for the cases $k_+<k_-$ and $k_+>k_-$, $G_{Res}$ and $H_{Res}$ satisfy the estimates
\begin{align*}
&\left\vert G_{Res}(x,y)\right\vert,~
\vert H_{Res}(x,y)\vert \le{C_{R_0}}{\vert x \vert^{-3/4}},\quad \vert x \vert\rightarrow+\infty,
\end{align*}
uniformly for all $\theta_{\hat{x}}\in(0,\pi)\cup(\pi,2\pi)$ and $y\in B^+_{R_0}\cup B^-_{R_0}$
with arbitrarily fixed $R_0>0$.
}
\end{remark}

\begin{remark}\label{re10}{\rm{\color{hw}
For any source point $y$ lying on the interface $\Gamma_0$,
due to the well-posedness of the scattering problem in a two-layered medium (see \cite{BHY,YLZ}), we can define the two-layered Green function $G(x,y)$ such that
$G(x,y)$ is the unique solution satisfying $G(\cdot,y)-\Phi_0(\cdot,y)\in H^1_{loc}(\mathbb{R}^2)$,
$\Delta_x G(x,y) +\ka^2 G(x,y)=-\delta(x,y)$ in $\mathbb{R}^2$ (in the distributional sense) and the Sommerfeld radiation condition (\ref{eq:0.4}), where
$\Phi_0(\cdot,y)$ denotes the fundamental solution of the Laplace equation
$\Delta w = 0$ in $\Rb^2$ and where
$\ka$ is the wave number defined by
$\ka:=k_+$ in $\mathbb{R}^2_+$ and $\ka:=k_-$ in $\mathbb{R}^2_-$.
Here, $H^{1}_{loc}(\mathbb{R}^2)$
denotes the space of all functions $\phi:\mathbb{R}^2\rightarrow \mathbb{C}$ such that $\phi\in H^{1}(B)$ for all open balls $B\subset\mathbb{R}^2$.
From the expression of the Hankel function $H^{(1)}_0(\cdot)$ given in \cite[Section 3.5]{DK13}
as well as the expression of
$G(x,y)$ given in \eqref{e:1.2} and \eqref{eq:4.1}, it can be seen that for any $y\in\mathbb{R}^2_+\cup\mathbb{R}^2_-$,
$G(x,y)$ also satisfies $G(\cdot,y)-\Phi_0(\cdot,y)\in H^1_{loc}(\mathbb{R}^2)$.
Then by using the well-posedness of the scattering problem in a two-layered medium again, we can apply the elliptic interior $H^2$-regularity (see, e.g., \cite[Section 6.3]{E10}) and
the Sobolev inequality given in \cite[formula (32) in Section 5.6.3]{E10} to deduce that
$G(\cdot,y)\in C(\mathbb{R}^2\ba\{y\})$ for any $y\in\mathbb{R}^2$
and that $G(x,\cdot)\in C(\mathbb{R}^2\ba\{x\})$ for any $x\in\mathbb{R}^2$.
Thus, employing the local regularity estimate in \cite[Theorem 2.7]{CZ98}, we can obtain that $G(\cdot,y)\in C^1(\mathbb{R}^2\ba\{y\})$ for any $y\in\mathbb{R}^2$.
Furthermore, it is easy to see from \eqref{e:1.2} and \eqref{eq:4.1} that  $G(x,y)$ satisfies the reciprocity relation
$G(x,y)=G(y,x)$ for any $x,y\in\mathbb{R}^2_+\cup\mathbb{R}^2_-$ with $x\neq y$
(see also \cite[(2.28)]{Car17}).
Hence, by the above discussions, we have
$G(x,\cdot)\in C^1(\mathbb{R}^2\ba\{x\})$ for any $x\in\mathbb{R}^2$.
On the other hand, it easily follows from a direct calculation that for any $\hat{x}\in\mathbb{S}^1_+\cup\mathbb{S}^1_-$,
$G^{\infty}(\hat x,\cdot)$ (see \eqref{eq:t7} and \eqref{eq:3.1}) and $H^\infty(\hat{x},\cdot)$
(see \eqref{eq:4.2.18} and \eqref{eq:3.2}) can be extended as continuous functions
in $\mathbb{R}^2$, which we denote by $G^{\infty}(\hat x,\cdot)$ and $H^\infty(\hat{x},\cdot)$, respectively, again.
Therefore, it can be seen that Theorems \ref{NLe:4}, \ref{NLe:2}, \ref{le:3.2} and \ref{le:3.1} still hold with
$y\in B^+_{R_0}\cup B^-_{R_0}$ replaced by $y\in \{y\in\Gamma_0:|y|<R_0\}$.}}
\end{remark}

\begin{remark}{\rm{\color{hw-a}
By Lebesgue's theorem, it follows that for any $x\in\mathbb{R}^2_+$,
$G_{\mathcal R}(x,\cdot)$ given in \eqref{eq4} (resp. $G_{\mathcal T}(x,\cdot)$ given in \eqref{eq5})
can be extended as a function in $C^\infty(\ov{\mathbb{R}^2_+})$ (resp.
$C^\infty(\ov{\mathbb{R}^2_-})$).
Similarly, for any $x\in\mathbb{R}^2_-$, $G_{\widetilde{\mathcal{T}}}(x,\cdot)$ given in \eqref{eq61}
(resp. $G_{\widetilde{\mathcal{R}}}(x,\cdot)$ given in \eqref{eq62})
can be extended as a function in $C^\infty(\ov{\mathbb{R}^2_+})$ (resp.
$C^\infty(\ov{\mathbb{R}^2_-})$).
By such extensions, we can employ the continuity of $G(x,y)$ presented in Remark \ref{re10}
as well as \eqref{e:1.2} and \eqref{eq:4.1} to obtain that for any $x\in\mathbb{R}^2_+$, $G(x,y)=\frac{i}{4}H^{(1)}_0(k_{+}\vert x- y\vert ) + G_{\mathcal R}(x,y)=G_{\mathcal T}(x,y)$ on $y\in\Gamma_0$
and that for any $x\in\mathbb{R}^2_-$, $G(x,y)=G_{\widetilde{\mathcal{T}}}(x,y)=\frac{i}{4}H^{(1)}_0(k_{-}\vert x- y\vert ) + G_{\widetilde{\mathcal{R}}}(x,y)$ on $y\in\Gamma_0$.}}
\end{remark}

\section{Uniform far-field asymptotics of the solution to the scattering problem in a two-layered medium}
\label{sec:4}

In this section, as an application of the results in Sections \ref{sec:2} and \ref{sec:3},
we study the uniform far-field asymptotics of the solution to the acoustic scattering
problem by buried obstacles in a two-layered medium with a locally rough interface.
To this end, we introduce some notations.
Let $\Gamma:=\{(x_1,x_2): x_2=h_\Gamma(x_1), x_1\in \Rb \}$
represent a locally rough surface, where $h_\Gamma$ is
a Lipschitz continuous function with compact support in $\Rb$.
Let $\Gamma_p:=\{(x_1,x_2): x_2=h_{\Gamma}(x_1), x_1\in {\rm Supp}(h_{\Gamma}) \}$ denote the
local perturbation of $\Gamma$.
Let $\Omega_{\pm}:=\{(x_1,x_2): x_2\gtrless h_\Gamma(x_1), x_1\in \Rb \}$ denote  the homogenous
media above and below $\Gamma$, respectively. We assume that the scattering obstacle $D$, described
by a bounded domain with $C^2$-boundary $\partial D$ and a connected complement, is fully embedded
in the lower medium $\Omega_-$.
Let $k_\pm={\omega}/{c_\pm}$ be the wave numbers in $\Omega_\pm$, respectively, with $\omega$ being
the wave frequency and $c_\pm$ being the wave speeds in the homogenous media $\Omega_{\pm}$, respectively.
Let $B_R:=\big\{x\in \Rb^2: \vert x\vert < R \big\}$ be a disk with radius $R>0$.
{\color{hw-a}
Let $H^{1}_{loc}(\mathbb{R}^2)$
be defined as in Remark \ref{re10}
and let $H^{1}_{loc}(\mathbb{R}^2\ba\ov{D})$
be the space of all functions $\phi:\mathbb{R}^2\ba\ov{D}\rightarrow \mathbb{C}$ such that $\phi\in H^{1}((\mathbb{R}^2\ba\ov{D})\cap B)$ for all open balls $B$ containing $\ov{D}$.}

Now we describe the considered scattering problem. Consider the time-harmonic ($e^{-i\omega t}$
time dependence) incident plane wave $u^{i}(x,d):=e^{ik_{+}x\cdot d}$ propagating in the direction
$d=(\cos \theta_d, \sin \theta_d)\in\Sp^1_-$ with $\theta_d\in(\pi,2\pi)$.
Then the acoustic scattering problem by buried obstacles in a two-layered medium
is to find the total field $u^{tot}(x,d)=u^{0}(x,d)+u^{s}(x,d)$, which is the sum of the
reference field $u^{0}(x,d)$ and the scattered field $u^{s}(x,d)$.
The reference wave $u^0(x,d)$ is generated by the incident field $u^i(x,d)$ and the two-layered
medium, and is given by (see, e.g., (2.13a) and (2.13b) in \cite{Car17})
\ben
u^0(x,d):=\begin{cases}
\ds u^i(x,d)+ u^r(x,d),\quad & x\in \Rb^2_{+},\\
\ds u^t(x,d), \quad & x\in  \Rb^2_{-},\\
\end{cases}
\enn
where the reflected wave $u^r(x,d)$ and transmitted wave $u^t(x,d)$ are given by
\ben
u^r(x,d):=\mathcal{R}(\pi+\theta_d)e^{ik_{+}x\cdot d^r},\quad
u^t(x,d):=\mathcal{T}(\pi+\theta_d)
e^{ik_{-}x\cdot d^t},
\enn
respectively.
Here,
$d^{r}:=(\cos\theta_d,-\sin\theta_d)$ denotes the reflection direction in $\mathbb{S}^1_+$,
$d^t:=n^{-1}(\cos\theta_d,-i\mathcal{S}(\cos\theta_d,n))$
with $n$ given as in Section \ref{sec:2}, and $\RR(\pi+\theta_d)$ and $\Tc(\pi+\theta_d)$ are
the reflection and transmission coefficients, respectively, with $\RR$ and $\Tc$ given by (\ref{eq:0}).
The definition of $\mathcal{S}(\cdot,\cdot)$ gives that
\begin{align*}
  d^t=\left\{
  \begin{aligned}
    &\left(n^{-1}\cos\theta_d,-\sqrt{1-(n^{-1}\cos\theta_d)^2}\right)&&
    \textrm{if}~n^{-1}\vert \cos\theta_d\vert \leq1,\\
    &\left(n^{-1}\cos\theta_d,-i\sqrt{(n^{-1}\cos\theta_d)^2-1}\right)&&
    \textrm{if}~n^{-1}\vert \cos\theta_d\vert >1.
  \end{aligned}
  \right.
\end{align*}
In particular, if $n^{-1}\vert \cos\theta_d\vert \leq1$, then $d^t=(\cos \theta^t_{d},\sin{\theta^t_{d}})$
is the transmission direction in $\mathbb{S}^1_-$ with $\theta^t_{d}\in[\pi,2\pi]$
satisfying $\cos\theta^t_{d}=n^{-1}\cos\theta_d$.
It is easily seen that {\color{hw-a}for any $d\in\mathbb{S}^1_-$, the reference wave $u^0(x,d)\in H^1_{loc}(\mathbb{R}^2)$} and $u^0(x,d)$ satisfies the Helmholtz equations by
the unperturbed two-layered medium together with the transmission condition on $\Gamma_0$, that is,
\begin{align*}
\qquad\qquad\ds\Delta u^{0} +{k}^2_\pm u^{0}=0&&& \text{in}\quad\Rb^2_\pm,\qquad\qquad \\
\ds[u^{0}]=0,\;\;\left[{\partial u^{0}}/{\partial\nu}\right]=0&& &\text{on}\quad\Gamma_0,
\end{align*}
where $\nu$ denotes the unit normal on $\Gamma_0$ pointing into $\Rb^2_+$ and
$[\cdot]$ denotes the jump across the interface $\Gamma_0$.

When $D$ is a penetrable obstacle, the total field $u^{tot}(x,d)$ and the scattered field $u^s(x,d)$
satisfy the following scattering problem:
\begin{align}
\qquad\qquad\ds\Delta u^{tot} +{k}^2_+u^{tot}&=0&& \text{in}\quad\Omega_+,\qquad\qquad \label{eq:1.2}\\
\qquad\qquad\ds\Delta u^{tot} +{k}^2_- n_D u^{tot}&=0&& \text{in} \quad \Omega_-,\qquad\qquad\\
\ds[u^{tot}]=0,\;\;\left[{\partial u^{tot}}/{\partial\nu}\right]&=0& &\text{on}\quad\Gamma,\\
\ds\lim_{\vert x\vert\rightarrow+\infty}\sqrt{\vert x\vert}\left(\frac{\partial u^s}{\partial \vert x\vert }-ik_\pm u^s\right)&=0 && \textrm{uniformly for all}~ \hat{x} \in \mathbb{S}^1_\pm,\label{eq:1.3}
\end{align}
where $n_D(x)\in L^{\infty}(\Omega_-)$ denotes the refractive index with $\Rt(n_D)>0$, $\I(n_D)\geq 0$
and $\textrm{Supp}(n_D-1)=\ov{D}$, $\nu$ denotes the unit normal on $\Gamma$ pointing into $\Omega_+$,
$[\cdot]$ denotes the jump across the interface $\Gamma$,  and (\ref{eq:1.3}) is the Sommerfeld
radiation condition.

When $D$ is an impenetrable obstacle, the total field $u^{tot}(x,d)$ and the scattered field
$u^s(x,d)$ satisfy the following scattering problem:
\begin{align} \label{e:1.1}
\qquad\qquad\ds\Delta u^{tot} +{k}^2_+ u^{tot}=0&&& \text{in}\quad\Omega_+,\qquad\qquad \\
\qquad\qquad\ds\Delta u^{tot} +{k}^2_- u^{tot}=0&&& \text{in}\quad\Omega_-\ba\ov{D},\qquad\qquad \\
\ds[u^{tot}]=0,\;\;\left[{\partial u^{tot}}/{\partial\nu}\right]=0&& &\text{on}\quad\Gamma,\\
\qquad\qquad \mathscr{B} (u^{tot})=0&&& \text{on}\quad\partial D,\qquad\qquad\\
\label{eq2}
{\ds\lim_{\vert x\vert\rightarrow+\infty}\sqrt{\vert x\vert}\left(\frac{\partial u^s}{\partial\vert x\vert }-ik_\pm u^s\right)=0}&
&& \textrm{uniformly for all}~ \hat{x} \in \mathbb{S}^1_\pm.
\end{align}
Here, $\mathscr{B}$ denotes one of the following three boundary conditions
\begin{align*}\left\{
\begin{aligned}
&\mathscr{B} (u^{tot}):= u^{tot}&&\textrm{on}\;\partial D,\quad &&\textrm{if}\;D\;
\textrm{is a sound-soft obstacle}, \\
&\mathscr{B} (u^{tot}):=\partial u^{tot}/\partial\nu && \textrm{on}\; \partial D, &&\textrm{if}\; D\;
\textrm{is a sound-hard obstacle},  \\
&\mathscr{B} (u^{tot}):=\partial u^{tot}/\partial\nu+i\lambda u^{tot}
&&\textrm{on}\; \partial D, &&\textrm{if}\; D \; \textrm{is an impedance obstacle},
\end{aligned}\right.
\end{align*}
where $\nu$ is the unit outward normal to $\partial D$, and the impedance function $\lambda$
is a real-valued, continuous and nonnegative function. See Figure \ref{sca} for the problem geometry.

\begin{figure}
\centering
\includegraphics[width=0.7\textwidth]{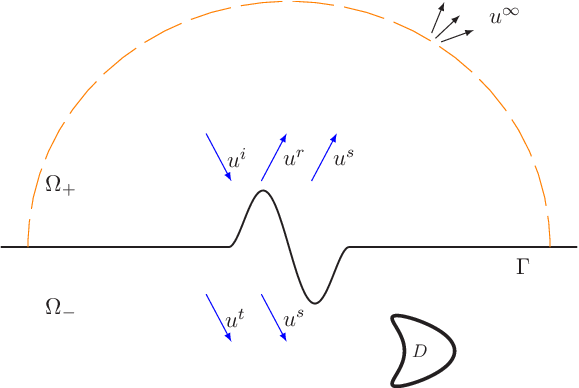}
\caption{The acoustic scattering problem by buried obstacles in a two-layered medium.}\label{sca}
\end{figure}

In the following theorem, we present some useful results for the well-posedness of the scattering
problem (\ref{eq:1.2})--(\ref{eq:1.3}) and the scattering problem (\ref{e:1.1})--(\ref{eq2}),
which are mainly based on \cite{BHY,YLZ}.
{\color{hw-a}We refer to \cite{CH98} for the well-posedness of the electromagnetic scattering problem in a two-layered medium.
Throughout the paper, we assume that the total field $u^{tot}(x,d)$ and the scattered
field $u^s(x,d)$ are given in the sense of Theorem \ref{th:4.1}.}

\begin{theorem}\label{th:4.1}
Let $R>0$ be an arbitrary fixed number
such that $\Gamma_p\cup \overline D \subset B_R$. Then given the reference field $u^0(x,d)$
generated by the incident field
$u^i(x,d)$ with $d\in\Sp^1_-$ and the two-layered medium, the following statements hold.
\begin{enumerate}
  \item \label{se1} {\color{hw-a}For any $d\in\mathbb{S}^1_-$,} there exists a unique solution $u^s(x,d)\in H^1_{\textrm{loc}}(\Rb^2)$ such that
      {\color{hw-a}$u^{tot}(x,d):=u^0(x,d)+u^s(x,d)\in H^1_{\textrm{loc}}(\Rb^2)$} and
      the total field $u^{tot}(x,d)$ and the scattered field $u^s(x,d)$ solve the scattering problem (\ref{eq:1.2})--(\ref{eq:1.3}). Furthermore,
      $\| u^s(\cdot,d)\|_{H^1(B_{R})}$ is uniformly bounded for all $d\in\Sp^1_-$.
  \item \label{se2} {\color{hw-a}For any $d\in\mathbb{S}^1_-$,}
  there exists a unique solution $u^s(x,d)\in H^1_{\textrm{loc}}(\Rb^2\ba\ov{D})$ such that
  {\color{hw-a}$u^{tot}(x,d):=u^0(x,d)+u^s(x,d)\in H^1_{\textrm{loc}}(\Rb^2\ba\ov{D})$} and
  the total field $u^{tot}(x,d)$ and the scattered field $u^s(x,d)$ solve the
  scattering problem (\ref{e:1.1})--(\ref{eq2}). Furthermore,
  $\|u^s(\cdot,d)\|_{H^1(B_{R}\ba\ov{D})}$ is uniformly bounded for all $d\in\Sp^1_-$.
\end{enumerate}
\end{theorem}

\begin{proof}
(\ref{se1}) This statement is a direct consequence of Theorem 2.5 in \cite{BHY} and the facts that
$u^0(\cdot,d)\in H^1_{loc}(\Rb^2)$ and $\| u^0(\cdot,d)\|_{H^1(B_{R})}$ is uniformly bounded
for all $d\in\Sp^1_-$.

(\ref{se2}) The uniqueness of the scattering problem (\ref{e:1.1})--(\ref{eq2}) has been proved in \cite{YLZ}.
We can argue similarly as in the proof of Theorem 2.5 in \cite{BHY} to show that this statement holds.
\end{proof}

Define the function spaces
\ben
C(\ov{\mathbb{S}^1_\pm}):=\{\varphi\in C(\mathbb{S}^1_\pm):\varphi~\textrm{is}~\textrm{uniformly}
~\textrm{continuous}~\textrm{on}~\mathbb{S}^1_\pm\}
\enn
{\color{hw-a}with the norms
$\|\varphi\|_{C(\ov{\mathbb{S}^1_\pm})}:=\sup_{x\in \mathbb{S}^1_\pm}|\varphi(x)|$, respectively,}
and define the function spaces
\ben
C^1(\ov{\Sp^1_\pm}):=\{\varphi\in C^1(\Sp^1_\pm):\varphi~\textrm{and}~
{\rm Grad}\,\varphi~\textrm{are}~\textrm{uniformly}
~\textrm{continuous}~\textrm{on}~\Sp^1_\pm\}
\enn
{\color{hw-a}with the norms
$\|\varphi\|_{C^1(\ov{\mathbb{S}^1_\pm})}:=\sup_{x\in\mathbb{S}^1_\pm}|\varphi(x)|
+\sup_{x\in\mathbb{S}^1_\pm}|\Grad \varphi(x)|$, respectively,}
where ${\rm Grad}$ denotes the surface gradient on $\Sp^1$.
With the help of Theorem \ref{th:4.1} and the uniform far-field asymptotic estimates of $G(x,y)$
presented in Theorems \ref{NLe:4} and \ref{le:3.2}, we have the following theorem for the
uniform far-field asymptotics of the scattered field $u^s(x,d)$ in the case $k_+<k_-$.

\begin{theorem} \label{NLe:3.2.3} Assume that $k_+<k_-$.
Let $R>0$ be large enough
such that $\Gamma_p\cup \overline D \subset B_R$ and
suppose that $x=\vert x\vert \hat x=\vert x\vert (\cos\theta_{\hat x}, \sin\theta_{\hat x})\in\Omega_+\cup\Omega_-$ with $\theta_{\hat{x}}\in(0,\pi)\cup(\pi,2\pi)$ and $\vert x\vert > R$.
For $d \in \Sp^1_{-}$,
let $u^s(x,d)$ be the scattered field for either the scattering problem (\ref{eq:1.2})--(\ref{eq:1.3})
or the scattering problem (\ref{e:1.1})--(\ref{eq2}). Then $u^s(x,d)$ has the asymptotic behavior
\begin{align} \label{Ne13}
u^{s}(x,d)=\left\{
\begin{aligned}
&\frac{e^{ik_{+}\vert x\vert }}{\sqrt{\vert x\vert }}u^{\infty}(\hat x,d)+u^s_{Res}(x,d) &&\textrm{for}~x\in\Omega_+\ba\ov{B_R},\\
&\frac{e^{ik_{-}\vert x\vert }}{\sqrt{\vert x\vert }}u^{\infty}(\hat x,d)+u^s_{Res}(x,d) &&\textrm{for}~x\in\Omega_-\ba\ov{B_R}
\end{aligned}
\right.
\end{align}
with
the far-field pattern  $u^{\infty}(\hat x, d)$ of the scattered field given by
\begin{align}\label{eq29}
u^\infty(\hat x,d)=\int_{\partial{B_R}}\left[{\frac{\partial G^\infty(\hat{x},y)}{\partial\nu(y)}}u^s(y,d)
-\frac{\partial u^s(y,d)}{\partial\nu(y)}G^{\infty}(\hat{x},y)\right]ds(y),\quad
\hat x \in \mathbb S^1_+\cup \mathbb S^1_-,
\end{align}
where $u^{\infty}(\hat x, d)$ satisfies
$u^{\infty}(\cdot, d)\in C^1(\ov{\Sp^1_{+}})$, $u^{\infty}(\cdot, d)\in C(\ov{\Sp^1_{-}})$
and ${\rm Grad}_{\hat x}\,u^{\infty}(\cdot, d)\in L^1(\Sp^1_{-})$ with
\begin{align*}
\| u^{\infty}(\cdot, d)\|_{C^1(\ov{\Sp^1_{+}})},~
\| u^{\infty}(\cdot, d)\|_{C(\ov{\Sp^1_{-}})},~
\| {\rm Grad}_{\hat x}\,u^{\infty}(\cdot, d)\|_{L^1(\Sp^1_{-})} \le C\quad\textrm{for~all}~d \in \Sp^1_{-},
\end{align*}
and $u^s_{Res}(x,d)$ satisfies
\ben
\vert u^s_{Res}(x,d)\vert \le {C}{\vert x\vert^{-3/2}},\quad \vert x\vert \rightarrow+\infty,
\enn
uniformly for all $\theta_{\hat{x}}\in(0,\pi)$ and $d \in \Sp^1_{-}$,
\ben
\left\vert u^s_{Res}(x,d)\right\vert\le{C}{\vert x\vert^{-3/4}},\quad \vert x\vert \rightarrow+\infty,
\enn
uniformly for all $\theta_{\hat{x}}\in(\pi,2\pi)$ and $d \in \Sp^1_{-}$,
\ben
\left\vert u^s_{Res}(x,d)\right\vert\le{C}{{{\left\vert \pi+\theta_c-\theta_{\hat x}
\right\vert}^{-\frac 32}}\vert x\vert^{-\frac 32}},\quad\vert x\vert \rightarrow+\infty,
\enn
uniformly for all $\theta_{\hat{x}}\in(\pi,\pi+\theta_c)\cup\left.\left(\pi+\theta_c, 3\pi/2\right.\right]$
and $d \in \Sp^1_{-}$, and
\ben
\left\vert u^s_{Res}(x,d)\right\vert\le{C}{{{\left\vert 2\pi-\theta_c-\theta_{\hat x}\right\vert}^{-\frac 32}}
\vert x\vert^{-\frac 32}},\quad\vert x\vert \rightarrow+\infty,
\enn
uniformly for all $\theta_{\hat{x}}\in (3\pi/2,2\pi-\theta_c)\cup(2\pi-\theta_c,2\pi)$ and $d\in\Sp^1_{-}$.
Here, $G^\infty(\hat{x},y)$ is defined by (\ref{eq:t7}) and (\ref{eq:3.1}),
and $C>0$ is a constant independent of $x$ and $d$.
\end{theorem}

\begin{proof}
Since $G(x,y)$ and $u^s(x,d)$ satisfy the Sommerfeld radiation condition (see formulas (\ref{eq:0.4}),
(\ref{eq:1.3}) and (\ref{eq2})), we can use a similar argument as in \cite[Theorem 2.5]{DK13} to obtain that for any $x\in(\Omega_+\cup\Omega_-)\ba\ov{B_{R}}$,
\begin{align} \label{eq:3.2.6}
u^s(x,d)= \int_{\partial{{B_{R}}}}\left[\frac{\partial G(x,y)}{\pa{\nu(y)}} u^s(y,d)
-\frac{\partial u^s(y,d)}{\pa\nu(y)} G(x,y) \right]ds(y),
\end{align}
where $\nu$ denotes the outward unit normal to the boundary $\partial B_{R}$ {\color{hw-a}(see also \cite[(3.11)]{BL16})}. With the aid
of Theorem \ref{th:4.1} and elliptic regularity estimates (see \cite{GT}), it follows that
\begin{align}\label{eq:4.2}
\| u^s(\cdot,d)\|_{L^2(\partial B_R)},~
\| \partial u^s(\cdot,d)/\pa\nu\|_{L^2(\partial B_R)}\le C_R\quad \textrm{for~all}~ d\in\Sp^1_-,
\end{align}
where $C_R>0$ is a constant independent of $d$ but dependent of $R$.
These, together with Theorems \ref{NLe:4} and \ref{le:3.2}, imply that $u^s(x,d)$ has the form (\ref{Ne13}) with $u^\infty(\hat{x},d)$ given by (\ref{eq29}) and with $u^s_{Res}(x,d)$ satisfying all its
properties presented in this theorem.
Further, since $n>1$, it is easily deduced that
\ben
\| \mathcal{S}(\cos(\cdot),n)\|_{C^1[0,\pi]},~
\| \mathcal{S}(\cos(\cdot),1/n)\|_{C[\pi,2\pi]},~
\int^{2\pi}_{\pi}\left\vert \frac{\pa\mathcal{S}(\cos(\theta),1/n)}{\pa\theta}\right\vert d\theta\leq C.
\enn
Thus,
from (\ref{eq:t7}), (\ref{eq:4.2.18}), (\ref{eq:3.1}) and (\ref{eq:3.2}), it is easy to verify
that $G^{\infty}(\cdot, y),H^{\infty}(\cdot, y)\in C^1(\ov{\Sp^1_{+}})$,
$G^{\infty}(\cdot, y),H^{\infty}(\cdot, y)\in C(\ov{\Sp^1_{-}})$
and ${\rm Grad}_{\hat{x}}\,G^{\infty}(\cdot, y),{\rm Grad}_{\hat{x}}
\left[H^{\infty}(\cdot, y)\cdot\nu(y)\right]\in L^1(\Sp^1_{-})$ with
\ben
&&\left\|G^{\infty}(\cdot, y)\right\|_{C^1(\ov{\Sp^1_{+}})},
~\left\|G^{\infty}(\cdot, y)\right\|_{C(\ov{\Sp^1_{-}})},
~\left\| {\rm Grad}_{\hat{x}}\,G^{\infty}(\cdot, y)\right\|_{L^1(\Sp^1_{-})}\leq C,\\
&&\left\|H^{\infty}(\cdot, y)\right\|_{C^1(\ov{\Sp^1_{+}})},
~\left\|H^{\infty}(\cdot, y)\right\|_{C(\ov{\Sp^1_{-}})},
~\left\|{\rm Grad}_{\hat{x}}\left[H^{\infty}(\cdot, y)\cdot\nu(y)\right]\right\|_{L^1(\Sp^1_{-})}\leq C
\enn
for all $y\in \pa{B_{R}}\ba\Gamma_0$. Using this, (\ref{eq29}) and (\ref{eq:4.2}) and noting that
$\nabla_y G^\infty(\hat{x},y)=H^\infty(\hat{x},y)$ for $\hat{x}\in\mathbb{S}^1_+\cup\mathbb{S}^1_-$ and $y\in\Rb^2_+\cup\Rb^2_-$, we obtain that $u^\infty(\hat{x},d)$ satisfies all its properties
presented in this theorem. The proof is thus complete.
\end{proof}

For the case  $k_+>k_-$, we note that (\ref{eq:3.2.6}) also holds
and thus we can employ Theorems \ref{NLe:2}, \ref{le:3.1} and \ref{th:4.1} to obtain the
following uniform far-field asymptotic estimates of $u^s(x,d)$.
The proof is similar to that of Theorem \ref{NLe:3.2.3} and is thus omitted.

\begin{theorem}\label{thm1}
Assume that $k_+ > k_-$.
Let $R>0$ be large enough
such that $\Gamma_p\cup \overline D \subset B_R$ and
suppose that $x=\vert x\vert \hat x=\vert x\vert (\cos\theta_{\hat x},\sin\theta_{\hat x})\in\Omega_+\cup\Omega_-$ with $\theta_{\hat{x}}\in(0,\pi)\cup(\pi,2\pi)$ and $\vert x\vert > R$.
For $d \in \Sp^1_{-}$,
let $u^s(x,d)$ be the scattered field given as in Theorem \ref{NLe:3.2.3}.
Then $u^s(x,d)$ has the asymptotic behavior (\ref{Ne13})
with the far-field pattern $u^\infty(\hat{x},d)$ of the scattered field given by (\ref{eq29}),
where $u^\infty(\hat{x},d)$ satisfies $u^{\infty}(\cdot, d)\in C(\ov{\Sp^1_{+}})$,
${\rm Grad}_{\hat x}\,u^{\infty}(\cdot, d)\in L^1(\Sp^1_{+})$ and
$u^{\infty}(\cdot, d)\in C^1(\ov{\Sp^1_{-}})$ with
\ben
\| u^{\infty}(\cdot,d)\|_{C(\ov{\Sp^1_{+}})},~
  \| {\rm Grad}_{\hat x}\,u^{\infty}(\cdot, d)\|_{L^1(\Sp^1_{+})},~
  \| u^{\infty}(\cdot, d)\|_{C^1(\ov{\Sp^1_{-}})}\le C\quad
  \mbox{for~all}~d\in\Sp^1_{-},
\enn
and $u^s_{Res}(x,d)$ satisfies
\ben
\left\vert u^s_{Res}(x,d)\right\vert\le{C}{\vert x\vert^{-3/4}},\quad \vert x\vert \rightarrow+\infty,
\enn
uniformly for all $\theta_{\hat{x}}\in(0,\pi)$ and $d \in \Sp^1_{-}$,
\ben
\left\vert u^s_{Res}(x,d)\right\vert\le{C}{{{\left\vert \theta_c-\theta_{\hat x}\right\vert}^{-\frac 32}}
\vert x\vert^{-\frac 32}},\quad\vert x\vert \rightarrow+\infty,
\enn
uniformly for all $\theta_{\hat{x}}\in (0,\theta_c)\cup(\theta_c,\pi/2)$ and $d \in \Sp^1_{-}$,
\ben
\left\vert u^s_{Res}(x,d)\right\vert\le{C}{{{\left\vert \pi-\theta_c-\theta_{\hat x}\right\vert}^{-\frac32}}
\vert x\vert^{-\frac32}},\quad\vert x\vert \rightarrow+\infty,
\enn
uniformly for all $\theta_{\hat{x}}\in\left.\left[\pi/2,\pi-\theta_c\right.\right)\cup(\pi-\theta_c,\pi)$
and $d \in \Sp^1_{-}$, and
\ben
\vert u^s_{Res}(x,d)\vert \le {C}{\vert x\vert^{-3/2}},\quad \vert x\vert \rightarrow+\infty,
\enn
uniformly for all $\theta_{\hat{x}}\in(\pi,2\pi)$ and $d \in \Sp^1_{-}$.
Here, $C>0$ is a constant independent of $x$ and $d$.
\end{theorem}

\begin{remark}{\rm
Let $u^s_{Res}(x,d)$ be given as in Theorem \ref{NLe:3.2.3} for the case $k_+<k_-$
and be given as in Theorem \ref{thm1} for the case $k_+>k_-$.
Then it easily follows from Theorems \ref{NLe:3.2.3} and \ref{thm1} that,
for both the case $k_+<k_-$ and the case $k_+>k_-$,
$u^s_{Res}(x,d)$ satisfies the estimate
\begin{align*}
&\left\vert u^s_{Res}(x,d)\right\vert \le C{\vert x\vert^{-3/4}},\quad \vert x\vert\rightarrow+\infty,
\end{align*}
uniformly for all $\theta_{\hat{x}}\in(0,\pi)\cup(\pi,2\pi)$ and $d\in\mathbb{S}^1_-$.}
\end{remark}

\section{Conclusion}\label{sec:5}

In this paper, we have established new results for the uniform far-field asymptotics of the two-dimensional
two-layered Green function $G(x,y)$ (together with its derivatives) in the frequency domain.
We note that, to the best of our knowledge, our results are the sharpest yet obtained.
The proofs of our results are based on the steepest descent method.
As an application of our results for $G(x,y)$ and its derivatives,
we have derived the uniform far-field asymptotic behaviors of the scattered field to the acoustic
scattering problem by buried obstacles in a two-layered medium with a locally rough interface.
Further, the results obtained in this paper provide a theoretical foundation for our recent
work \cite{LYZZ2}, where direct imaging methods have been proposed to numerically recover the
locally rough interface from phaseless total-field data or phased far-field data at a fixed frequency.
It is believed that the uniform asymptotic results obtained in this paper will also be useful on its own right.
Moreover, it is interesting to study the uniform far-field asymptotics of the Green
function with the background medium consisting of more than two layers. This is more challenging
and will be considered as a future work.

\section*{Acknowledgement}
This work was partially supported by Beijing Natural Science Foundation Z210001,
NNSF of China grants 11961141007, 61520106004 and 12271515,
Microsoft Research of Asia, and Youth Innovation Promotion Association CAS.

\appendix

\section{Proof of formula (\ref{neq:4.2.33}) in Lemma \ref{NLL1}} \label{sec:a}

In order to prove formula (\ref{neq:4.2.33}) in Lemma \ref{NLL1}, we need the following lemma
given in \cite{BH86}.

\begin{lemma}[see formula (9.4.28) on page 384 in \cite{BH86}]\label{lem1}
Assume $\beta>-1$. Define the function $W^{+}_{\beta}(z):=\int_{\mathcal C_{+}}t^{\beta}e^{-zt-\frac{t^2}{2}}dt$,
where the path $\mathcal C_{+}$ is depicted in Figure \ref{figfl.1}. Then we have
\ben
W^{+}_{\beta}(z)= \sqrt{2\pi}e^{\frac{i\pi}{2}\beta+\frac{z^2}{4}}D_{\beta}(iz).
\enn
\end{lemma}

\begin{figure}
\centering
\includegraphics[width=0.6\textwidth]{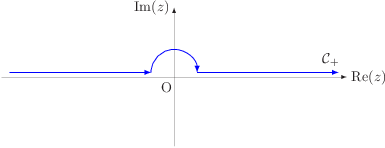}
\caption{The path $\mathcal C_{+}$.}\label{figfl.1}
\end{figure}

We now prove formula (\ref{neq:4.2.33}).

\begin{proof}[Proof of formula (\ref{neq:4.2.33})]
Assume $\I(b)<0$ and let $\alpha= -b$.
Let $\Ima(\rho)>0$ and then we can write $\rho = \vert \rho\vert e^{i\theta_\rho}$ with $\theta_\rho \in (0,\pi)$.
By a change of variable $t=s-b$, we have
\ben
F_2(\rho,b,\beta)= \int_{l_{\mathbb R + \alpha}}t^{\beta}e^{i\rho{(t-\alpha)}^2}dt,
\enn
where $l_{\mathbb R +\alpha}$ denotes the path $\{s+\alpha: s\in \mathbb R\}$ with
the orientation from $-\infty + \alpha $ to $+\infty + \alpha$.

Let the path $l_{\alpha}$ denote the curve $\{(\sqrt{2\rho})^{-1}e^{{i\pi}/4}(s+\alpha):s\in\Rb\}$
with the orientation from $(\sqrt{2\rho})^{-1}e^{{i\pi}/4}(-\infty + \alpha)$ to
$(\sqrt{2\rho})^{-1}e^{{i\pi}/4}(+\infty + \alpha)$
(see Figure \ref{fig1}).
Here, we note that $(\sqrt{2\rho})^{-1}e^{{i\pi}/4}=(2\vert\rho\vert)^{-1/2}e^{i(\pi/4-\theta_\rho/2)}$,
where $\pi/4-\theta_\rho/2\in(0,\pi/4)$ for the case $\Rt(\rho)>0$,
$\pi/4-\theta_\rho/2\in(-\pi/4,0)$ for the case $\Rt(\rho)<0$
and
$\pi/4-\theta_\rho/2=0$ for the case $\Rt(\rho)=0$.
For the case $\Rt(\rho)>0$, let the paths $l^{\pm }_{\alpha}$ denote the curves
$\left\{ z\in l_{\alpha}: \pm \Ima(z)> 0\right\}$, respectively, with the same orientations
as $l_{\alpha}$ (see Figure \ref{fig1}(a)). We now claim that
\begin{empheq}[left={F_2(\rho,b,\beta)= \empheqlbrace}]{align}
&\int_{l_{\alpha}}t^{\beta}e^{i\rho{(t-\alpha)}^2}dt,  & \textrm{if}~\Rt(\rho)\le 0, \label{eq19}\\
&\int_{ l^+_{\alpha}}t^{\beta}e^{i\rho{(t-\alpha)}^2}dt
+e^{i2\pi\beta}\int_{l^-_{\alpha}}t^{\beta}e^{i\rho{(t-\alpha)}^2}dt, & \textrm{if}~\Rt(\rho)>0, \label{eq20}
\end{empheq}
We only prove (\ref{eq20}) since the proof of (\ref{eq19}) is similar and easier.
In what follows, we assume that $\Rt(\rho)>0$. Let $p_{\alpha}\in\mathbb{R}$ denote
the intersection point of $l_{\alpha}$ and real axis, and let the path $\Rb_{p_\alpha}$ denote the
curve $\{z\in \Rb: z < p_{\alpha}\}$ with the orientation from $-\infty$ to $p_{\alpha}$.
Moreover, define
\ben
&&I_1:=\{z\in\Cb:\I(z)>\I(\alpha),~\I(\sqrt{2\rho}e^{-i\pi/4}z)<\I(\alpha)\},\\
&&I_2:=\{z\in\Cb:0<\I(z)<\I(\alpha),~\I(\sqrt{2\rho}e^{-i\pi/4}z)>\I(\alpha)\},
\enn
and let $I_3$ denote the domain enclosed by
$l^-_{\alpha}$ and $\Rb_{p_{\alpha}}$ (see Figure \ref{fig1}(a)). For $t \in \Rb_{p_{\alpha}}$, define $t^{\beta}_{\pm}:=\lim_{\epsilon\in\mathbb{R},\epsilon\rightarrow +0}(t\pm i\epsilon)^\beta$.
It is clear that $t^{\beta}_{\pm}$ exist and
\be\label{eq21}
t^{\beta}_{+} = e^{i2\pi\beta}t^{\beta}_{-}\quad \textrm{for}~ t\in \Rb_{p_{\alpha}},
\en
since $p_{\alpha}<0$ due to $\Rt(\rho)>0$ and $\I(\alpha)>0$. Further, for any $r\in\Rb$
it is easy to verify that $\vert t\vert^r\cdot\big\vert e^{i\rho{(t-\alpha)}^2}\big\vert\rightarrow 0$ as
$\vert t\vert\rightarrow+\infty$ uniformly for all $t\in \ov{I_1}\cup \ov{I_2}\cup \ov{I_3}$.
This, together with (\ref{eq21}) and Cauchy integral theorem, gives that
\begin{align*}
F_2(\rho,b,\beta)&= \int_{l^+_{\alpha}}t^{\beta}e^{i\rho{(t-\alpha)}^2}dt + \int_{\Rb_{p_\alpha}}t_{+}^{\beta}e^{i\rho{(t-\alpha)}^2}dt\\
&= \int_{l^+_{\alpha}}t^{\beta}e^{i\rho{(t-\alpha)}^2}dt + e^{i 2\pi\beta}\int_{\Rb_{p_\alpha}}t_{-}^{\beta}e^{i\rho{(t-\alpha)}^2}dt\\
&=\int_{ l^+_{\alpha}}t^{\beta}e^{i\rho{(t-\alpha)}^2}dt
 +e^{i2\pi\beta} \int_{l^-_{\alpha}}t^{\beta}e^{i\rho{(t-\alpha)}^2}dt.
\end{align*}
Thus we obtain (\ref{eq20}).

\begin{figure}
\centering
\includegraphics[width=0.9\textwidth]{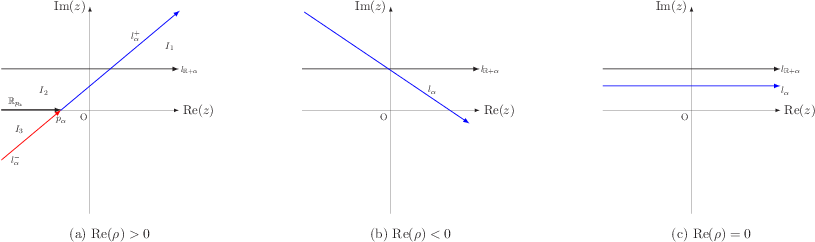}
\caption{The paths $l_{\Rb+\alpha}$, $l_{\alpha}$, $l^\pm_{\alpha}$, $\Rb_{p_{\alpha}}$,
 and the domains $I_1$, $I_2$, $I_3$.
 For the case $\Rt(\rho) > 0$, $l_{\alpha}$ is divided into $l^+_{\alpha}$ (blue line) and
 $l^-_{\alpha}$ (red line), and $p_{\alpha}$ denotes the intersection point of $l_{\alpha}$
 and real axis.}\label{fig1}
\end{figure}

On the other hand, for any $\eta\in\mathbb{C}_0$, we write $\eta=\vert\eta\vert e^{i\theta_\eta}$
with $\theta_\eta\in(-\pi,\pi)$. It is easily seen that ${\theta_\rho}/{2}- {{\pi}/4}\in(-\pi/4,\pi/4)$ and
\ben
\begin{cases}
\theta_\rho\in\left[\left.\frac{\pi}{2},\pi\right)\right.,~\theta_\eta\in(0,\pi),
~\frac{\pi}{4}+\theta_\eta-\frac{\theta_\rho}{2}\in(-\frac{\pi}{4},\pi),
&\mbox{if $\Rt(\rho)\le 0$, $\eta\in l_{\Rb+\alpha}$},\\
\theta_\rho\in(0,\frac\pi2),~\theta_\eta\in(0,\pi),~\frac{\pi}{4}+\theta_\eta-\frac{\theta_\rho}{2}\in(0,\pi),
&\mbox{if $\Rt(\rho)>0$, ${e^{{i\pi}/{4}}\eta}/{\sqrt{2\rho}}\in l^+_{\alpha}$},\\
\theta_\rho\in(0,\frac\pi2),~\theta_\eta\in(0,\pi),~\frac{\pi}{4}+\theta_\eta-\frac{\theta_\rho}{2}\in(\pi,2\pi),
&\mbox{if $\Rt(\rho)>0$, ${e^{{i\pi}/{4}}\eta}/{\sqrt{2\rho}}\in l^-_{\alpha}$}.
\end{cases}
\enn
Hence, it follows from the definition of $(\cdot)^\beta$ that
$\left(\sqrt{2\rho} e^{-i\pi/4}\right)^\beta =(2\rho)^{\beta/2}e^{-i\beta\pi/4}$ and
\ben
\begin{cases}
\left(\frac{1}{\sqrt{2\rho}e^{{-i\pi}/4}}\eta\right)^\beta
=\left(\frac{1}{\sqrt{2\rho}e^{{-i\pi}/4}}\right)^\beta \eta^\beta , &\mbox{if $\Rt(\rho)\le 0$,
 $\eta\in l_{\Rb +\alpha}$}, \\
\left(\frac{1}{\sqrt{2\rho}e^{{-i\pi}/4}}\eta\right)^\beta
= \left(\frac{1}{\sqrt{2\rho}e^{{-i\pi}/4}}\right)^\beta \eta^\beta,
& \mbox{if $\Rt(\rho)>0$, ${e^{{i\pi}/{4}}\eta}/{\sqrt{2\rho}}\in l^+_{\alpha}$},\\
\left(\frac{1}{\sqrt{2\rho}e^{{-i\pi}/4}}\eta\right)^\beta
= e^{-i2\pi \beta}\left(\frac{1}{\sqrt{2\rho}e^{{-i\pi}/4}}\right)^\beta \eta^\beta,
& \mbox{if $\Rt(\rho)>0$, ${e^{{i\pi}/{4}}\eta}/{\sqrt{2\rho}}\in l^-_{\alpha}$.}
\end{cases}
\enn
Thus, inserting $t = t(z):= (\sqrt{2\rho})^{-1}e^{\frac{i\pi}4}z$ into \eqref{eq19}
and (\ref{eq20}), we have
\ben
F_2(\rho,b,\beta)=e^{i\rho b^2}\int_{l_{\mathbb R+\alpha}}{\left(\frac{1}{2\rho}\right)}^{\frac{\beta+1}2}
e^{\frac{i\pi(\beta+1)}4}z^{\beta}e^{-\frac{1}{2}z^2-\gamma z}dz
\enn
for both the case $\Rt(\rho)\leq 0$ and the case $\Rt(\rho)>0$,
where $\gamma:= -\sqrt{2\rho} e^{\frac{-i\pi}4}\alpha$ and we use the fact that
$i\rho{(t-\alpha)}^2- i\rho{\alpha}^2= -\frac{1}{2}z^2-\gamma z$.
Note that $l_{\mathbb{R}+\alpha}$ lies in the upper-half complex plane
due to $\Ima(b)<0$,
and $\vert z^{\beta}e^{-\frac{1}{2}z^2-\gamma z}\vert\rightarrow 0$
as $\vert z\vert\rightarrow+\infty$ uniformly for all $z\in\{z\in\mathbb{C}:0\leq\I(z)\leq\I(\alpha)\}$.
Thus using Cauchy integral theorem, we deduce that
\ben
F_2(\rho,b,\beta)= e^{i\rho b^2}\int_{\mathcal C_{+}}{\left(\frac{1}{2\rho}\right)}^{\frac{\beta+1}2}
e^{\frac{i\pi(\beta+1)}4}z^{\beta}e^{-\frac{1}{2}z^2-\gamma z}dz.
\enn
This, together with Lemma \ref{lem1}, implies that (\ref{neq:4.2.33}) holds.
\end{proof}

\section{Proofs of Lemmas \ref{Nle:1} and \ref{Le:3.2.1}}\label{s3}

\subsection{Proof of Lemma \ref{Nle:1}}\label{sec:b1}

\begin{proof}
Let $\Rmnum{1}:=\left\{s\in L_{\mathsf w_0}:\vert\I(s)\vert<\Rt(s)\right\}$,
$\Rmnum{2}:=\left\{s\in L_{\mathsf w_0}: \I(s)>\vert\Rt(s)\vert\right\}$,
$\Rmnum{3}: =\left\{s\in L_{\mathsf w_0}: \vert\I(s)\vert< -\Rt(s)\right\}$,
$\Rmnum{4}: =\left\{s\in L_{\mathsf w_0}:\I(s)<-\vert\Rt(s)\vert\right\}$,
$Q_1:=\{s\in\Cb:\I(s)=\Rt(s)\}$, and $Q_2:=\{s\in\Cb:\I(s)= -\Rt(s)\}$.

For any $\mathsf w>0$, by a straightforward calculation, it is easy to verify that the function $h(s):=1-{s^2}/{(2k_{+}i)}=1-\Rt(s)\I(s)/k_+ + i(\vert\Rt(s)\vert^2-\vert\I(s)\vert^2)/(2k_+)$ maps the domain
$L_{\mathsf w}$  into the  domain  $\mathcal P_{\mathsf w}$ defined by
$$
\mathcal{P}_{\mathsf w}:=\left\{z\in \Cb:\frac{k_{+}^2(\Rt(z)-1)^2}{\mathsf w^2}< 2k_{+}\I(z)
+ \mathsf w^2 \right\}.
$$
It is easily seen that $\mathcal{P}_{\sqrt{k_+}}\cap \{s\in \Cb: \I(s)=0,\; \Rt(s)<0\}=\emptyset$.
Hence, it follows that $P(s)$ is an analytic function in $L_{\sqrt{k_+}}$ and satisfies
\be\label{eq18}
\Rt{(P(s))}>0\quad\textrm{for}~s\in L_{\sqrt{k_+}}.
\en

First, we consider the statement (\ref{se3}). Let $a_2\geq 0$.
It follows from (\ref{eq18}) that $\Rt(f(s))=a_0+a_1\Rt(P(s))+a_2(\Rt(s)+\I(s))/(2\sqrt{k_+})>0$
for $s \in\Rmnum{1}\cap \Rmnum{2}$.
Further, it follows from the definition of $h(s)$ that $\I (h(s))>0$ for
$s\in \Rmnum{3} $ and $\I(h(s))<0$ for $s\in \Rmnum{4}$, which implies that $\I{(P(s))}>0$
for $s\in\Rmnum{3}$ and $ \I{(P(s) )}<0$ for $s\in \Rmnum{4}$. This, together with the fact
that $\I(s)-\Rt(s)> 0$ for $s\in \Rmnum{3}$ and $\I(s)-\Rt(s)<0$ for $s\in \Rmnum{4}$, gives $\I{(f(s))}=a_1\I(P(s))+a_2(\I(s)-\Rt(s))/(2\sqrt{k_+})\ne0$ for $s \in\Rmnum{3}\cup \Rmnum{4}$.
Moreover, for $ s\in Q_1 \cap  L_{-\mathsf w_0, \sqrt{k_{+}}}$, we have
$$
f(s)=a_0+a_1P(s)+ a_2\frac{se^{-i\frac{\pi}4}}{\sqrt{2k_{+}}}
= a_0+a_1\sqrt{1- \frac{(\I(s))^2}{k_{+}}} + \frac{a_2}{\sqrt{k_{+}}}\I(s) > 0.
$$
For $s\in ( L_{-\mathsf w_0, \sqrt{k_{+}}} \cap  Q_2)\backslash \{0\}$, we obtain
from (\ref{eq18}) that $\Rt(f(s))= a_0 + a_1 \Rt(P(s))> 0$. Therefore, from the above arguments,
we obtain that statement (\ref{se3}) holds.

Secondly, we consider the statement (\ref{se4}). Let $a_2<0$. By similar arguments as in the proof of statement (\ref{se3}), we can deduce that $\I(f(s))>0$ for $s \in\Rmnum{1}$, $\I(f(s))<0$ for $s \in \Rmnum{2}$,
$\Rt(f(s))>0$ for $s \in\Rmnum{3}\cup \Rmnum{4}$, $f(s)>0$ for $s\in L_{-\sqrt{k_{+}},\mathsf w_0}\cap Q_1$
and $\Rt(f(s))> 0$ for $s\in ( L_{-\sqrt{k_{+}}, \mathsf w_0}\cap Q_2)\backslash \{0\}$.
Thus it easily follows that statement (\ref{se4}) holds.
\end{proof}

\subsection{Proof of Lemma \ref{Le:3.2.1}}\label{sec:b2}

\begin{proof}
From the analyticity of $A(z)$ and $B(z)$, there exists a disk $B_{\epsilon}:=\{z\in\mathbb{C}:\vert z\vert<\epsilon\}$
such that: for $z\in B_{\epsilon}$,
\ben
A(z)=\sum^{+\infty}_{j=0} A_j z^j,\quad B(z)=\sum^{+\infty}_{j=0}B_j z^j,\quad
A^2(z)=\sum^{+\infty}_{j=0} A^{(2)}_j z^j,\quad B^2(z) =  \sum^{+\infty}_{j=0} B^{(2)}_j z^j.
\enn
Since $A^2(z)=B^2(z)$ in $L_{\mathsf w}$, it is easily seen that
\be\label{eq17}
A^{(2)}_m=\sum^{m}_{j=0}A_jA_{m-j} =B^{(2)}_m= \sum^{m}_{j=0}B_j B_{m-j},~m\geq 0.
\en
Due to the fact that $A_0=A(0)= B_0=B(0)\neq 0$,  we can apply (\ref{eq17}) with $m=1$ to obtain that
$A_1=B_1$.
Then, by repeating the same
argument, we can apply (\ref{eq17}) with $m=2,3,...$ to obtain that $A_j=B_j$ for $j =0,1,2,\ldots$.
This implies that $A(z)=B(z)$ in $B_\epsilon$. Therefore, the statement of this lemma follows from
the analyticity of $A(z)$ and $B(z)$ in $L_{\mathsf w}$.
\end{proof}


\end{document}

%% file: ex_shared.tex
\usepackage{lipsum}
\usepackage{amsfonts}
\usepackage{graphicx}
\usepackage{epstopdf}
\usepackage{algorithmic}
\ifpdf
  \DeclareGraphicsExtensions{.eps,.pdf,.png,.jpg}
\else
  \DeclareGraphicsExtensions{.eps}
\fi

\newsiamremark{remark}{Remark}
\newsiamremark{hypothesis}{Hypothesis}
\crefname{hypothesis}{Hypothesis}{Hypotheses}
\newsiamthm{claim}{Claim}

\headers{Uniform asymptotics of two-layered Green function}{L. Li, J. Yang, B. Zhang, and H. Zhang}

\title{Uniform far-field asymptotics of the two-layered Green function in 2D and application
to wave scattering in a two-layered medium}

\author{Long Li \thanks{Academy of Mathematics and Systems Science, Chinese Academy of Sciences, Beijing {100190}, China
  (\email{longli@amss.ac.cn}).}
\and Jiansheng Yang \thanks{LMAM, School of Mathematical Sciences, Peking University, Beijing 100871, China
  (\email{jsyang@pku.edu.cn}).}
\and Bo Zhang \thanks{LSEC and Academy of Mathematics and Systems Science, Chinese Academy of
Sciences, Beijing 100190, China and School of Mathematical Sciences, University of Chinese
Academy of Sciences, Beijing 100049, China (\email{b.zhang@amt.ac.cn}).}
\and Haiwen Zhang \thanks{Corresponding author. Academy of Mathematics and Systems Science, Chinese Academy of Sciences, Beijing {100190}, China
  (\email{zhanghaiwen@amss.ac.cn})}.}

\usepackage{amsopn}
